\documentclass[reqno, 10pt]{amsart}
\usepackage[top=4cm, bottom=3cm, left=3.2cm, right=3.2cm]{geometry}
\usepackage{enumerate}
\usepackage{amssymb,amsmath}
\usepackage{graphicx}
\usepackage{subfigure}

\usepackage{amsthm}
\usepackage{amsmath}
\usepackage{amssymb}

\usepackage{extarrows}
\usepackage{epstopdf}
\usepackage{xcolor} 
\usepackage{bm}
\usepackage{geometry} 
\usepackage{graphicx} 

\usepackage{booktabs} 
\usepackage{array} 
\usepackage{paralist} 
\usepackage{verbatim} 
\usepackage{caption} 
\usepackage{cases}

\usepackage{multirow}
\usepackage{graphicx}
\usepackage{float}
\usepackage{subfigure}
\usepackage{algorithmicx,algorithm}
\usepackage{booktabs}
\usepackage{psfrag}
\usepackage[noend]{algpseudocode}

\newtheorem{thm}{Theorem}
\newtheorem{cor}{Corollary}
\newtheorem{lem}{Lemma}

\newtheorem{rem}{Remark}
\allowdisplaybreaks

\newcommand{\bx}{\boldsymbol{x}}
\newcommand{\by}{\boldsymbol{y}}

\numberwithin{equation}{section} \numberwithin{lem}{section}
\numberwithin{thm}{section} \numberwithin{prop}{section}
\numberwithin{cor}{section} \numberwithin{rem}{section}
\title[A unified structure preserving scheme for a multi-species model]{A unified structure preserving scheme for a multi-species model with a gradient flow structure and nonlocal interactions via singular kernels}

\author{
	Yong Zhang$^{\,1}$, Yu Zhao$^{\,2}$and Zhennan Zhou$^{\,3}$ }
\thanks{Corresponding author: Zhennan Zhou.}
\begin{document}
	\maketitle
	\begin{center}
		{\footnotesize
			1-Center for Applied Mathematics, Tianjin University, Tianjin,300072, P.R.China. \\
email: Zhang\_Yong@tju.edu.cn \\
\smallskip
2-School of Mathematical Sciences, Peking University, Beijing, 100871, P.R.China. \\  
email: y.zhao@pku.edu.cn \\	
\smallskip
3-Beijing International Center for Mathematical Research, Peking University, Beijing,\\ 100871, P.R.China. 
email: zhennan@bicmr.pku.edu.cn 
}
	\end{center}
	\maketitle
	\date{}
	\begin{abstract}
		In this paper, we consider a nonlinear and nonlocal parabolic model for multi-species ionic fluids and introduce a semi-implicit finite volume scheme, which is second order accurate in space, first order in time and satisfies the following properties: positivity preserving, mass conservation and energy dissipation. Besides, our scheme involves a fast algorithm on the convolution terms with singular but integrable kernels, which otherwise impedes the accuracy and efficiency of the whole scheme. Error estimates on the fast convolution algorithm are shown next. Numerous numerical tests  are provided to demonstrate the properties, such as unconditional stability, order of convergence, energy dissipation and the complexity of the fast convolution algorithm. Furthermore, extensive numerical experiments are carried out to explore the modeling effects in specific examples, such as, the steric repulsion, the concentration of ions at the boundary and the blowup phenomenon of the Keller-Segel equations.
		 
	\end{abstract}
	\section{Introduction}
	
	
	In this paper, we consider the following parabolic model for multi-species ionic fluids: for $m=1,\cdots,M$, we have
	\begin{align}
	\partial_t c_m(\boldsymbol{x}, t) &= \nabla \cdot \bigl(c_m \nabla \left( 1 + \log c_m + z_m \mathcal{K}*\rho + \mathcal{W}*\theta \right) \bigr), \quad \boldsymbol{x} \in \Omega, ~~ t>0, \label{sym1a}\\
	c_m(\boldsymbol{x},0) &= c^0_{m}(\boldsymbol{x}), \quad \boldsymbol{x} \in \Omega, \label{sym1c}
	\end{align}
	where $c_m = c_m(\boldsymbol{x}, t)$, $m=1,...,M$, is the concentration of the m-th ionic species and $\Omega \subset \mathbb{R}^d, d = 1, 2, 3,$ is a bounded domain or the whole space. 
	Here, $\rho(\boldsymbol{x})$ is the total charge density with $z_m \in \mathbb{Z}$ being the valence of the $m$-th ionic species and $\theta(\boldsymbol{x})$ is the total mass density, which are respectively given by
	\begin{equation}
	\rho(\boldsymbol{x}) = \sum_{m = 1}^{M} z_m c_m, \quad \quad \theta(\boldsymbol{x}) = \sum_{m = 1}^{M} c_m, \quad \quad \boldsymbol{x} \in \Omega. \label{sym1b}
	\end{equation}
	The kernel $\mathcal{K}(\boldsymbol{x})$ represents the effect of the electrostatic interaction while the kernel $\mathcal{W}(\boldsymbol{x})$ represents the generic nonlocal effects associated with the total mass density, such as the gravitational attraction \cite{Biler1993}, the chemotaxis-type aggregation \cite{Keller1970, Keller1971}, the steric repulsion arising from the finite size \cite{fawcett2004liquids,Lee2008,KS1991,barthel1998physical} etc. Thus, 
	\[
	(\mathcal{K}*\rho)(\boldsymbol{x})=\int_{\Omega} \mathcal{K}(\boldsymbol{x} - \boldsymbol{y}) \rho(\boldsymbol{y}) \,\mathrm{d}\boldsymbol{y}\quad \text{and} \quad (\mathcal{W}*\theta)(\boldsymbol{x})=\int_{\Omega} \mathcal{W}(\boldsymbol{x} - \boldsymbol{y}) \theta(\boldsymbol{y}) \,\mathrm{d}\boldsymbol{y}, \quad \boldsymbol{x} \in \Omega
	\]
	serve as potentials of the collective fields that lead to various collection motions of the fluids. Also, we observe that
	\[
	\nabla \cdot \bigl(c_m \nabla \left( 1 + \log c_m \right) \bigr)= \Delta c_m
	\]
	denotes the regular diffusion term as in the heat equation. To complete such a problem, certain boundary conditions are required on $\partial \Omega$, and without loss of generality, we assume no-flux boundary conditions for the fluids unless otherwise specified.
	In this drift-diffusion model, all physical parameters are set as 1 for simplicity in the representation except for $z_m$.
	
	
	Such a drift-diffusion problem is naturally endowed with a gradient flow structure, which is associated with the free energy given by
	\begin{align}
	\label{free2}
	\mathcal{F}
	=&\sum_{m=1}^M \int_{\Omega} c_m\log c_m\,\mathrm{d} \boldsymbol{x}
	+\frac{1}{2}\int_{\Omega}\rho(\boldsymbol{x}) (\mathcal{K}*\rho)(\boldsymbol{x})\,\mathrm{d} \boldsymbol{x}
	+\frac{1}{2}\int_{\Omega}\theta(\boldsymbol{x})(\mathcal{W}*\theta)(\boldsymbol{x})\,\mathrm{d}\boldsymbol{x}.
	\end{align}
	Thus, the dynamics of the fluids are driven by the chemical potential $\mu_m$ in the following way
	\begin{align}
	\partial_t c_m(\boldsymbol{x}, t) = \nabla \cdot (c_m \nabla {\mu}_m ), \quad m = 1, \cdots, M,
	\end{align}
	where $\mu_m$ is  the variational derivative of the free energy $\mathcal{F}$, i.e.,
	\begin{equation}
	\label{potential}
	{\mu}_m = \dfrac{\delta {\mathcal{F}}}{\delta c_m} = 1 + \log c_m + z_m \mathcal{K}*\rho + \mathcal{W}*\theta
	, \quad m = 1, \cdots, M.
	\end{equation}
	
	
	The system (\ref{sym1a})-(\ref{sym1c}) provides a broad framework for modeling nonlinear and nonlocal phenomenons of complex fluids, which naturally appears in various scientific problems. For example, when the kernel $\mathcal W$ vanishes, it reduces to the electrokinetic system for ion transport which is widely used in medicine and biology \cite{Alverts1994,Boron2008}. The transport and distribution of charged particles are crucial in the study of many physical and biological problems, such as ion particles in the electrokinetic fluids \cite{Jin2007}, and ion channels in cell membranes \cite{Bazant2004,Eisenberg1996}. In this scenario, the inclusion of the kernel $\mathcal W$ takes the nonideal properties of ionic solutions, such as the steric repulsion, into consideration, which is an important model to describe electrorheological fluids containing charged solid balls or some other complex fluids in biological applications properly \cite{Liu2019}. The Keller-Segel system is another notable example of (\ref{sym1a})-(\ref{sym1c}), which is of great significance for its vast applications in biology, social sciences and astrophysics \cite{Alt1987, Keller1970, Hofer1995}. In the case, the kernel function $\mathcal W$ takes the Newtonian potential, and such a system describes the competition between the diffusion and the nonlocal aggregation. Albeit the unique solution structures for the scalar case, such the blowup criterion and diverse spatial patterns, the multi-species Keller-Segel system manifests enriched phenomenon \cite{Kurganov2014, He2019}.
	

	Although the system (\ref{sym1a})-(\ref{sym1c}) incorporates specific models with distinct solution properties, designing numerical methods for such a generic model of complex fluids should meet certain unified conditions. On one hand, proper numerical discretization should preserve the properties satisfied by the solutions to the continuous  model, such as positivity preserving, mass conservation and energy dissipation. As we shall demonstrate in the following, such properties are indeed shared by the models in the form of (\ref{sym1a})-(\ref{sym1c}). On the other hand, the numerical methods should involve accurate and efficient treatment on the convolution terms with singular but integrable kernels. The Coulomb potential, the steric repulsion potential and the Newtonian potential are all examples of such type, while naive numerical treatment of the associated convolutions may result in inconsistent approximations. In addition, the complexity of directly computing the convolution terms is $O(N^2)$ where $N$ denotes the number of spatial grid points, and thus a fast algorithm is desired to mitigate the computation burden. 
	
	
	There have been only a few known results in designing structure preserving schemes for general nonlocal drift-diffusion equations with smooth kernels, which in theory can be extended to the multi-species cases. 
	Carrillo, Chertock and Huang \cite{Carrillo2014} propose a finite volume scheme with positivity preserving and entropy decreasing properties when a CFL-type constraint is satisfied, where the issue of singular kernels is briefly mentioned but not explored. In \cite{Pareschi2018}, Pareschi and Zanella construct explicit and semi-implicit numerical schemes with the Chang-Cooper formulation, which are shown to capture the asymptotic steady states.
	
	
	It is also worth noting that in each specific model of (\ref{sym1a})-(\ref{sym1c}), one may find ample results in numerical methods with improved properties which are only valid for a small class of models or their equivalent forms. In particular, when the kernel function takes some special form, the convolution term can be replaced by a Poisson equation for the effective potential. 
	When the kernel $\mathcal W$ vanishes, the electrostatic potential
	$\Phi_{\mathcal{K}}(\boldsymbol{x}) = (\mathcal{K}*\rho)(\boldsymbol{x})$ related to concentrations of the ionic species can be determined by the Gauss's law, i.e.,
	\begin{equation}
	-\Delta \Phi_{\mathcal{K}}(\boldsymbol{x}) = \rho(\boldsymbol{x}),\quad \boldsymbol{x} \in \Omega,
	\end{equation} 
	with certain boundary conditions, and the system (\ref{sym1a})-(\ref{sym1c}) becomes the familiar PNP (Poisson-Nernst-Planck) system and there have been quite a few numeric studies on such simplified model. For instance, Liu and Wang \cite{Hai2014} have designed and analyzed a free energy satisfying finite difference method in a bounded domain 
	that are conservative, positivity preserving and of the first order in time and the second order in space under a parabolic CFL condition $\Delta t = O((\Delta x)^2)$. After that, an arbitrary-order discontinuous Galerkin version was given by them \cite{Hai2017}. Besides, a finite element method using a method of lines approached developed by Metti, Xu and Liu \cite{Metti2016} enforces the positivity of the computed solutions and obtains the discrete energy decay but works for the certain boundary while the scheme developed by Hu and Huang \cite{hu2019} works for the general boundaries. In addition, Liu and Maimaitiyiming propose a second order unconditional positivity preserving scheme \cite{Hai2019}. etc.  When the kernel $\mathcal W$ takes the Newtonian potential, the system (\ref{sym1a})-(\ref{sym1c}) becomes the Keller-Segel equations, and most of the numerical methods are designed for the augmented system, where the aggration potential function $\mathcal{W}*\theta$ solves an additional elliptic equation (or its parabolic counterpart) \cite{Filbet2006, Chertock2008, Kurganov2014, Liu2018, Almeida2019}. Besides, Liu, Wang, Zhao and Zhou adopt the finite volume scheme as in \cite{Carrillo2014} to approximate the general model (\ref{sym1a})-(\ref{sym1c}), which is of first-order in space, satisfies positivity preserving, mass conservation and energy dissipation properties and can also deal with the regularized singular kernel preliminarily, but also suffers from the parabolic-type CFL constraint \cite{Liu2019}. 
	
	
	The evaluation of the convolution-type nonlocal field usually bottlenecks the simulation in terms of both accuracy and efficiency, especially when the kernel is singular. Direct summation of the discrete convolution resulted from proper quadrature requires $O(N^2)$ costs (recall that $N$ is the number of uniform grid points), and it is not practical even for one dimension case, let alone the most interesting two and three dimensional Newtonian potential.  Therefore, it is imperative to design fast algorithms for such nonlocal field when the densities are given on uniform/nonuniform grid points. 
	In fact, there has been a lot of research focused on such problems, including the famous fast multipole method (FMM) \cite{Yarvin1999,Gimbutas2019}, which is  some challenging in coding though. Therefore, we would like to design simpler codes even with some efficiency sacrifice. In fact, the densities are usually given  on {\sl uniform }grids in our problem. Convolution with smooth kernels has been exploited in \cite{Qiang2010} based on Newton-Cotes quadrature. The key point therein lies in the fact that the resulted summation is of discrete convolution and it can be accelerated with discrete Fast Fourier Transform, see also \cite{vico2016jcp} as an another example. It will not be different for singular kernels as long as the singularity is properly treated.  
	It is worth mentioning a fast convolution algorithm on {\sl nonuniform} general grid \cite{Zhang2020} if one has to employ adaptive grid in simulation, for example, in the case of boundary layers.
	
	
	Our primary objective of this work is to propose a numerical scheme which satisfies the properties such as positivity preserving, mass conservation and energy dissipation, while the convolution of singular kernel should be computed efficiently and accurately. And it is certainly favored if the limitation of the CFL condition is removed or relieved. Also, we expect the scheme has the flexibility to handle the problems with various boundary conditions as well as the Cauchy problem. For the purpose, we consider a reformulation of the model (\ref{sym1a})-(\ref{sym1c}) in the symmetric form, which is common reformulation in relative entropy estimates of Fokker-Planck equations. Such an equivalent form facilitates the design of numerical methods which preserves various properties of the original model. 
	And it is worth mentioning constructing numerical schemes method based on such a reformulation has been proven successful for some specific examples of the system (\ref{sym1a})-(\ref{sym1c}) and a few other Fokker-Planck type models, e.g., \cite{Hai2014, Hai2019, Hu20192, Xu2016, Liu2018, Jin2011, Jin20112}. 
	
	
	The symmetric form of (\ref{sym1a})-(\ref{sym1c}) is given by
	\begin{align} 
	&\partial_t c_m(\boldsymbol{x}, t) = \nabla \cdot \left( \exp\{-\left( z_m \mathcal{K}*\rho + \mathcal{W}*\theta \right)\} \nabla \dfrac{c_m}{\exp\{-\left( z_m \mathcal{K}*\rho + \mathcal{W}*\theta \right)\}} \right), \label{model1a} \\
	&c_m(\boldsymbol{x},0) = c^0_{m}(\boldsymbol{x}), ~~m=1,\cdots,M. \label{model1c}
	\end{align}
	And with the symmetrization, the free energy can be rewritten as
	\begin{equation}
	\label{free1}
	\mathcal{F} = \int_{\Omega} \left\{ \sum_{m = 1}^M c_m \log \dfrac{c_m}{\exp\{ -\frac{1}{2} \{z_m (\mathcal{K}*\rho)(\boldsymbol{x}) + (\mathcal{W}*\theta)(\boldsymbol{x})\}\} } \right\} \,\mathrm{d} \boldsymbol{x}.
	\end{equation}
	It is well known that the free energy dissipates along the dynamics, which can be seen by directly taking a derivative of $\mathcal{F}$ with respect to $t$ and using the self-adjointness of the kernel $\mathcal{K}(\boldsymbol{x})$ and $\mathcal{W}(\boldsymbol{x})$
	\begin{align}
	\frac{\mathrm{d}}{\mathrm{d}t}\mathcal{F} &= \sum_{m = 1}^{M} \int_{\Omega} \left\{ \log \dfrac{c_m}{\exp\{ -\frac{1}{2} f_m \}} \dfrac{\partial }{\partial t} c_m + c_m \dfrac{\partial}{\partial t} \log \dfrac{c_m}{\exp\{ -\frac{1}{2} f_m \}} \right\} \,\mathrm{d} \boldsymbol{x} \nonumber\\
	&= - \sum_{m = 1}^{M} \int_{\Omega} \dfrac{\exp^2\{ - f_m \}}{c_m} \left| \nabla \dfrac{c_m}{\exp\{ - f_m \}}\right|^2 \,\mathrm{d} \boldsymbol{x} \leqslant 0, \nonumber
	\end{align}
	where we introduce the auxiliary convolution-type field 
	\begin{equation}
	\label{fm}
	f_m = f_m(\boldsymbol{x}, t) := z_m (\mathcal{K}*\rho)(\boldsymbol{x}) + (\mathcal{W}*\theta)(\boldsymbol{x}).
	\end{equation}
	If we denote the dissipation by
	\begin{align}
	\label{disspation}
	D = \sum_{m = 1}^{M} \int_{\Omega} \dfrac{\exp^2\{ - f_m \}}{c_m} \left| \nabla \dfrac{c_m}{\exp\{ - f_m \}}\right|^2 \,\mathrm{d} \boldsymbol{x},
	\end{align}
	which is clearly nonnegative, and hence we have derived the energy dissipation relation
	\begin{align}
	\frac{\mathrm{d}}{\mathrm{d}t} \mathcal{F} + D = 0.\label{free0}
	\end{align}

	In this work, we first construct a semi-discrete  finite volume scheme for the drift-diffusion equation in the symmetric form (\ref{model1a}) that is of second order accuracy in space. It not only satisfies positivity preserving and mass conservation properties, but also maintains the energy dissipation relation at the semi-discrete level. With the semi-implicit treatment, our scheme has shown improved stability performance with no CFL-type constraint. To numerically approximate the convolution terms with singular kernel functions,
	we propose a fast algorithm on the uniform grid with almost optimal accuracy and efficiency where the complexity in computing the convolution with singular kernels is reduced from $O(N^2)$ to $O\left(N \log (N)\right)$, and thus the numerical cost is tremendously lowered, especially in high dimensional cases. We have provided extensive numerical tests to verify the properties of the scheme, such as unconditional stability, numerical convergence, energy dissipation, accurate capture of the equilibrium states. In addition, plenty of specified numerical experiments, such as the finite-size effect of ionic fluids, the concentration of ions at the boundary, the blowup phenomenon of the Keller-Segel equations, are carried out with care, showing the strong promise for practical simulations of realistic scientific problems. 
	
	We would like to emphasize that, another strength of the numerical scheme is its remarkable capability to preserve additional properties of specific models. Because the symmetric form rather than the original form of the Fokker-Planck equations is frequently leveraged to investigate the unique solution structure of the models, such as strong stability estimates and relative entropy estimates, etc., it is shown that in a few cases that the numerical discretization of the symmetric form can inherit the similar properties on the discrete level, see e.g., \cite{Hu20192, Jin2011, Jin20112, Liu2018}. 
	
	
	The rest of the paper is organized as follows. In Section 2, we construct a finite volume scheme to the system (\ref{model1a})-(\ref{model1c}) in 1D in the semi-discrete level, prove its properties: positivity preserving, mass conservation and discrete free energy dissipation, and present the fast algorithm for computing the convolution-type field. In the same section, the fully-discrete scheme and the extension to the multi-dimensional cases are also discussed. In Section 3, we give the error estimates of the fast convolution algorithm to show the second order convergence. And it's observed that the error estimates hold no matter whether the kernel is smooth or singular. In Section 4 we verify the properties of our numerical method with numerous test examples, such as unconditional stability, numerical convergence, energy dissipation and the complexity in computing
	the convolution with singular kernels. Additionally, we provide series of numerical experiments to demonstrate the finite size effect and the concentration of ions at the boundary as well as the blowup phenomenon of the Keller-Segel equations. Concluding remarks and the expectations in the following research are given in Section 5.
	
	
	\section{Numerical Schemes}
	In this section, we propose a semi-implicit finite volume scheme for the multi-species model 
	(\ref{model1a})-(\ref{model1c}) in one-dimension and two-dimension which is of second order in space, first order in time. Furthermore, our scheme involves accurate and efficient fast algorithm on the convolution terms with singular but integrable kernels. And then we show that the scheme satisfies mass conversation, positivity preserving and entropy dissipation properties. Firstly, we describe the one-dimensional numerical scheme in the following.
	
	\subsection{The one-dimensional Case}
	We choose  the computational domain as $[-L, L]$ and a uniform mesh grid $\mathcal T =\{x_{j}\big|~~x_{j} = -L + (j + N) \Delta x, j = -N,
	\ldots, N, \Delta x= L/N\}$.  
	 A semi-discrete finite volume scheme reads as
	\begin{equation}
	\label{bsys2}
	\frac{\mathrm{d} \bar{c}_{m, j}(t)}{\mathrm{d} t} = - \frac{F_{m, j + \frac{1}{2}}(t) - F_{m, j - \frac{1}{2}}(t)}{\Delta x}, ~~m = 1, \cdots, M,
	\end{equation}
	where $\bar{c}_{m, j}(t)$ is the average concentration of the $m-$th ionic species on $[x_{j - \frac{1}{2}}, x_{j + \frac{1}{2}}]$ for $j = -N+1, \cdots, N - 1$  and $ [x_{-N}, x_{-N + \frac{1}{2}}]$ or $[x_{N - \frac{1}{2}}, x_{N}]$ for $j = -N$ or $N$. The numerical flux ${F}_{m, j + \frac{1}{2}}(t)$ is defined below
	\begin{equation}
	\label{bsys3}
	F_{m, j + \frac{1}{2}}(t) := -\frac{1}{\Delta x} \exp\{-f_{m, j + \frac{1}{2}}(t)\} \left\{ \frac{\bar{c}_{m, j + 1}(t)}{\exp\{-f_{m, j + 1}(t)\}} - \frac{\bar{c}_{m, j }(t)}{\exp\{-f_{m, j}(t)\}}\right\},
	\end{equation}
	where $\exp\{-f_{m, j + \frac{1}{2}}(t)\}$ takes the harmonic mean of $\exp\{-f_{m, j}(t)\}$ and $\exp\{-f_{m, j + 1}(t)\}$, i.e.,
	\begin{equation}
	\label{bsys5}
	\left\{\exp\{-f_{m, j + \frac{1}{2}}(t)\}\right\}^{-1} = \frac{1}{2} \left\{\left\{ \exp\{- f_{m, j}(t)\}\right\}^{-1} + \left\{\exp\{- f_{m, j + 1}(t)\} \right\}^{-1}\right\},
	\end{equation} with $f_{m,j}(t)$ denoting the numerical approximation of  $f_{m}(x_{j},t)$ at time $t$.
The harmonic mean (\ref{bsys5}) has been used in numerics \cite{Hu20192} , but it is not necessary, see \cite{Hai2014} for an alternative choice: the algebraic mean. 

	A fully-discrete finite volume scheme by applying the backward Euler method while treating the convolution-type field $f_{m, j + 1}$ explicitly in numerical flux term (\ref{bsys3}) reads as follows
	\begin{align}
	&\frac{\bar{c}_{m, j}^{n + 1} - \bar{c}_{m, j}^{n}}{\Delta t} = - \frac{{F}_{m, j + \frac{1}{2}}^{n + 1} - {F}_{m, j - \frac{1}{2}}^{n + 1}}{\Delta x},\label{sysodeb} \\
	&{F}_{m, j + \frac{1}{2}}^{n + 1} = -\frac{1}{\Delta x} \exp\{-f_{m, j + \frac{1}{2}}^{n }\} \left\{ \frac{\bar{c}_{m, j + 1}^{n + 1}}{\exp\{-f_{m, j + 1}^{n }\}} - \frac{\bar{c}_{m, j}^{n + 1}}{\exp\{-f_{m, j}^{n }\}}\right\},  \label{sysfluxb}
	\end{align}
	for all $m = 1, \ldots, M$ and $j = -N,\ldots, N$. The scheme will be complemented by a fast solver of the field $f_{m,j}$, on which we shall elaborate immediately in the coming subsection. We emphasize that the fully-discrete scheme (\ref{sysodeb})-(\ref{sysfluxb}) is only linearly implicit, and thus it avoids the use of nonlinear solvers. Besides, it satisfies the following structure preserving properties. 
		\begin{thm}\label{thm1b}(1D fully-discrete positivity preserving)
		Consider the one-dimensional fully-discrete finite volume scheme (\ref{sysodeb})-(\ref{sysfluxb}) of the system (\ref{model1a})-(\ref{model1c}) with initial data $\bar c_{m, j}^0 \geqslant 0, ~\forall ~ m = 1, \cdots, M$,  $j=-N,\ldots,N$, 
		then the cell averages $\bar{c}_{m, j}^n \geqslant 0$ for all species and grid points.
	\end{thm}

	\begin{thm}\label{thm2b}(1D semi-discrete free energy dissipation estimate)
		Consider the one-dimensional semi-discrete finite volume scheme (\ref{bsys2})-(\ref{bsys4}) of the system (\ref{model1a})-(\ref{model1c}) with initial data $\bar c_{m, j}(0) > 0, \, m = 1, \cdots, M$. Assume we take no flux boundary conditions, i.e., the discrete boundary conditions satisfy $F_{m, -N - \frac{1}{2}} = F_{m, N+ \frac{1}{2}} = 0, \, m = 1, \cdots, M$.
		Then, for the semi-discrete form of the free energy $\mathcal{F}$ 
		and the disspation $D$, we have
		\begin{equation}
		\dfrac{\mathrm{d}}{\mathrm{d} t} E_{\Delta}(t) = - D_{\Delta}(t) \leqslant 0, \quad \forall \,t \geqslant 0.
		\end{equation}
	\end{thm}
Here $E_{\Delta}(t)$ and $D_{\Delta}(t)$, 
semi-discrete free energy
	and dissipation, are defined explicitly as follows
	\begin{equation}
	\label{disenergyb}
	\begin{aligned}
	E_{\Delta}(t) = \Delta x \sum_{m = 1}^M \sum_{j = -N_x}^{N_x} \bar{c}_{m, j}(t) \log  \frac{\bar{c}_{m, j}(t)}{\exp\{- \frac{1}{2} f_{m, j}(t)\}}, 
	\end{aligned}
	\end{equation}
	\begin{equation}
	\begin{aligned}
	D_{\Delta}(t) &= \frac{1}{\Delta x} \sum_{m = 1}^{M} \sum_{j = -N_x}^{N_x} \exp\{-f_{m, j + \frac{1}{2}}(t)\} \cdot \frac{1}{\beta_{m, j}(t)} \left\{ \frac{\bar{c}_{m, j + 1}(t)}{\exp\{-f_{m, j + 1}(t)\}} - \frac{\bar{c}_{m, j }(t)}{\exp\{- f_{m, j}(t)\}}\right\}^2,
	\end{aligned}
	\end{equation}
	where $\beta_{m, j}(t)$ sits between $\frac{\bar{c}_{m, j }(t)}{\exp\{-f_{m, j}(t)\}}$ and $\frac{\bar{c}_{m, j + 1}(t)}{\exp\{- f_{m, j + 1}(t)\}}$.

\begin{rem}
The proof of Theorem \ref{thm1b} is similar to that in \cite{Hu20192} although the numerical scheme is different,  thus we choose to put the proof in Appendix \ref{app1}.
	Similarly, the proof of Theorem \ref{thm2b} is similar to that in \cite{Liu2019} and we present its proof in Appendix \ref{app2}.
	\end{rem}

\begin{rem}
	
When considering the PNP system for simulation of ionic channels instead of the field model (\ref{sym1a})-(\ref{sym1c}) in our paper, Liu and Maimaitiyiming have given 1D fully-discrete free energy dissipation estimate in \cite{Hai2019}. The main difference in our method is that we involve accurate and efficient fast algorithm on the convolution terms with singular but integrable kernels.
	\end{rem}

\subsection{Fast and accurate evaluation of the convolution-type field}

For practical simulation reasons,  we aim to design  a \textsl{fast} and \textsl{accurate} solve for the field \eqref{fm} evaluation on such uniform grid $\mathcal T$.	
The field $f_{m}(x,t)$ is given as convolutions and its numerical approximation $f_{m,j}(t)$ can be computed as follows
	\begin{equation}
	\label{bsys4}
f_{m,j}(t)= \int_{-L}^{L}  \left( z_m \mathcal{K} \left(x_{j}-x\right) {\rho}^h(x, t) + \mathcal{W}\left(x_{j}-x\right) {\theta}^h(x, t) \right) \,\mathrm{d} x,
	\end{equation}
where
	\begin{equation}
	\rho^h(x, t) = \sum_{m = 1}^M z_m {c}_{m}^h(x, t), \quad \quad 
	\theta^h(x, t) = \sum_{m = 1}^M {c}_{m}^h(x, t).
	\end{equation}
	Here, ${c}_{m}^h(x, t)$ is chosen as the piecewise linear interpolation of $c_{m}(x)$ using $\bar c_{m,j}$, and  is given explicitly 
	\begin{equation}\label{chFun}
	{c}_{m}^h(x, t) = \sum\nolimits_{j = -N}^{N} \bar c_{m, j}(t) e_j(x),  \quad ~\forall~x \in [-L, L],
	\end{equation}
	with $e_j$ being the piecewise linear interpolation function, i.e., the typical hat function. Obviously, $c_{m}^{h}$ is a second order approximation of $c_{m}$. Equivalently, $f_{m,j}(t)$ can be rewritten as a summation of convolutions of kernels $\mathcal{K}(x), \mathcal{W}(x)$ with  density $\rho^{h}(x), \theta^{h}(x)$ on grid $\mathcal T$ as follows
\begin{equation}
	\label{f1}
	f_{m, j}(t) =  \sum_{p = 1}^M   z_m z_p\int_{-L}^L  \mathcal{K} \left(x_{j}-x\right) {c}_{p}^h(x, t) \,\mathrm{d} x +  \sum_{p = 1}^M \int_{-L}^L\mathcal{W}\left(x_{j}-x\right) {c}_{p}^h(x, t) \,\mathrm{d} x.
	\end{equation}

Plugging ${c}^h_m$ (\ref{chFun}) into the above equation, we have 
		\begin{eqnarray}
		\nonumber
		f_{m, j}(t)  &= &\sum_{p = 1}^M   z_m z_p \sum_{i=-N}^{N}  \bar c_{p, i}(t) \int_{-L}^{L} \mathcal{K} (x_{j}-x)  e_i(x) \,\mathrm{d} x \\
		\label{f2}& +&\sum_{p = 1}^M   ~ \sum_{i=-N}^{N}  \bar c_{p, i}(t) \int_{-L}^{L} \mathcal{W} (x_{j}-x)  e_i(x) \,\mathrm{d} x.
			\end{eqnarray}	
		Direct computation of \eqref{f1} by simply collecting all terms in \eqref{f2} leads to a simple but inefficient field solver,  
	because the corresponding complexity $O(N^2)$ shall bottleneck the simulation for large $N$. Therefore, for practical simulations, it is imperative to design
	an efficient field solver of a smaller complexity and sufficient spatial accuracy inherited from the density function.

	\
	
		Because the basis function $e_j(x)$ has local compact support, the whole interval integrals in \eqref{f2} can be reduced to local integrals. 
		Take the interior point $x_{j}$ for example, the integral with kernel $\mathcal K$ is reduced to 
		\begin{equation}\label{Tensor}
		T_{j,i}:=\int_{-L}^L \mathcal K(x_j-x) e_i(x) {\rm d} x = \Delta x \int_{-1}^1 {\mathcal K}(\,(j-i-x)\Delta x\,)  ~\widehat e_0(x) {\rm d} x,
		\end{equation}
		 after a change of variables $x = x_i + \tilde x \Delta x $ and removing unnecessary $\tilde \ $. 
		 Here, $T_{j,i}$ denotes the convolution matrix element, and $\widehat e_0(x)$ is the hat function over interval $[-1,1]$, i.e., $\widehat e_0(x) = 1-x,  \mbox{ for } x  \in [0,1] ; 1+x,  \mbox{ for }  x\in [-1,0]$.
		 From \eqref{Tensor}, it is clear that the double-index matrix element $T_{j,i}$ depends only on the index difference $j-i$, therefore, it naturally reduces to a single-index element $T_{j-i}:=T_{j,i}$.
		 It is important to address that such index reduction only holds for \textsl{uniform } mesh grid. Similar argument also applies to integral of boundary points, and we shall omit details for brevity.
		 
		 \

		 For fixed uniform grid, after suitable change of variables, the nonlocal field can be reformulated as 
	\begin{equation}
	\label{ffinal}
	f_{m, j}(t) =  \sum_{p = 1}^M z_m  z_p  \left[ \sum_{i=-N+1}^{N-1} \bar c_{p, i}(t)  T^{\mathcal{K}}_{j-i} \right] 
	+  \sum_{p = 1}^M \left[ \sum_{i=-N+1}^{N-1} \bar c_{p, i}(t)  T^{\mathcal{W}}_{j-i} \right ]+ H_{m, j}^{+} + H_{m, j}^{-},
	\end{equation}
	for all $j = -N,\ldots, N$. Here, the convolution tensor are defined as $$T^{U}_{j-i} := \Delta x \int_{-1}^{1} U\left((j - i-x) \Delta x  \right)  \widehat{e}_0(x) \,\mathrm{d} x, \quad U = \mathcal K \mbox{ or } \mathcal W,$$
	and the boundary contribution $H_{m,j}^{\pm }$ are explicitly given 
	\begin{eqnarray*}	
	H_{m, j}^{\pm} & = &  \sum_{p = 1}^M  z_m z_p \bar c_{p, \pm N} \, \Delta x \!\int_{0}^{1} \mathcal{K} \left((j \mp  (N - x)) \Delta x \right)  \widehat{e}_0^{+}(x) \,\mathrm{d} x \\
&+ &  \sum_{p = 1}^M \bar c_{p, \pm N}\,\Delta x \int_{0}^{1} \mathcal{W} \left((j \mp ( N - x) ) \Delta x \right)  \widehat{e}_0^{+}(x) \,\mathrm{d} x, 
	\end{eqnarray*}
with $\widehat{e}_0^{+}(x) := 1-x$. 
It is clear that, in \eqref{ffinal},
the first two inner summations over index $i$ are discrete convolutions,  and the last two terms can be directly computed within $O(M N)$ operations for all $m$ and $j$. As is well known in \cite{vico2016jcp,Zhang2020,Qiang2010}, 
discrete convolutions can be efficiently accelerated via the discrete Fast Fourier Transform (FFT)  on a doubly zero-padded discrete density, therefore, the first two summations can be computed within $O(M N\log N)$ operations. In summary, the total computational cost for \eqref{ffinal} is $O(M N + M N \log N)$.  It is worth noting that the tensors and integrals involved depend on only the mesh size $\Delta x$ and index difference, then for a given mesh grid, these terms can be precomputed.
The pre-computation cost is $O(N)$ thanks to the single-index property, and it can be further reduced by half if the kernel is symmetric, i.e., $\mathcal K(x) = \mathcal K(-x)$.
\begin{rem}
The spatial accuracy (second order) inherits from the finite-volume discretization of the density. It does not matter whether the kernel is regular or singular, the second order convergence can be easily proved as long as the kernel is absolutely integrable. Therefore, such fast algorithm allows us to treat generic models which includes the well-known Newtonian.
Detailed error estimates for the 1D and 2D field evaluations are presented in Section \ref{Sec_Error}.
\end{rem}

\begin{rem}
The above fast algorithm can be extended easily, a little tedious though, to multi-dimension problems. As long as the discrete tensors are precomputed,  the $d$-dimension convolution can be evaluated efficiently with the FFT similarly.  Special care should be taken of the boundary points in higher dimensions problem. The 2D case is detailed in the coming subsection.
\end{rem}

	\subsection{The multi-dimensional Case}
	It is natural to extend our scheme to the multi-dimensional problems. In this part, we take the two-dimensional case as an example. In this case, we choose the computational domain as $[-L_x, L_x] \times [-L_y, L_y]$ and take a uniform mesh grid $\mathcal T =\{\left( x_{j}, y_{k} \right) \big|~~x_{j} = -L_x + (j + N_x) \Delta x, j = -N_x,
	\ldots, N_x, \Delta x= L_x/N_x, y_{k} = -L_y + (k + N_y) \Delta y, k = -N_y,
	\ldots, N_y, \Delta y= L_y/N_y\}$.  
	A semi-discrete finite volume scheme reads as for $m = 1, \cdots, M,$
	\begin{equation}
	\label{bsys22d}
	\frac{\mathrm{d} \bar{c}_{m, j, k}(t)}{\mathrm{d} t} = - \frac{{F}^x_{m, j + \frac{1}{2}, k}(t) - {F}^x_{m, j - \frac{1}{2}, k}(t)}{\Delta x} - \frac{{F}^y_{m, j, k + \frac{1}{2}}(t) - {F}^y_{m, j, k - \frac{1}{2}}(t)}{\Delta y}, 
	\end{equation}
	where $\bar{c}_{m, j, k}$ is the average concentration of the $m-$th ionic species on $[x_{j - \frac{1}{2}}, x_{j + \frac{1}{2}}] \times [y_{k - \frac{1}{2}}, y_{k + \frac{1}{2}}]$ for $j = -N_x + 1, \cdots, N_x - 1, k = -N_y+1, \cdots, N_y - 1,$ or on boundary cell for $j = \pm N_x$ or $k = \pm N_y$. 
	Similarly, the numerical flux ${F}^x_{m, j + \frac{1}{2}, k}(t)$ and ${F}^y_{m, j, k + \frac{1}{2}}(t)$ are defined in the following forms respectively:
	\begin{align}
	&{F}^x_{m, j + \frac{1}{2}, k}(t) = -\frac{1}{\Delta x} \exp\{-f_{m, j + \frac{1}{2}, k}(t)\} \left\{ \frac{\bar{c}_{m, j + 1, k}(t)}{\exp\{-f_{m, j + 1, k}(t)\}} - \frac{\bar{c}_{m, j, k}(t)}{\exp\{-f_{m, j, k}(t)\}}\right\}, \label{bsys32dx} \\
	&{F}^y_{m, j, k + \frac{1}{2}}(t) = -\frac{1}{\Delta y} \exp\{-f_{m, j, k + \frac{1}{2}}(t)\} \left\{ \frac{\bar{c}_{m, j, k + 1}(t)}{\exp\{-f_{m, j, k + 1}(t)\}} - \frac{\bar{c}_{m, j, k}(t)}{\exp\{-f_{m, j, k}(t)\}}\right\}. \label{bsys32dy}
	\end{align}
	And $\exp\{-f_{m, j + \frac{1}{2}, k}(t)\}$ and $\exp\{-f_{m, j, k + \frac{1}{2}}(t)\}$ take the harmonic mean similarly:
	\begin{align}
	\label{bsys52dx}
	\left\{\exp\{-f_{m, j + \frac{1}{2}, k}(t)\}\right\}^{-1} = \left\{\frac{1}{2} \left\{ \exp\{-f_{m, j, k}(t)\}\right\}^{-1} + \left\{\exp\{-f_{m, j + 1, k}(t)\} \right\}^{-1}\right\}, \\
	\label{bsys52dy}
	\left\{\exp\{-f_{m, j, k + \frac{1}{2}}(t)\}\right\}^{-1} = \left\{\frac{1}{2} \left\{ \exp\{-f_{m, j, k}(t)\}\right\}^{-1} + \left\{\exp\{-f_{m, j, k + 1}(t)\} \right\}^{-1}\right\},
	\end{align}
	where $f_{m,j, k}(t)$ denotes the numerical approximation of  $f_{m}(x_{j},y_{k}, t)$ at time $t$.
	
	In this paper, a fully-discrete finite volume scheme of the ODEs system (\ref{bsys22d}) is proposed by applying the backward Euler method while treating the convolution-type field $f_{m, j \pm 1, k}$ and $f_{m, j , k \pm 1}$ in the numerical flux (\ref{bsys32dx}) and (\ref{bsys32dy}) explicitly for all $j= -N_x,\ldots, N_x$, $k = -N_y,\ldots, N_y$ and $m = 1, \ldots, M$, which is
	\begin{align}
	&\frac{\bar{c}_{m, j, k}^{n + 1} - \bar{c}_{m, j, k}^{n}}{\Delta t} = - \frac{F_{m, j + \frac{1}{2}, k}^{n + 1} - F_{m, j - \frac{1}{2}, k}^{n + 1}}{\Delta x} - \frac{F_{m, j, k + \frac{1}{2}}^{n + 1} - F_{m, j, k - \frac{1}{2}}^{n + 1}}{\Delta y}, \label{ful1}\\
	&F_{m, j + \frac{1}{2}, k}^{n + 1} = -\frac{1}{\Delta x} \exp\{-f_{m, j + \frac{1}{2}, k}^{n }\} \left\{ \frac{\bar{c}_{m, j + 1, k}^{n + 1}}{\exp\{-f_{m, j + 1, k}^{n }\}} - \frac{\bar{c}_{m, j, k}^{n + 1}}{\exp\{-f_{m, j, k}^{n }\}}\right\}, \label{ful2}\\
	&F_{m, j, k + \frac{1}{2}}^{n + 1} = -\frac{1}{\Delta y} \exp\{-f_{m, j, k + \frac{1}{2}}^{n }\} \left\{ \frac{\bar{c}_{m, j , k + 1}^{n + 1}}{\exp\{-f_{m, j, k + 1}^{n }\}} - \frac{\bar{c}_{m, j, k}^{n + 1}}{\exp\{-f_{m, j, k}^{n }\}}\right\}. \label{ful3}
	\end{align}
	
	Clearly, the fully-discrete scheme (\ref{ful1})-(\ref{ful3}) is again only linearly implicit, and thus it avoids the use of nonlinear solvers. The fully-discrete positivity preserving and the semi-discrete entropy dissipation properties for two-dimensional case are given as follows. All the proofs for two-dimensional case are similar to one-dimensional case, so we omit them in the following.
	
	\begin{thm}\label{thm1c}(2D fully-discrete positivity preserving)
		Consider the two-dimensional fully-discrete finite volume scheme (\ref{ful1})-(\ref{ful3}) of the system (\ref{model1a})-(\ref{model1c}) with initial data $\bar c_{m, j, k}^0 \geqslant 0, ~\forall ~ m = 1, \cdots, M$, $\forall ~ j = -N_x, \cdots, N_x, ~k = -N_y, \cdots, N_y$. 
		Then, the cell averages $\bar{c}_{m, j, k}^n \geqslant 0, ~\forall ~ m = 1, \cdots, M, \,\forall ~ j, k, n$.
	\end{thm} 
	
	\begin{thm}\label{thm2c}(2D semi-discrete free energy dissipation estimate)
		Consider the two-dimensional semi-discrete finite volume scheme (\ref{bsys22d})-(\ref{bsys32dy}) of the system (\ref{model1a})-(\ref{model1c}) with initial data $\bar c_{m, j, k}(0) > 0, \, m = 1, \cdots, M$. Assume we take no flux boundary conditions on $[-L_x, L_x]\times [-L_y, L_y]$, i.e., the discrete boundary conditions satisfy $F_{m, -N_x - \frac{1}{2}, k} = F_{m, N_x + \frac{1}{2}, k} = F_{m, j, -N_y - \frac{1}{2}} = F_{m, j, N_y + \frac{1}{2}} = 0, \, m = 1, \cdots, M$. 
		Then we have
		\begin{equation*}
		\dfrac{\mathrm{d}}{\mathrm{d} t} E_{\Delta}(t) 
		\leqslant 0.
		\end{equation*}
	\end{thm}
	Here for the two-dimensional case, the semi-discrete free energy with respect to $\mathcal{F}$ is defined in the following form
	\begin{equation}
	\begin{aligned}
	E_{\Delta}(t) = \Delta x \Delta y \sum_{m = 1}^M \sum_{j = -N_x}^{N_x} \sum_{k = -N_y}^{N_y} \bar{c}_{m, j,k}(t) \log  \frac{\bar{c}_{m, j,k}(t)}{\exp\{- \frac{1}{2} f_{m, j,k}(t)\}}.
	\end{aligned}
	\end{equation}

	\begin{rem}		
	The two-dimensional scheme will also be complemented by a fast solver of the field $f_{m,j, k}$ which is a natural extension of the one-dimensional case. Thus, we briefly describe the fast and accurate evaluation of the convolution-type field in two-dimension in the following. 
	\end{rem}
	
	The two-dimensional numerical approximation convolution $f_{m, j, k}(t)$ on grid $\mathcal T$ can be calculated by using the bilinear approximation ${c}^h_m$ instead of $c_m$:
	\begin{equation}
	\label{bsys42d}
	f_{m, j, k}(t) = \int_{-L_y}^{L_y} \int_{-L_x}^{L_x} \left( z_m \mathcal{K} \left(x_{j}-x, y_{k}-y\right) {\rho}^h(x, y, t) + \mathcal{W}\left(x_{j}-x, y_k - y \right) {\theta}^h(x, y, t) \right) \,\mathrm{d} x  \,\mathrm{d} y,
	\end{equation}
	where
	\begin{align}
	\rho^h(x, y, t) &= \sum_{m = 1}^M z_m {c}^h_{m}(x, y, t), \\
	\theta^h(x, y, t) &= \sum_{m = 1}^M {c}^h_{m}(x, y, t).
	\end{align}
	The bilinear approximation ${c}^h_m(x, y, t)$ of the concentration $c_m$ is given explicitly in the following,
	\begin{equation}
	\label{tildec_2d}
	{c}^h_{m}(x, y, t) = \sum_{j=-N_x}^{N_x} \sum_{k=-N_y}^{N_y} \bar c_{m, j, k}(t) e_{j, k}(x, y), 
	\end{equation}
	with $e_{j, k}(x, y) = e_j(x)e_k(y)$ for all $(x, y) \in [-L_x, L_x] \times [-L_y, L_y]$ (recall that $e_j$ is the hat function).  Equivalently, $f_{m,j, k}(t)$ can be rewritten as a summation of convolutions of kernels $\mathcal{K}(x, y), \mathcal{W}(x, y)$ with  density $\rho^{h}(x, y, t), \theta^{h}(x, y, t)$ on grid $\mathcal T$ and then we have 

	\begin{equation}
	\label{f2_2d}
	\begin{aligned}
	f_{m, j, k}(t) &= \sum_{p = 1}^M z_m  z_p \sum_{i=-N_x}^{N_x} \sum_{l=-N_y}^{N_y}  \bar c_{p, i, l}(t) \int_{-L_y}^{L_y} \int_{-L_x}^{L_x}  \mathcal{K} \left(x_{j}-x, y_k - y \right)  e_i(x) e_l(y) \,\mathrm{d} x \,\mathrm{d} y \\
	&+  \sum_{p = 1}^M \sum_{i=-N_x}^{N_x} \sum_{l=-N_y}^{N_y}  \bar c_{p, i, l}(t) \int_{-L_y}^{L_y} \int_{-L_x}^{L_x}  \mathcal{W}\left(x_{j}-x, y_k - y \right)   e_i(x) e_l(y) \,\mathrm{d} x \,\mathrm{d} y.
	\end{aligned} 
	\end{equation}
	
	Similarly, we adopt the efficient field solver of a smaller complexity and sufficient spatial accuracy inherited from the density function  in two-dimension instead of direct computation of $f_{m, j, k}$. It reduces the complexity $O(N^2)$ to $O(N \log N)$ (recall that $N = N_x*N_y$ is the total number of the grid points). Specifically, due to that $e_{j, k}(x, y) = e_j(x) e_k(y)$ has a local support, the whole space integrals in \eqref{f2_2d} can be reduced to local integrals.
	Then for a fixed uniform grid, after suitable change of variables, the nonlocal field can be reformulated as 		
	\begin{equation}
	\begin{aligned}
	f_{m, j, k}(t) &= z_m \Delta x \Delta y \sum_{p = 1}^M z_p \left[ \sum_{i=-N_x+1}^{N_x-1} \sum_{l=-N_y+1}^{N_y-1} \bar c_{p, i, l}(t)  T^{\mathcal{K}}_{j-i, k-l} \right] \\ &+  \Delta x \Delta y \sum_{p = 1}^M \left[ \sum_{i=-N_x+1}^{N_x-1} \sum_{l=-N_y+1}^{N_y-1} \bar c_{p, i, l}(t)  T^{\mathcal{W}}_{j-i, k-l} \right] + H_{m, j, k}^{+} + H_{m, j, k}^{-},
	\end{aligned} 
	\end{equation}
	for all $j = -N_x,\ldots, N_x$ and $k = -N_y,\ldots, N_y$ and the convolution tensor $T^{\mathcal{K}}_{j-i, k-l}$, $T^{\mathcal{W}}_{j-i, k-l}$ and the boundary contribution $H_{m, j, k}^{+}$, $H_{m, j, k}^{-}$ can be obtained similarly to the one-dimensional case. So for brevity, we omit the details of the expression of these variables.
		
	\begin{rem}
	Such fast convolution algorithm in two-dimension only holds for \textsl{uniform } mesh grids and the spatial accuracy (second order) inherits from the finite-volume discretization of the density.
	\end{rem}

	\section{Error Estimates on the Convolution}\label{Sec_Error}
	In this section, we shall analyze the error estimates of the above fast convolution algorithm for computing  the following convolution-type field
	\begin{equation}
	\label{convDim}
	\psi(\boldsymbol{x}) := \int_\Omega\mathcal{K}(\boldsymbol{x}-\boldsymbol{y}) \rho(\boldsymbol{y}) \,\mathrm{d}\boldsymbol{y}, \quad  \boldsymbol{x} \in \Omega:=[-L,L]^{d}
	\end{equation}
	where the kernel can be smooth or singular as long as the convolution is well-defined, and the density $\rho(\boldsymbol{x})$ is given on uniform grid mesh $\mathcal T_{h}$ that is inherited from the finite volume discretization.  The numerical approximation is given by replacing $\rho$ by its (bi)linear interpolation $\rho_{h}(\boldsymbol{x})$ as follows
	\begin{equation}
	\label{convDiscrete}
	\psi_{h}(\bx) := \int_\Omega\mathcal{K}(\boldsymbol{x}-\by) \rho_{h}(\by) \,\mathrm{d}\by, \quad  \boldsymbol{x} \in \Omega,
	\end{equation}
	with $h$ being the mesh size of the grid points.
	The error function is defined as $e_{h}(\boldsymbol{x}):= \psi(\boldsymbol{x})-\psi_{h}(\boldsymbol{x})$. Without loss of generality, we shall first present the error estimates for the one dimensional convolution, then give the results for the two dimensional case. 
	
	\begin{thm}\label{thm3}(Error estimates for the 1D convolution)
		Assuming density $\rho(x) \in C^{2}(\Omega)$, for absolutely integrable kernels $\mathcal{K}(x)$, the error $e_h$ of approximating (\ref{convDim}) with (\ref{convDiscrete}) satisfies
		\begin{equation}
		\label{es1}
		\|e_h\|_{L^{\infty}(\Omega)} \leqslant \frac{h^2}{8} \|\rho^{(2)}\|_{L^{\infty}(\Omega)} \; \max\limits_{x\in\Omega} \int_{\Omega} |\mathcal{K}(x-y)| \,\mathrm{d} y.
		\end{equation}
	\end{thm}
	
	\begin{proof}
		First,  the linear interpolation error of $\rho(x)$ on interval 
		$[x_j,x_{j}+h]$ reads as follows
		\begin{equation*}
		|\rho(x)-\rho_{h}(x)| \leq  ~ \left(\frac{h}{2}\right)^{2}  \big| \frac{\rho^{(2)}(\eta)}{2} \big| = \frac{h^{2}}{8} \big| \rho^{(2)}(\eta)\big|, \quad \eta \in [x_{j}, x_{j}+h].
		\end{equation*}
		Then, substituting the above estimate into $e_{h}(\bx)$, we have 
		\begin{equation*}
		\begin{aligned}
		|e_h(x)| 
		&\leqslant \int_{-L}^L |\mathcal{K}(x-y)| |\rho(y) - \rho_h(y)|\,\mathrm{d} y \\
		&\leqslant \frac{h^2}{8} \|\rho^{(2)}\|_{L^{\infty}} \int_{-L}^L |\mathcal{K}(x-y)| \,\mathrm{d} y.
		\end{aligned}
		\end{equation*}
		Taking maximum with respect to $x$ on both sides, the following estimates hold directly
		\begin{equation*}
		\|e_h\|_{L^{\infty}(\Omega)}\leqslant \frac{h^2}{8} \|\rho^{(2)}\|_{L^{\infty}} \max\limits_{x\in \Omega} \int_\Omega |\mathcal{K}(x-y)| \,\mathrm{d} y.
		\end{equation*}
	\end{proof}
	
	\begin{rem}
		From the proof, it is clearly seen that the second order convergence in mesh size $h$ stems from the linear interpolation of the density. As long as the kernel $\mathcal K$ is absolutely integrable, the error estimates hold no matter whether the kernel is smooth or singular. \end{rem}
	
	\begin{rem}
		In fact, the convergence order in $\psi$ is equal to the interpolation of density $\rho$. It can be raised to higher order as the density is approximated with some higher order interpolation scheme, for example, piecewise high order local polynomials.
	\end{rem}
	
	The above error estimates can be naturally extended to higher dimensional convolutions.  For the two dimensional convolution, we shall have the following error estimates.
	\begin{thm}\label{Err2D}
		(Error estimates for the 2D convolution)
		Assuming density $\rho(\bx) \in C^{(2,2)}(\Omega)$, for absolutely integrable kernels $\mathcal{K}(\bx)$, the error $e_h$ of approximating (\ref{convDim}) with (\ref{convDiscrete}) satisfies
		\begin{equation}
		\label{es3}
	\|e_h\|_{L^{\infty}(\Omega)} \leqslant {C} ~ h_{\rm max}^{2} \; \max\limits_{\bx\in\Omega} \int_{\Omega} |\mathcal{K}(\bx-\by)| \,\mathrm{d} \by,
		\end{equation}
		where $h_{\rm max}:=\max\{h_{x},h_{y}\}$ with $h_{x},h_{y}$ being the mesh size in $x$ and $y$ directions and 
		\begin{equation}
		\label{Cnst2Int}
		C = \frac{1}{4} \max \left\{ {\|  \partial_{y}^{2}\rho\|_{L^{\infty}(\Omega)}}, \|  \partial_{x}^{2}\rho\|_{L^{\infty}(\Omega)} + \frac{\| \partial_{x}^{2} \partial_{y}^{2}\rho\|_{L^{\infty}(\Omega)}}{8} h_{y}^2  \right\}.\end{equation}
		
	\end{thm}
	\begin{proof}
		The bilinear interpolation error estimates can be found in textbooks \cite{Guan}. Then, the remaining part follows directly from the 1D case, and we omit details here for brevity. 
	\end{proof}
	
	Hereafter, we present some commonly used kernels for ionic fluid in different dimensions. 
	\begin{cor}
		For kernels $\mathcal{K}(\bx) = |\bx|^{-\alpha}$, we have in 1D	\begin{equation}
		\|e_h\|_{L^{\infty}} \leqslant  \frac{1}{4(1-\alpha)} (2L)^{1-\alpha} h^{2}\|\rho^{(2)}\|_{L^{\infty}} ,\forall ~\alpha \in (0,1),
		\end{equation}
		and in 2D
		\begin{equation}
		\|e_h\|_{L^{\infty}} \leqslant      C~C_{\alpha} (2L)^{2-\alpha} ~h_{\rm max}^2  ,\forall ~\alpha \in (0,2),
		\end{equation}
		with $C_{\alpha} :=\int_{[-1,1]^{2}} |\bx|^{-\alpha} \,\mathrm{d}\bx < \infty$ and $C$ is given by \eqref{Cnst2Int}.
	
	\end{cor}
	\begin{proof} 
		The error estimates are reduced to compute $\max_{\bx\in \Omega} \int_{\Omega} |\mathcal{K}(\bx-\by)| \,\mathrm{d} \by$. In fact, it suffices to compute $\max_{\bx\in 2\Omega} \int_{2\Omega} |\mathcal{K}(\by)| \,\mathrm{d} \by$, where $2\Omega :=[-2L,2L]^{d}$. It is easy to prove the 1D results with simple calculation, while in 2D we have 
		\begin{eqnarray*}
			\max_{\bx\in \Omega} \int_{\Omega} |\mathcal{K}(\bx-\by)| \,\mathrm{d} \by &\leqslant & \max_{
				\bx\in 2\Omega} \int_{2\Omega} |\mathcal{K}(\by)| \,\mathrm{d} \by 
			\leqslant \max_{\bx \in 2\Omega}\int_{2\Omega} |\by|^{-\alpha} \,\mathrm{d} \by\\
			&\leqslant&  (2L)^{2-\alpha}\max_{\bx \in [-1,1]^{2}} \int_{ [-1,1]^{2}} |\by|^{-\alpha} \,\mathrm{d}\by:= (2L)^{2-\alpha} C_{\alpha}.
		\end{eqnarray*}
		The last inequality holds using a change of variables $\tilde \by = \by/(2L)$.
	\end{proof}
	
	\begin{cor}
		In 2D, for  kernel $\mathcal{K}(\boldsymbol{x}) = \ln |\boldsymbol{x}|$, we have
		the following estimate:
		\begin{equation}
		\label{ineq2}
		\|e_h\|_{L^{\infty}} \leqslant C~ C_{L} ~h_{\rm max}^2.
		\end{equation}
		where $C_{L} >0$ is a constant depending only on $L$ and $C$ is defined in \eqref{Cnst2Int}.
	\end{cor}
	\begin{proof}
		On $\Omega:=[-L, L]^{2}$, we have 
		\begin{eqnarray}
		\int_{\Omega} |\mathcal{K}(\bx-\by)| \,\mathrm{d} \by &\leqslant& 
		\int_{2\Omega} \big|\ln|\by| \big| \,\mathrm{d} \by \leqslant L_{0}^{2} \int_{[-\frac{\sqrt{2}}{2},\frac{\sqrt{2}}{2}]^{2} } (-\ln |\by| + \ln L_{0} )\,\mathrm{d}\by\\
		&\leqslant&2 L_{0}^{2}\ln L_{0}  + L_{0}^{2} \int_{[-\frac{\sqrt{2}}{2},\frac{\sqrt{2}}{2}]^{2}} -\ln |\by| \,\mathrm{d}\by \\
		&\leqslant& L_{0}^{2} \left( 2 \ln L_{0} + \pi/2\right),
		\end{eqnarray}
		with $L_{0}=2\sqrt{2} L$.
		
	\end{proof}

	\section{Numerical Tests and Experiments}
	In this section, we present extensive one- and two-dimensional numerical examples to verify the convergence order, the asymptotic complexity of the fast convolution algorithm and the properties of the numerical schemes, and explore various modeling phenomena, such as the finite size effect, the nonlocal aggregation, etc. 
	
	In the following numerical examples, we sometimes add some additional external field into the original model (\ref{model1a})-(\ref{model1c}) to make sure the steady states are effectively localized, that is to say, the  system we consider becomes for $m=1,\cdots,M$,
	\begin{align}
	&\partial_t c_m(\boldsymbol{x}, t) = \nabla \cdot \left( \exp\{-\left( z_m \mathcal{K}*\rho + \mathcal{W}*\theta + V_{\text{ext}} \right)\} \nabla \frac{c_m}{\exp\{-\left( z_m \mathcal{K}*\rho + \mathcal{W}*\theta + V_{\text{ext}} \right)\}} \right), \label{model2a} \\
	&c_m(\boldsymbol{x},0) = c^0_{m}(\boldsymbol{x}), \label{model2c}
	\end{align}
	where $V_{\text{ext}}(\boldsymbol{x})$ is the external potential. 
	
	\subsection{Convergence Test}
	
	The errors of numerical solutions $\left(c_{\Delta x, \Delta t}\right)_{m}, m = 1, \cdots, M,$ are computed as follows
	\begin{align}
	&\|\boldsymbol{e}_{\Delta x, \Delta t}\|_{l^{\infty}} := \max_{m, j} \max_{x \in C_j}\left|\left(c_{\Delta x, \Delta t}\right)_{m}(x) - c^{\text{ref}}_m(x)\right|, \\
	&\|\boldsymbol{e}_{\Delta x, \Delta t}\|_{l^p} := \left( \sum_{m = 1}^{M} \sum_{j} \int_{C_j} \left|\left(c_{\Delta x, \Delta t}\right)_{m}(x) - c^{\text{ref}}_m(x)\right|^p \,\mathrm{d} x \right)^{\frac{1}{p}}, ~p = 1, 2, 
	\end{align}
	where $\Delta x$ is the spatial mesh size, $\Delta t$ is the time step and $c^{\text{ref}}_m$ is the reference solution of the $m$-th species which is obtained by small mesh size $\Delta x$ and small time step $\Delta t$.
	
	\subsubsection{Convergence Test for One-dimensional Problems}
	Consider the system (\ref{model2a}) for complex ionic fluids in one-dimension with the singular kernel $\mathcal{W}(x) = \frac{1}{|x|^{1/2}}$, the kernel $\mathcal{K}(x) = \exp(-|x|)$ and the external potential  $V_{\text{ext}}(x) = \frac{1}{2} x^2$. Note that, the electrostatic kernel $\mathcal{K}(x)$ in one-dimension is not physically relevant, and thus this numerical example is a toy model, which only serves the purpose of stability and convergence tests.  The initial conditions are taken as
	\begin{equation}
	\label{inc}
	\left\{
	\begin{array}{lll}
	c_1(x, 0) = \frac{1}{2 \sqrt{2 \pi}} \exp \left(- \frac{1}{2} (x - 2)^2 \right) &\text{with} &z_1 = 1, \\
	c_2(x, 0) = \frac{1}{\sqrt{2 \pi}} \exp \left(-\frac{1}{2} (x + 2)^2 \right) &\text{with} &z_2 = -1.
	\end{array}
	\right.
	\end{equation}
	We verify the convergence order of our scheme in both space and time. Here, we take the computation domain as $[-L, L], \ L = 10$, then the results of the second order convergence in space of error $\boldsymbol{e}_{\Delta x, \Delta t}$ in $l^{\infty}, \ l^1$ and $l^2$ norms at time $t = 0.1$ are shown in Fig \ref{con}. Left, where we take the uniform mesh size $\Delta x_j = \Delta x_0 2^{-j}, j = 0, 1, 2, 3, 4$ with $\Delta x_0 = 0.6250$, 
	and $\Delta t$ = $10^{-5}$. 
	In this test, the solution on mesh with mesh size $\Delta x = \Delta x_0 2^{-6}$
	is taken as the reference solution. 
	
	Next, in terms of the convergence order in time, we show the numerical errors at time $t = 1$ in Fig \ref{con}. Right, where the numerical solutions are computed with the time step $\Delta t_j = \Delta t_0 2^{-j}, j = 0, 1, 2, 3, 4, 5, 6, 7, 8, 9, 10$ respectively with $\Delta t_0 = 0.1024$,
    and $\Delta x = 0.02$.
	Here, the solution with $\Delta x = 0.02$ and $\Delta t = 10^{-5}$ is taken as the reference solution. From Fig \ref{con}. Right we can conclude that our scheme is of first order in time and has improved stability performance. In particular, such a scheme is free of CFL conditions and the numerical evidence strongly suggests that it is unconditionally stable.
	
	\begin{figure}[htp] 
		\subfigure{
			\begin{minipage}[t]{0.5\linewidth}
				\centering
				\includegraphics[width=1.0\linewidth]{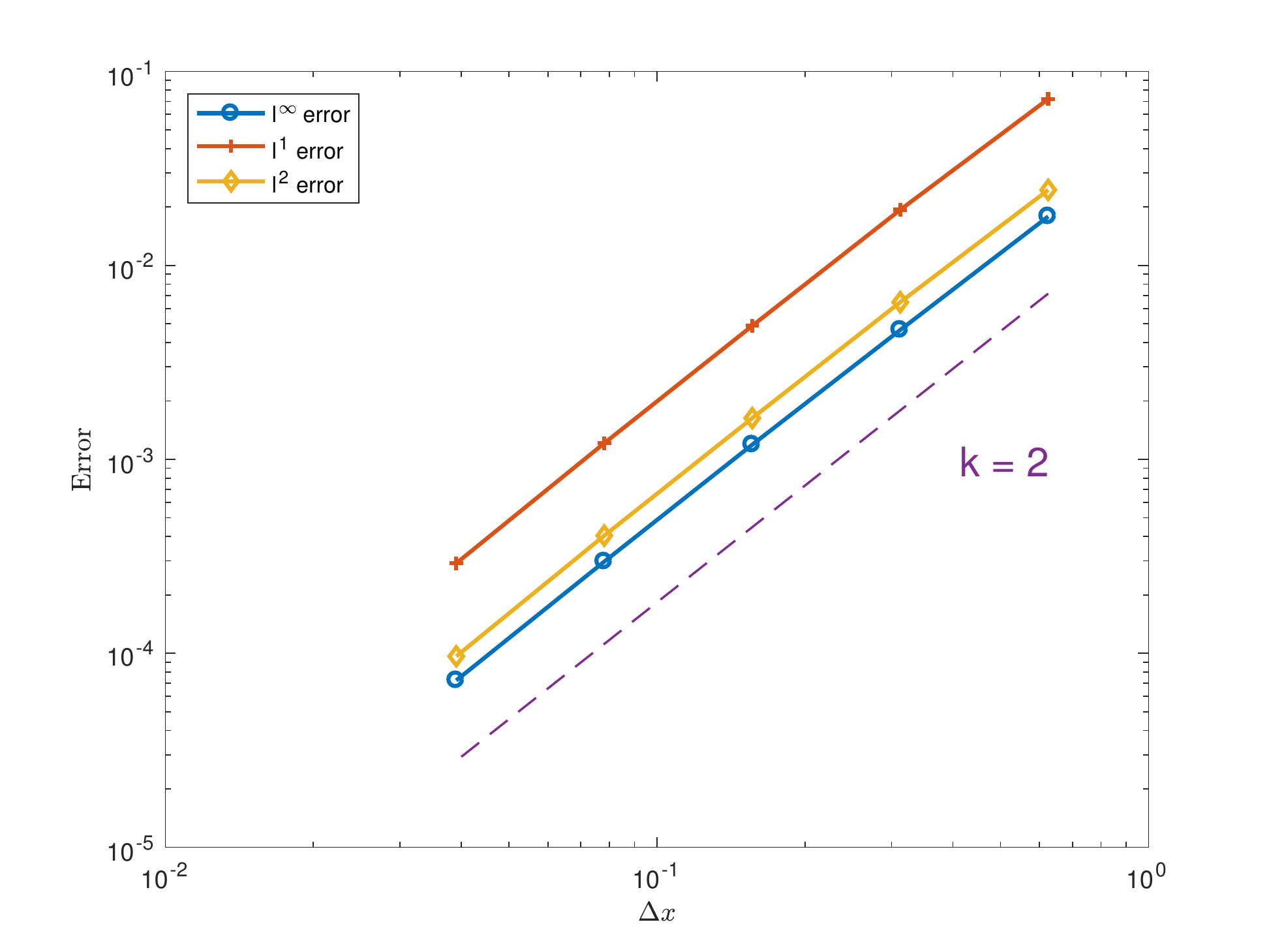}
			\end{minipage}
			\begin{minipage}[t]{0.5\linewidth}
				\centering
				\includegraphics[width=1.0\linewidth]{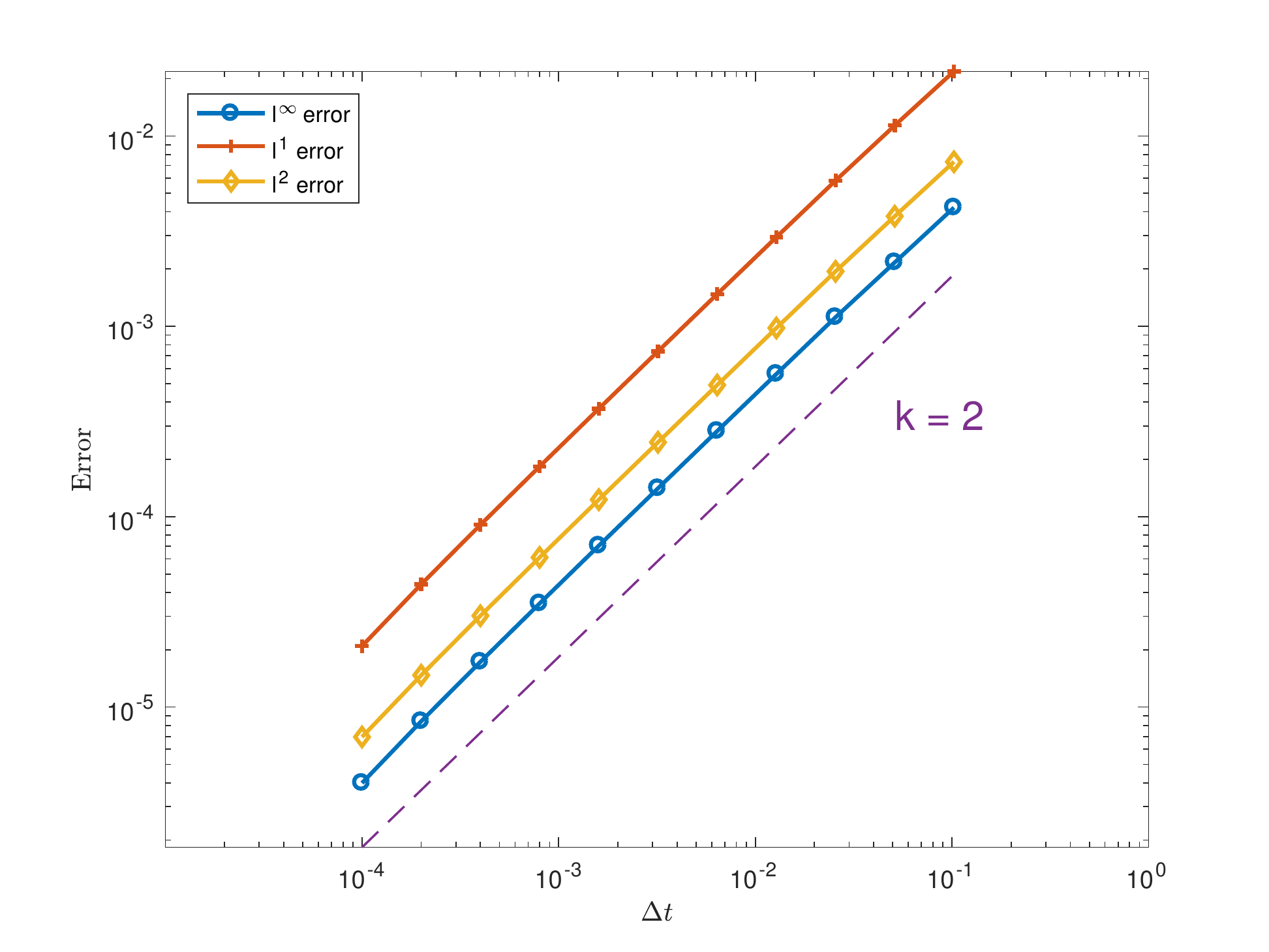}
			\end{minipage}
		}
		\centering
		\caption{Left: {\bf Second order convergence in space in 1D.} The loglog plot of errors with the mesh size $\Delta x_j = \Delta x_0 2^{-j}, j = 0, 1, 2, 3, 4$ with $\Delta x_0 = 0.6250$, at time $t = 0.1$ ($\Delta t = 10^{-5}$). The reference solution is computed on the grid $\Delta x = \Delta x_0 2^{-6}$ and $\Delta t = 10^{-5}$. Right: {\bf First order convergence in time in 1D and unconditionally stability.} The loglog plot of errors with the time step $\Delta t_j = \Delta t_0 2^{-j}, j = 0, 1, 2, 3, 4, 5, 6, 7, 8, 9, 10$ with $\Delta t_0 = 0.1024$, at time $t = 1$ ($\Delta x = 0.02$). The reference solution is computed on the grid $\Delta x = 0.02$ and $\Delta t = 10^{-5}$.}
		\label{con}
	\end{figure}
	
	\subsubsection{Comparison with Regularized Kernels}
	In order to demonstrate with necessity of constructing a numerical scheme which is compatible with singular kernels, we compare the numerical performance of the system with a singular kernel and its counterpart with a regularized kernel. Consider the system (\ref{model2a}) with $\mathcal{K}(x) = \exp(-|x|)$, $\mathcal{W}(x) = \frac{1}{|x|^{1/2}}$, $V_{\text{ext}}(x) = 10 x^2$ and the initial conditions are given by 
	\begin{equation}
	\label{iniex1}
	\left\{
	\begin{array}{lll}
	c_1(x, 0) = \frac{1}{2 \sqrt{2 \pi}} \exp \left(-20 \left(x - \frac{1}{5}\right)^2 \right) &\text{with} &z_1 = 1, \\
	c_2(x, 0) = \frac{1}{\sqrt{2 \pi}} \exp \left(-20 \left(x + \frac{1}{5}\right)^2 \right) &\text{with} &z_2 = -1.
	\end{array}
	\right.
	\end{equation}
	
	In this test, we aim to investigate the discrepancy of the solutions if the singular kernel $\mathcal{W}(x) = \frac{1}{|x|^{1/2}}$ is replaced by the corresponding regularized kernel $\mathcal{W}^{\epsilon}(x) = \frac{1}{|x|^{1/2} + \epsilon}$ while other conditions, including the kernel $\mathcal{K}$, remain the same. 
	
	For the problem with regularized kernel $\mathcal{W}^{\epsilon}(x) = \frac{1}{|x|^{1/2} + \epsilon}$, we adopt the same numerical scheme as (\ref{bsys2})-(\ref{bsys3}) while the only difference is that, due to the non-singularity, the numerical convolution of the approximate regularized kernel $\mathcal{W}^{\epsilon}(x) = \frac{1}{|x|^{1/2} + \epsilon}$ is computed as	
	\begin{equation}
	\label{sys4r}
	f_{m, j}(t) = \Delta x \sum_{i = -N_x}^{N_x} \left( z_m \mathcal{K}_{j-i} \rho_{i} + \mathcal{W}^{\epsilon}_{j-i} \theta_{i} \right),
	\end{equation}
	where the discrete kernel $\mathcal{K}_{j - i} = \mathcal{K}(x_j - x_i)$ and $\mathcal{W}^{\epsilon}_{j - i} = \mathcal{W}^{\epsilon}(x_j - x_i)$.
	Moreover, the discrete total charge density $\rho_{j}$ and the discrete total mass density $\theta_{j}$ are denoted respectively by
	\begin{align}
	\rho_{j} = \sum_{m = 1}^M z_m \bar{c}_{m, j}, ~~
	\theta_{j} = \sum_{m = 1}^M \bar{c}_{m, j}.
	\end{align}

	Here, we observe the discrepancy between the solution with singular kernel $\mathcal{W}$ solved by the scheme (\ref{sysodeb}), (\ref{sysfluxb}) and the solution with the regularized kernel $\mathcal{W}^{\epsilon}$ solved by the scheme (\ref{bsys2}), (\ref{bsys3}), (\ref{sys4r}) with the same time discretization as in (\ref{sysodeb}), (\ref{sysfluxb}).
	And we aim to observe that how the discrepancy changes with different $\epsilon$ and different grid numbers $N_x$.
		
	In this test, we take $\epsilon = \frac{1}{2}, \frac{1}{8}, \frac{1}{32}$, the computation domain as $[-L, L], \ L = 1$, then the results of discrepancy in $l^{\infty}$ norms at time $t = 0.5$ are shown in Fig \ref{im1} where we take the number of the grid points $N_x = 2^{3+j} , j = 0, 1, 2, 3, 4, 5, 6, 7, 8, 9$ (i.e. the uniform mesh size $\Delta x_j = \Delta x_0 2^{-j}$ with $\Delta x_0 = 0.25$)
	and $\Delta t = 0.001$. 
	
	From Fig \ref{im1}, we learn that for a given $\epsilon$, if the mesh size $\Delta x$ is relatively large compared with the regularization parameter $\epsilon$, i.e. the spatial mesh fails to resolve the regularized kernel function, the discrepancy is dominated by the discretization error. In particular, when the regularized kernel is not well resolved, the discrepancy may be larger for smaller $\epsilon$, which seems to suggest that, the regularization parameter gives additional constraint in the spatial mesh in order to accurately capture the solution behavior of the original system. As a consequence, when $\Delta x$ is fixed, smaller $\epsilon$ may not lead to smaller discrepancies, which is the case when  $\Delta x =0.25, 0.125,  0.0625$ shown in Fig \ref{im1}.
	
	On the other hand, if the mesh size $\Delta x$ is small enough, the discrepancy will be dominated by the approximation error introduced by the regularization and thus solutions with regularized kernel can be a good approximation to that with the corresponding singular kernel as studied in \cite{Liu2019}.  Furthermore, assuming the mesh size $\Delta x$ is sufficiently small  for different $\epsilon$, respectively, smaller $\epsilon$ leads to a smaller discrepancy, which means a better approximation.
	
	This test demonstrates the significance of designing a numerical scheme which is compatible with singular kernels.  Our proposed scheme does not reply on any regularization, and thus is free of the spatial mesh constraint.
	
	\begin{figure}[htp] 
		\subfigure{
			\begin{minipage}[t]{0.5\linewidth}
				\centering
				\includegraphics[width=1.0\linewidth]{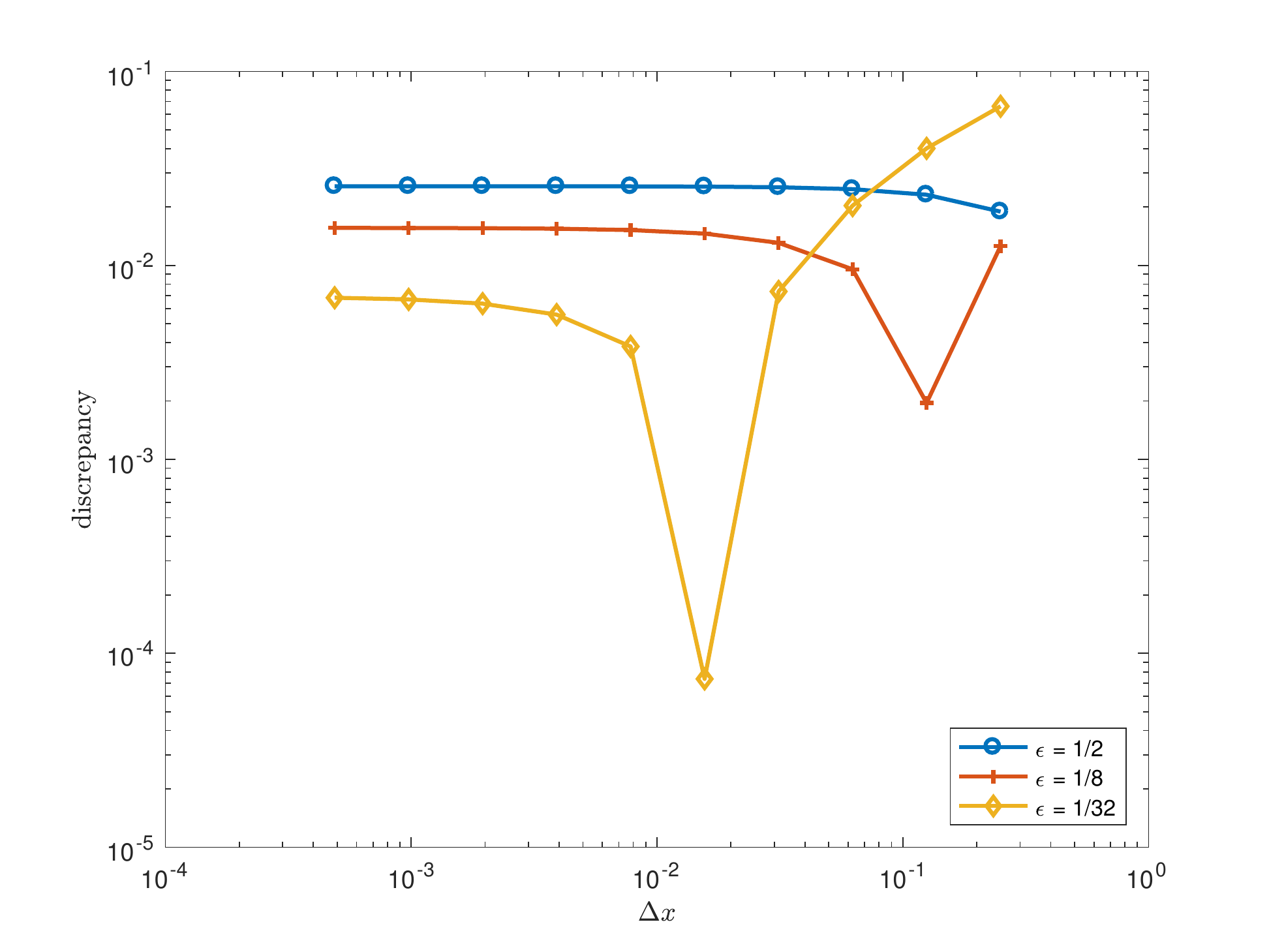}
			\end{minipage}
		}
		\centering
		\caption{ {\bf Convergence test with different $\epsilon$.} The loglog plot of the discrepancy with respect to the number of the grid points $N_x = 2^{3+j} , j = 0, 1, 2, 3, 4, 5, 6, 7, 8, 9$ (the uniform mesh size $\Delta x_j = \Delta x_0 2^{-j}$ with $\Delta x_0 = 0.25$) at time $t = 0.5$. The solution solved by (\ref{bsys4}) is considered as the reference solution. }
		\label{im1}
	\end{figure}
	
	\subsubsection{Convergence Test for Two-dimensional Problems}
	
	In this test, we verify that our scheme in 2D is of second order in space. Here we consider the two-dimensional singular kernel $\mathcal{W}(x, y) = \frac{1}{r^{3/2}}, r = \sqrt{x^2 + y^2}$, $\mathcal{K}(x, y) = - \frac{1}{2 \pi} \ln r$, the external potential $V_{\text{ext}}(x, y) = 10 r^2$ and the initial conditions  are given by 
	\begin{equation}
	\label{iniex3}
	\left\{
	\begin{array}{lll}
	c_1^0 = \frac{1}{2 \pi} \exp \left(-20\left(\left(x - \frac{1}{5}\right)^2 + \left(y - \frac{1}{5}\right)^2\right)\right) &\text{with} &z_1 = 1, \\
	c_2^0 = \frac{1}{2 \pi} \exp \left(-20\left(\left(x + \frac{1}{5}\right)^2 + \left(y + \frac{1}{5}\right)^2\right) \right) &\text{with} &z_2 = -1.
	\end{array}
	\right.
	\end{equation}
	The computation domain is $[-L, L] \times [-L, L], \ L = 1$ and let the mesh size $\Delta x_j = \Delta  x_0 2^{-j}, j = 0, 1, 2, 3$ with $\Delta x_0 = 0.04$ 
	and time step $\Delta t$ be $10^{-5}$. Fig \ref{im1b} shows the results of error $\tilde{\boldsymbol{e}}_{\Delta x, \Delta t}$ in $l^{\infty}, l^1$ and $l^2$ norms at time $t = 0.01$ where
	\begin{align}
	&\|\tilde{\boldsymbol{e}}_{\Delta x, \Delta t}\|_{l^{\infty}} := \max_{m, j} \max_{x \in C_j}\left|\left(c_{\Delta x, \Delta t}\right)_{m}(x) - \left(c_{2\Delta x, \Delta t}\right)_{m}(x)\right|, \\
	&\|\tilde{\boldsymbol{e}}_{\Delta x, \Delta t}\|_{l^p} := \left( \sum_{m = 1}^{M} \sum_{j} \int_{C_j} \left|\left(c_{\Delta x, \Delta t}\right)_{m}(x) - \left(c_{2\Delta x, \Delta t}\right)_{m}(x)\right|^p \,\mathrm{d} x \right)^{\frac{1}{p}}, ~p = 1, 2, 
	\end{align}	
	thus we can conclude that the two-dimensional scheme is of second order in space as well. 
	\begin{figure}[htp]
		\subfigure{
			\begin{minipage}[t]{0.5\linewidth}
				\centering
				\includegraphics[width=1.0\linewidth]{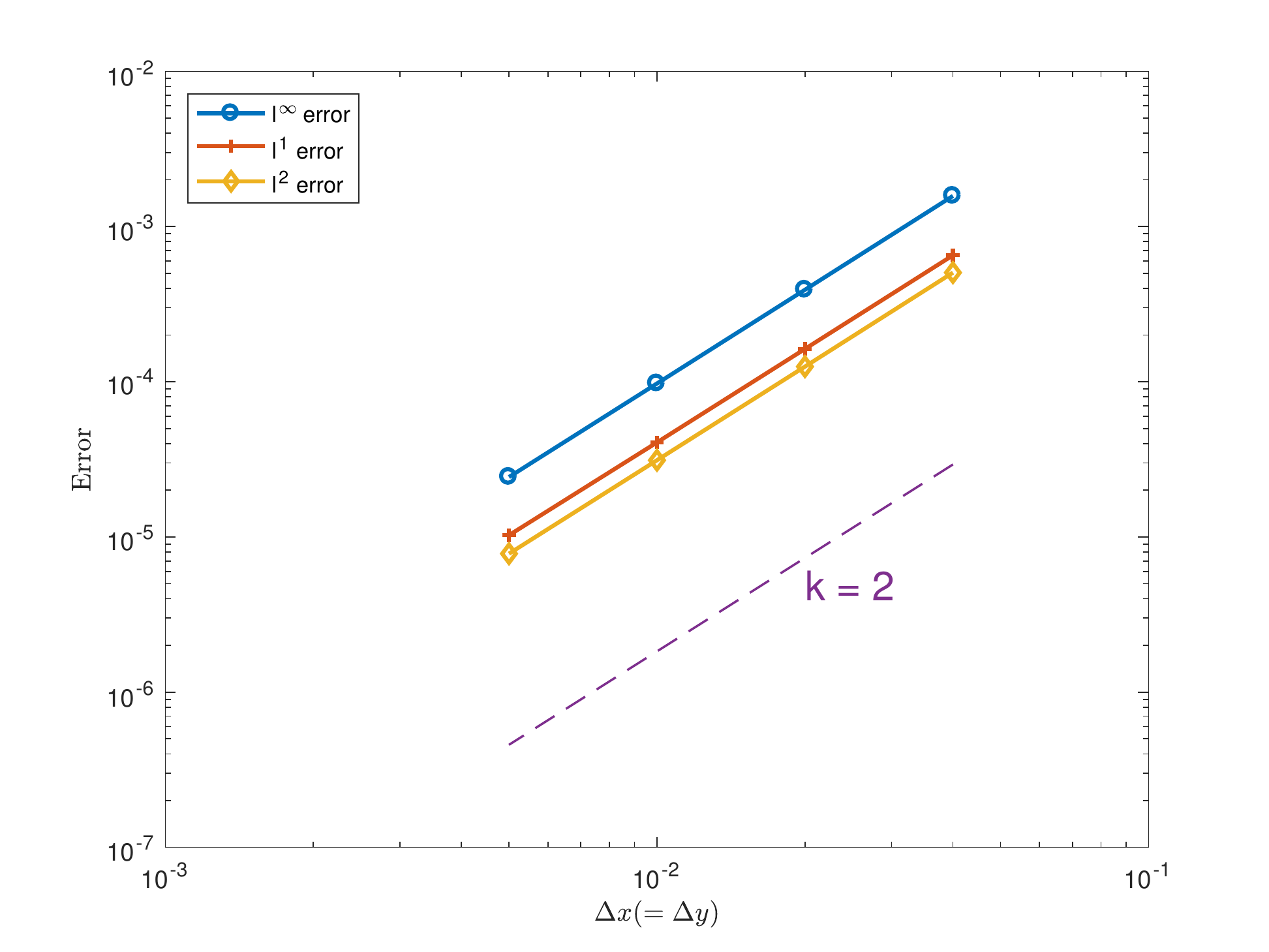}
			\end{minipage}
		}
		\centering
		\caption{{\bf Second order convergence in space in 2D.} The loglog plot of errors with mesh size $\Delta x_j = \Delta  x_0 2^{-j}, j = 0, 1, 2, 3$ with $\Delta x_0 = 0.04$ , at time $t = 0.01$ ($\Delta t = 10^{-5}$).}
		\label{im1b}
	\end{figure}
	
	\subsection{Numerical Experiments in One-dimension}
	
	Consider the equations (\ref{model2a}) in one-dimension with the singular kernel $\mathcal{W}(x) = \frac{\eta}{|x|^{1/2}}$, non-singular kernel $\mathcal{K}(x) = \exp(-|x|)$, the added external potential $V_{\text{ext}}(x) = 10 x^2$ and the initial conditions (\ref{model2c}) are given by (\ref{iniex1}). The kernel $\mathcal{W}(x)$ corresponds to the steric repulsion arising from the finite size and this model can be viewed as a modified PNP model where there are the additional nonlocal repulsion and external potential. We remark that the numerical example in this subsection is also a toy model because the electrostatic kernel $\mathcal{K}(x)$ in one-dimension is not physically relevant. Like the previous one-dimensional example, this example serves the purpose of exploring modeling phenomena, such as the finite size effect, the concentration of the ions at the boundary, etc. 
	
	\subsubsection{Steady State} 
	In this part, we study the steady state of this example. Take the parameter $\eta = 1$, the computation domain as $[-L, L], \ L = 1$ and the uniform mesh size $\Delta x = 0.001, \Delta t = 0.0001$.  Then Fig \ref{3c2} shows the transport of the ionic species: the concentrations of the positive ions and the negative ions move towards each other due to the electrostatic attraction with time $t$ and the density functions converge to the equilibrium.
	
	\begin{figure}[htp] 
		\subfigure{
			\begin{minipage}[t]{0.3\linewidth}
				\centering
				\includegraphics[width=1.0\linewidth]{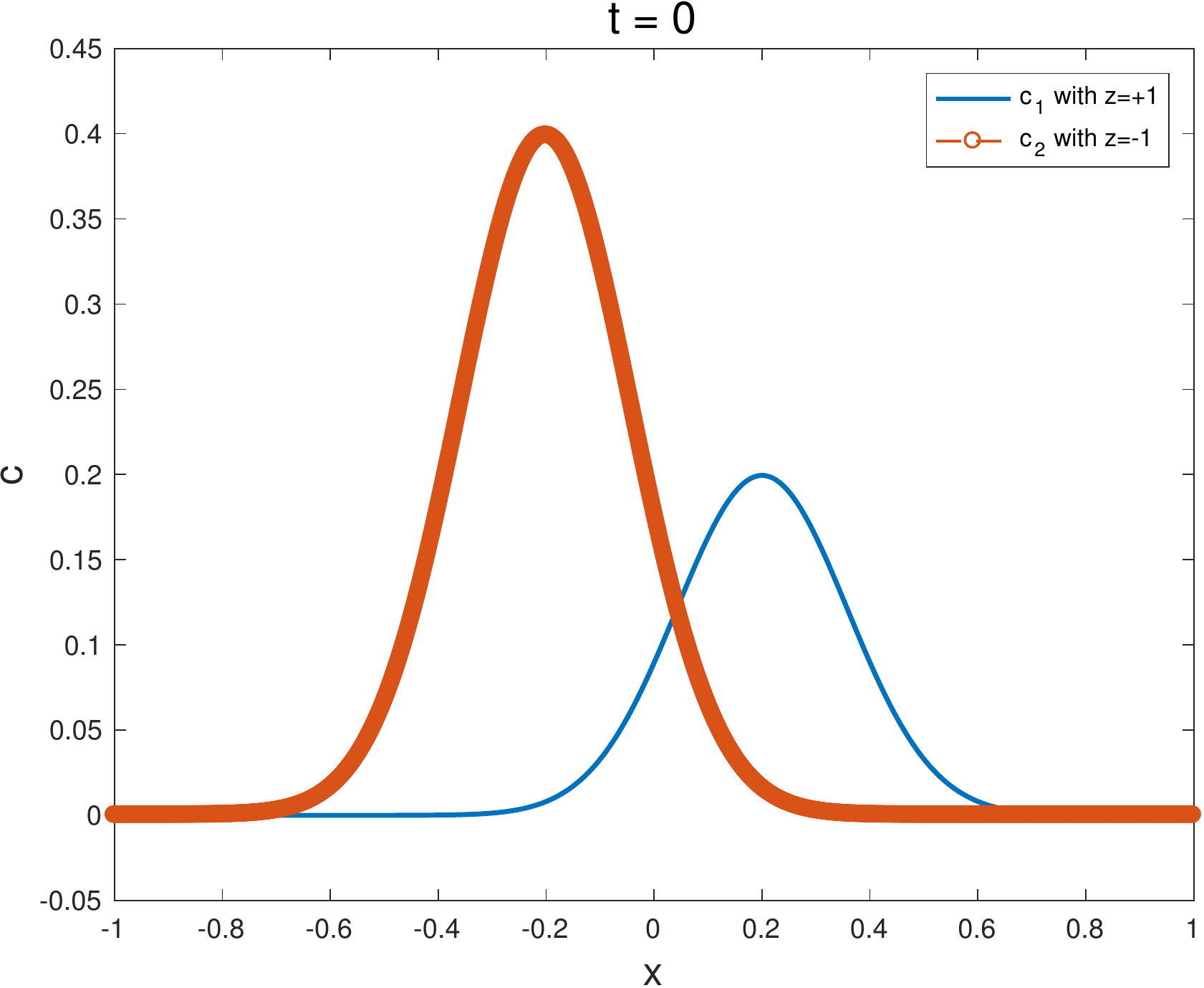}
			\end{minipage}
			\begin{minipage}[t]{0.3\linewidth}
				\centering
				\includegraphics[width=1.0\linewidth]{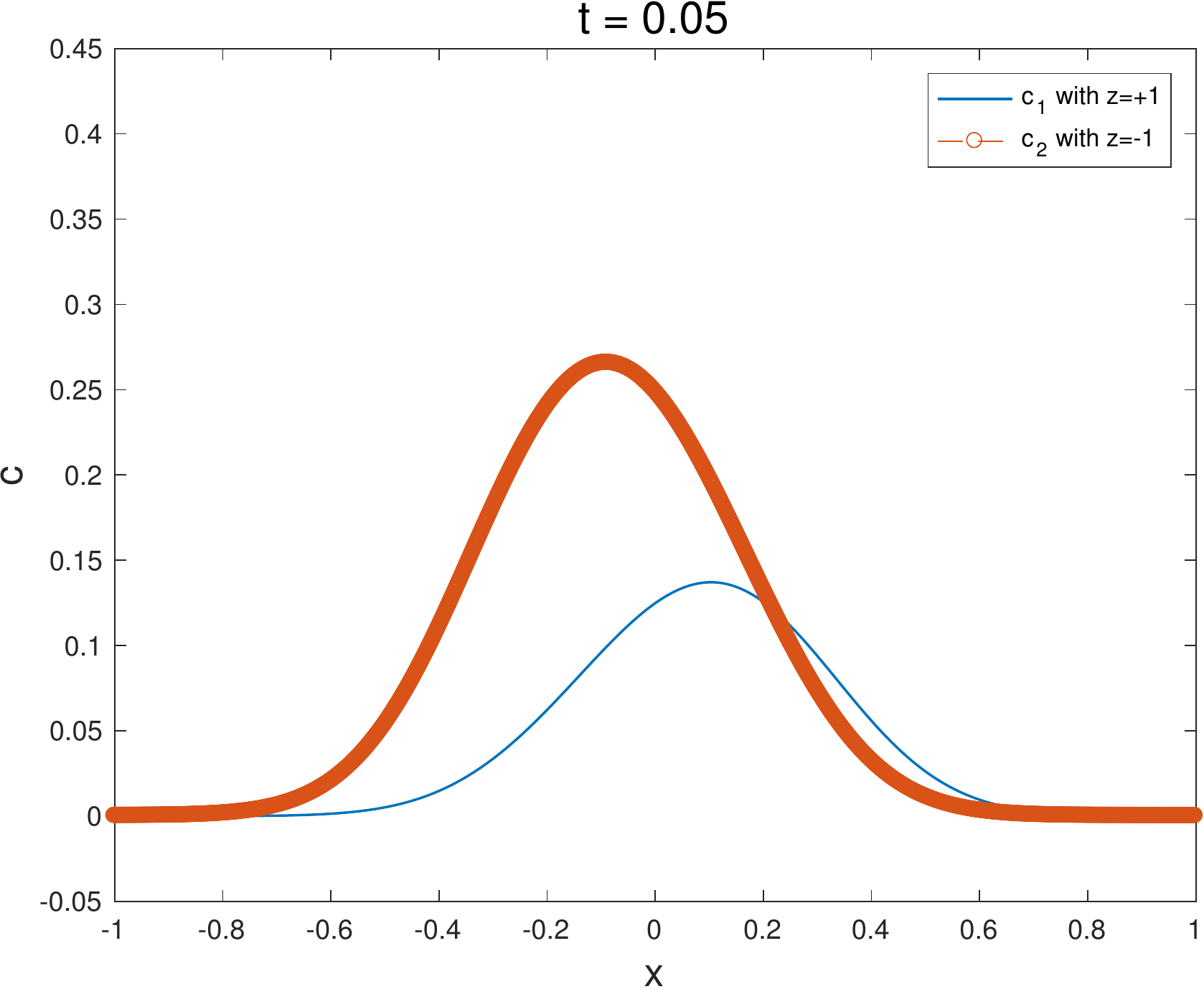}
			\end{minipage}
			\begin{minipage}[t]{0.3\linewidth}
				\centering
				\includegraphics[width=1.0\linewidth]{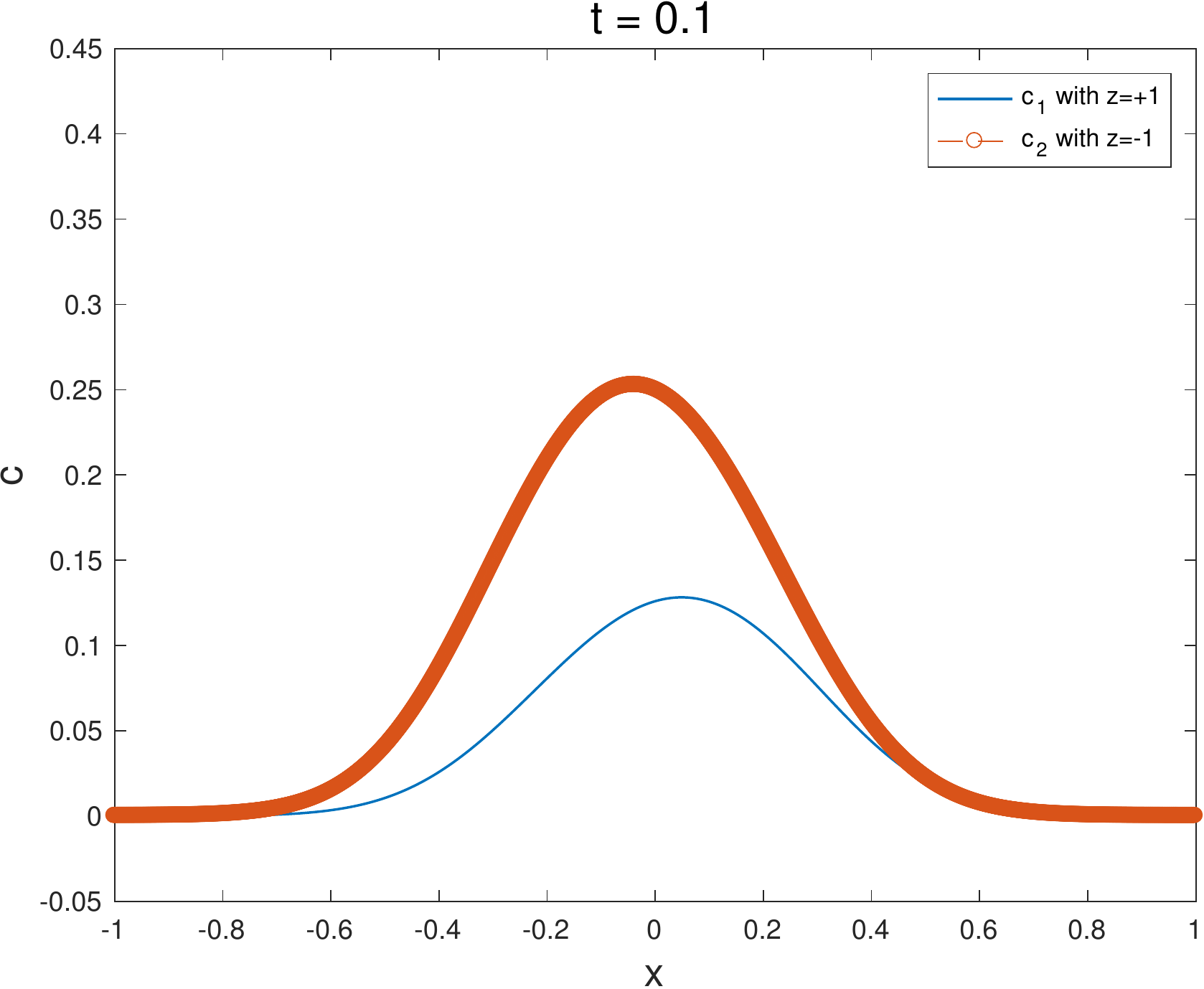}
			\end{minipage}
		}
		\subfigure{
			\begin{minipage}[t]{0.3\linewidth}
				\centering
				\includegraphics[width=1.0\linewidth]{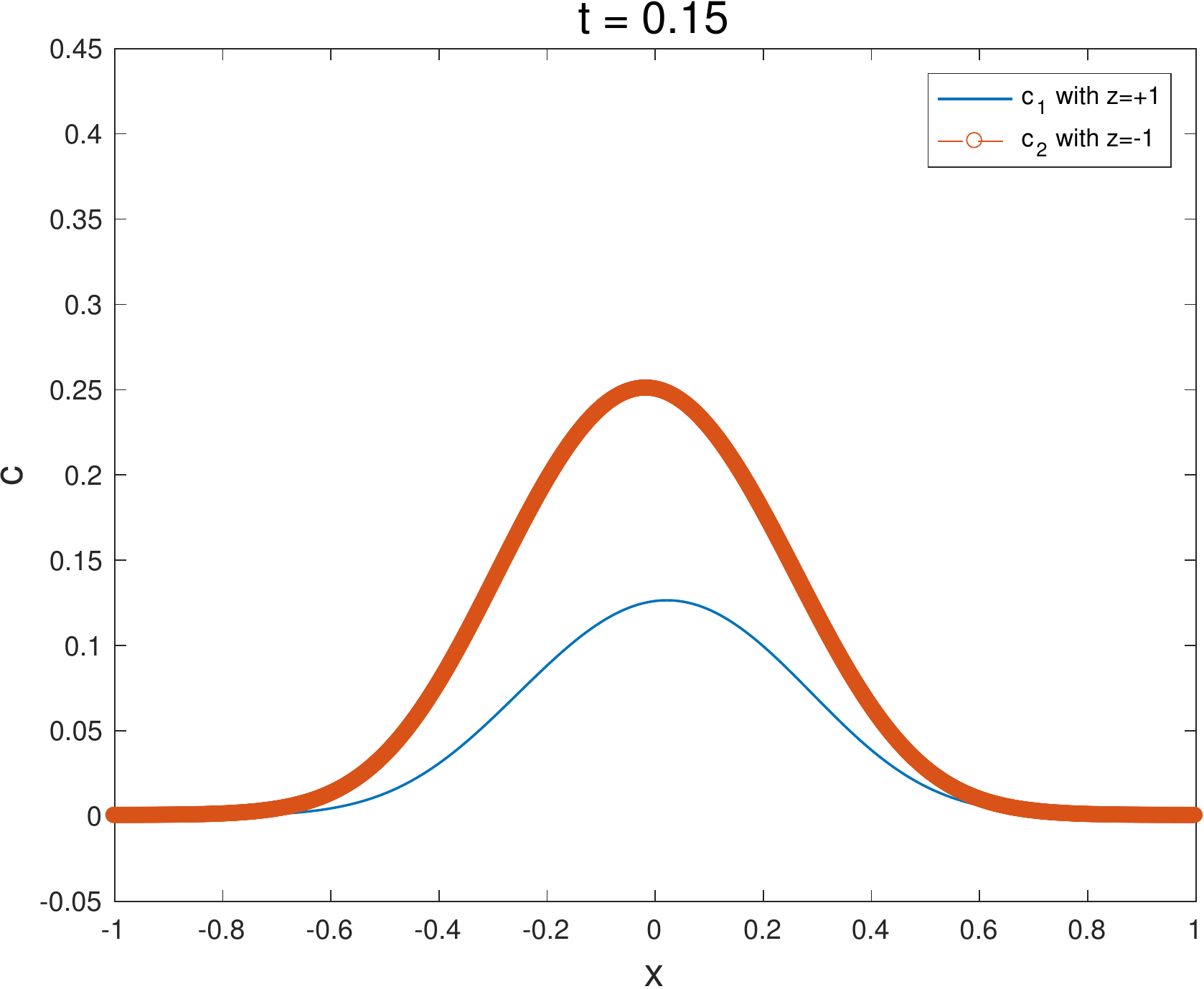}
			\end{minipage}
			\begin{minipage}[t]{0.3\linewidth}
				\centering
				\includegraphics[width=1.0\linewidth]{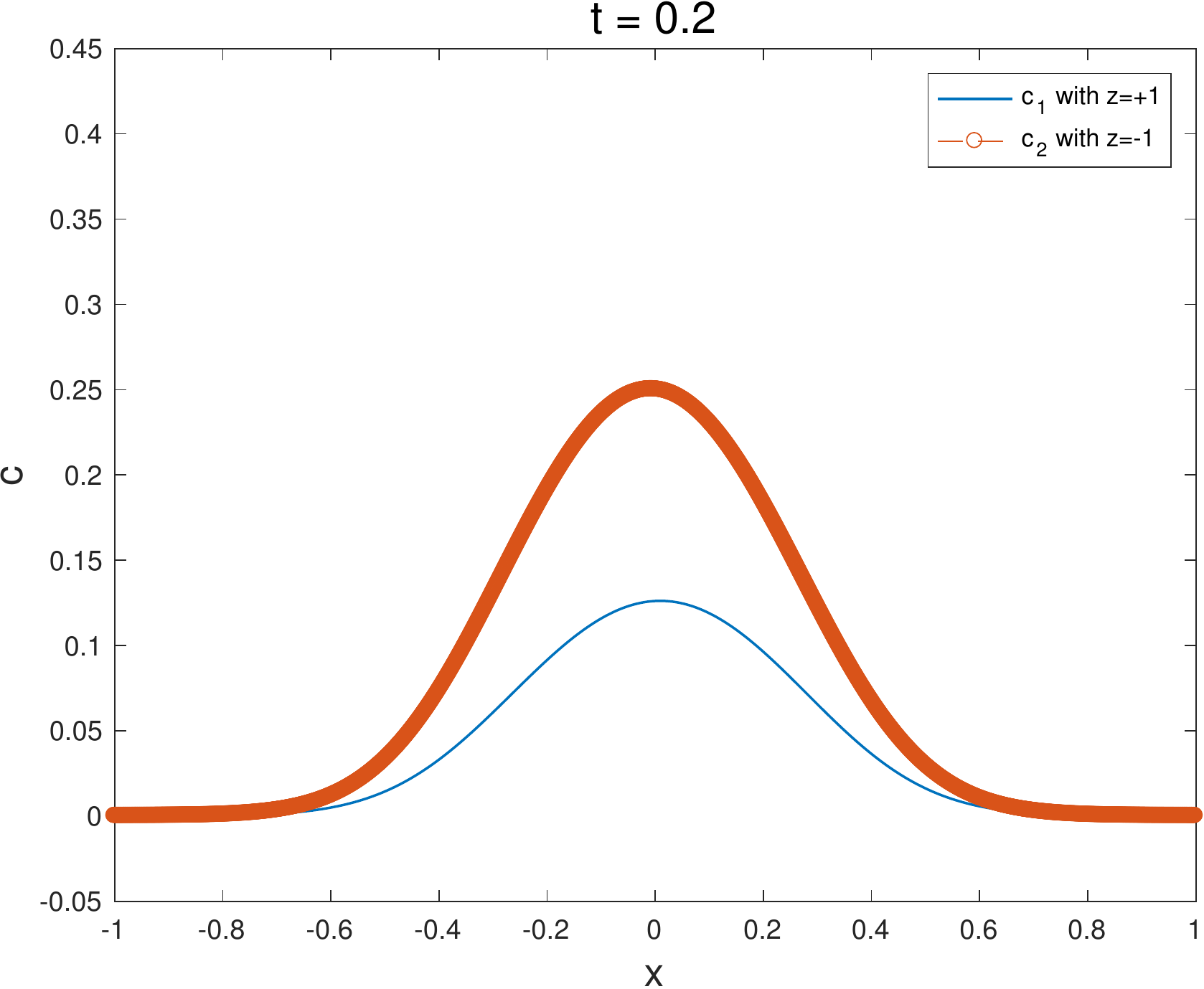}
			\end{minipage}
			\begin{minipage}[t]{0.3\linewidth}
				\centering
				\includegraphics[width=1.0\linewidth]{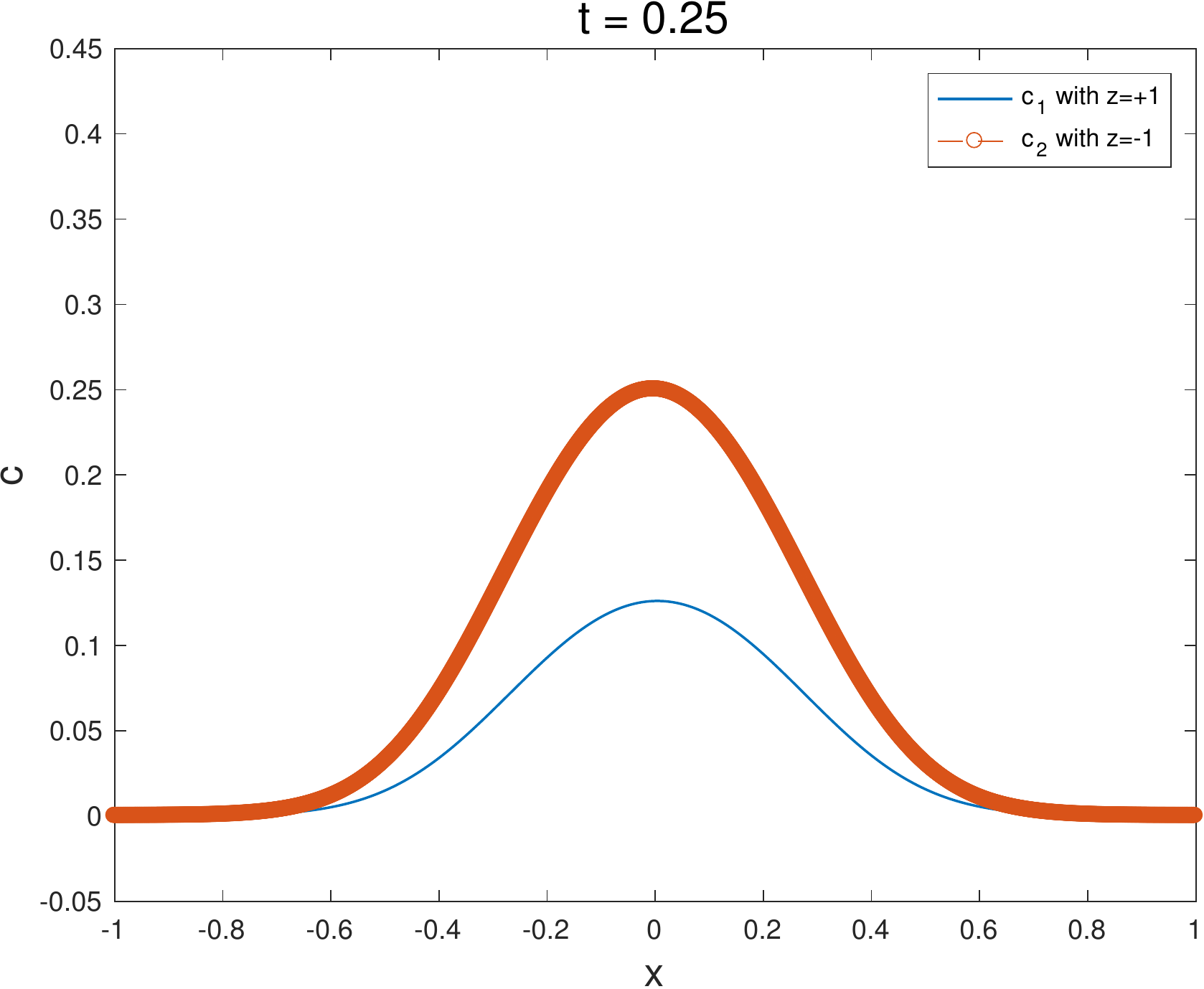}
			\end{minipage}
		}
		\subfigure{
			\begin{minipage}[t]{0.3\linewidth}
				\centering
				\includegraphics[width=1.0\linewidth]{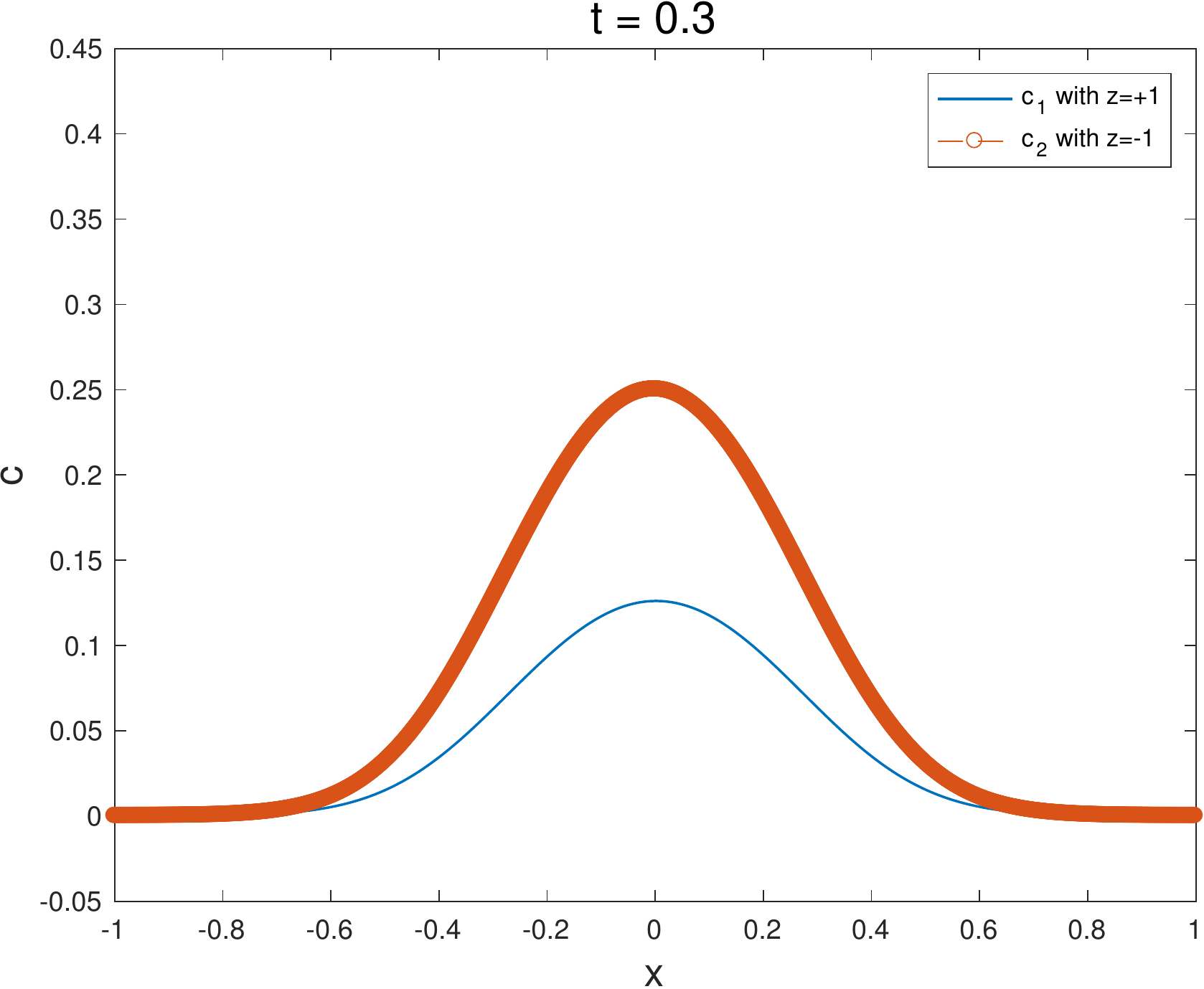}
			\end{minipage}
			\begin{minipage}[t]{0.3\linewidth}
				\centering
				\includegraphics[width=1.0\linewidth]{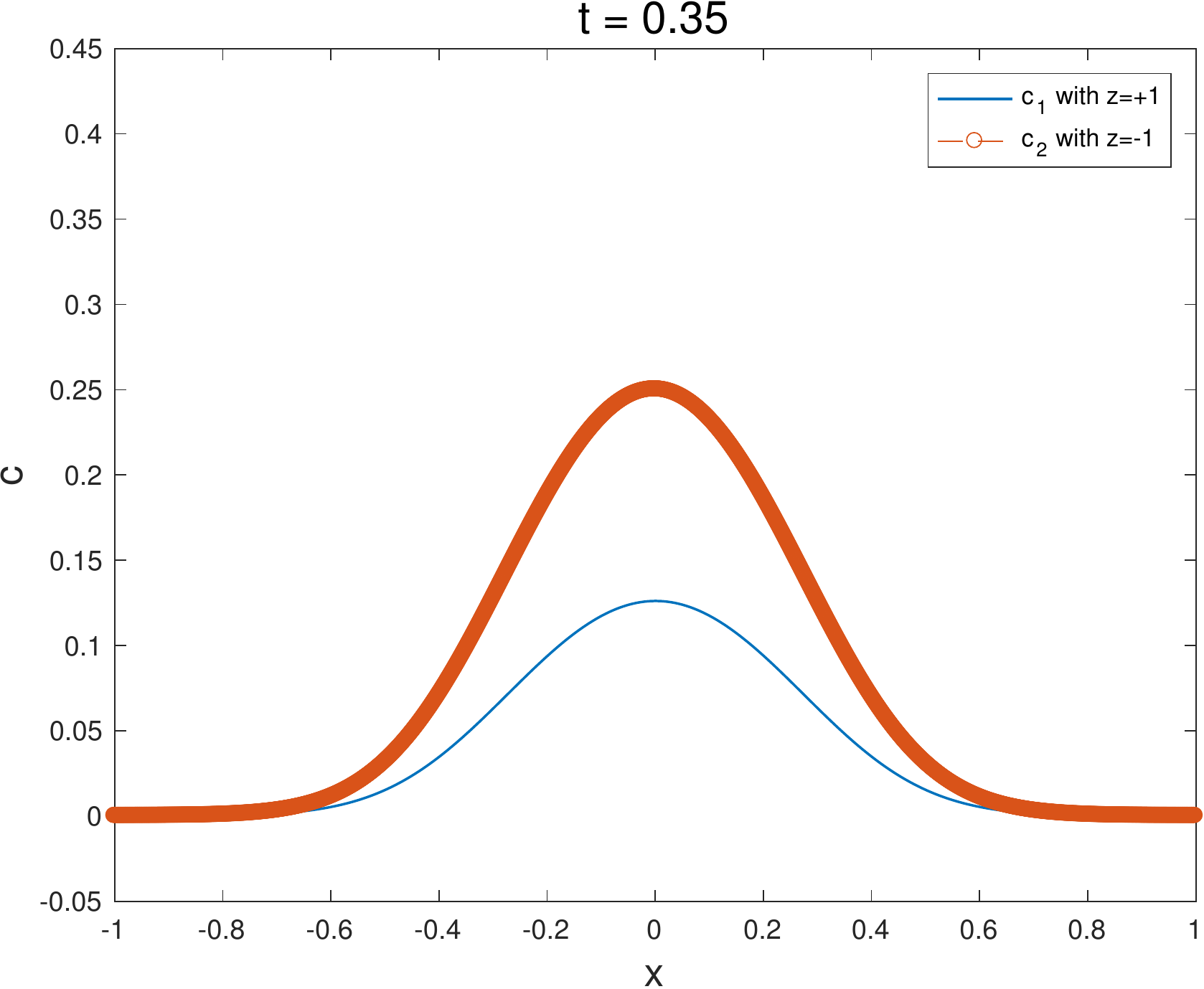}
			\end{minipage}
			\begin{minipage}[t]{0.3\linewidth}
				\centering
				\includegraphics[width=1.0\linewidth]{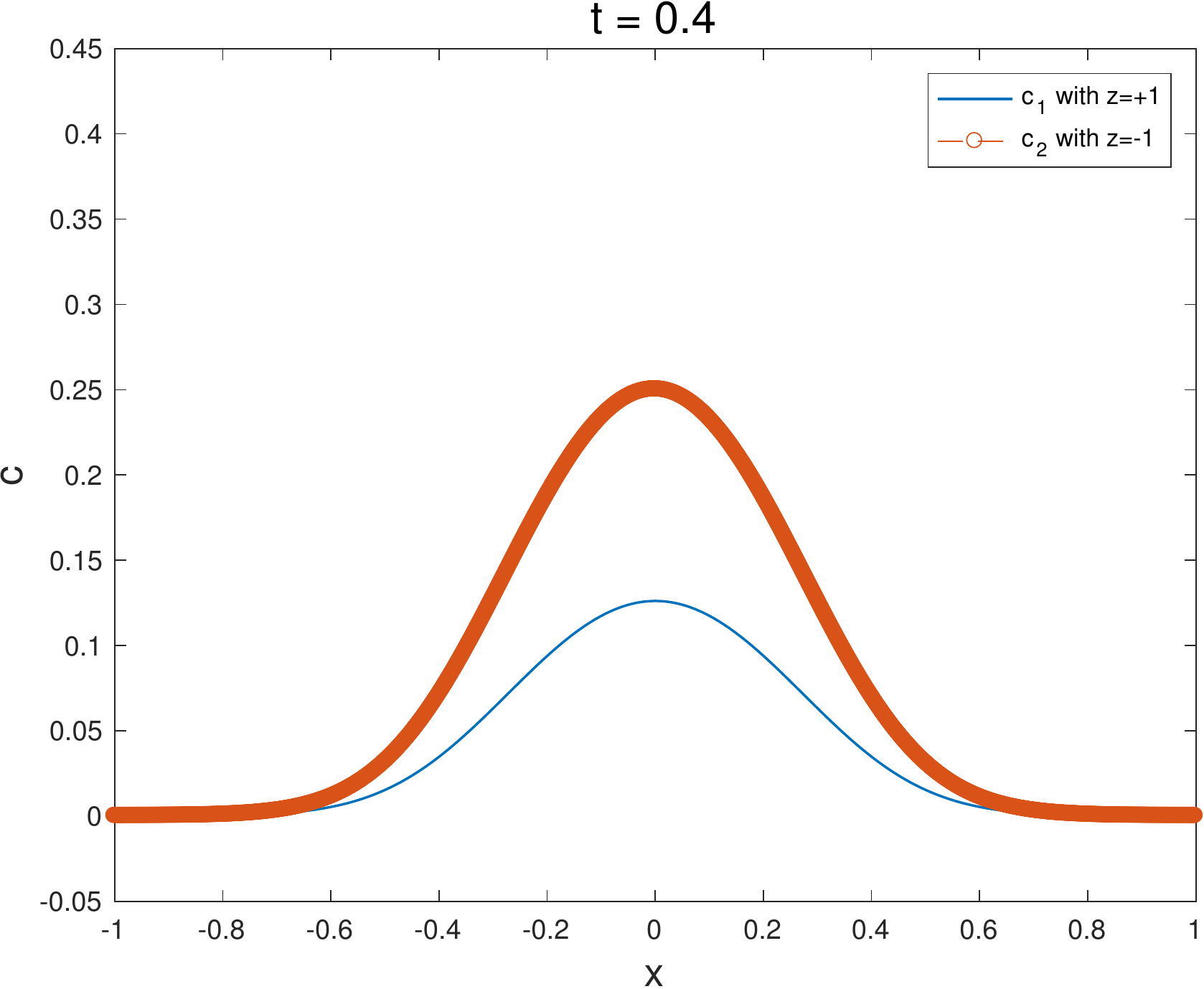}
			\end{minipage}
		}
		\centering
		\caption{Multiple Species Example in One-dimension: The space-concentration curves with the mesh size $\Delta x$ being 0.001 and $\Delta t$ being 0.0001 and time $t$ changing from 0 to 0.4.}
		\label{3c2}
	\end{figure}
	
	Additionally, Fig \ref{im2}.Left shows how the discrete form of the energy $\mathcal{F}$ defined in (\ref{disenergyb}) changes with time $t$ and 
	Fig \ref{im2}.Right shows the chemical potential $\mu_m, \ m = 1, 2,$ at time $t = 0.4$. It's observed that the discrete free energy decays and the chemical potential $\mu_m$ goes to a constant while the field model goes to the equilibrium for all $m$. The results are consistent with the analytical results in \cite{Liu2019}.
	
	\begin{figure}[htp] 
		\subfigure{
			\begin{minipage}[t]{0.5\linewidth}
				\centering
				\includegraphics[width=1.0\linewidth]{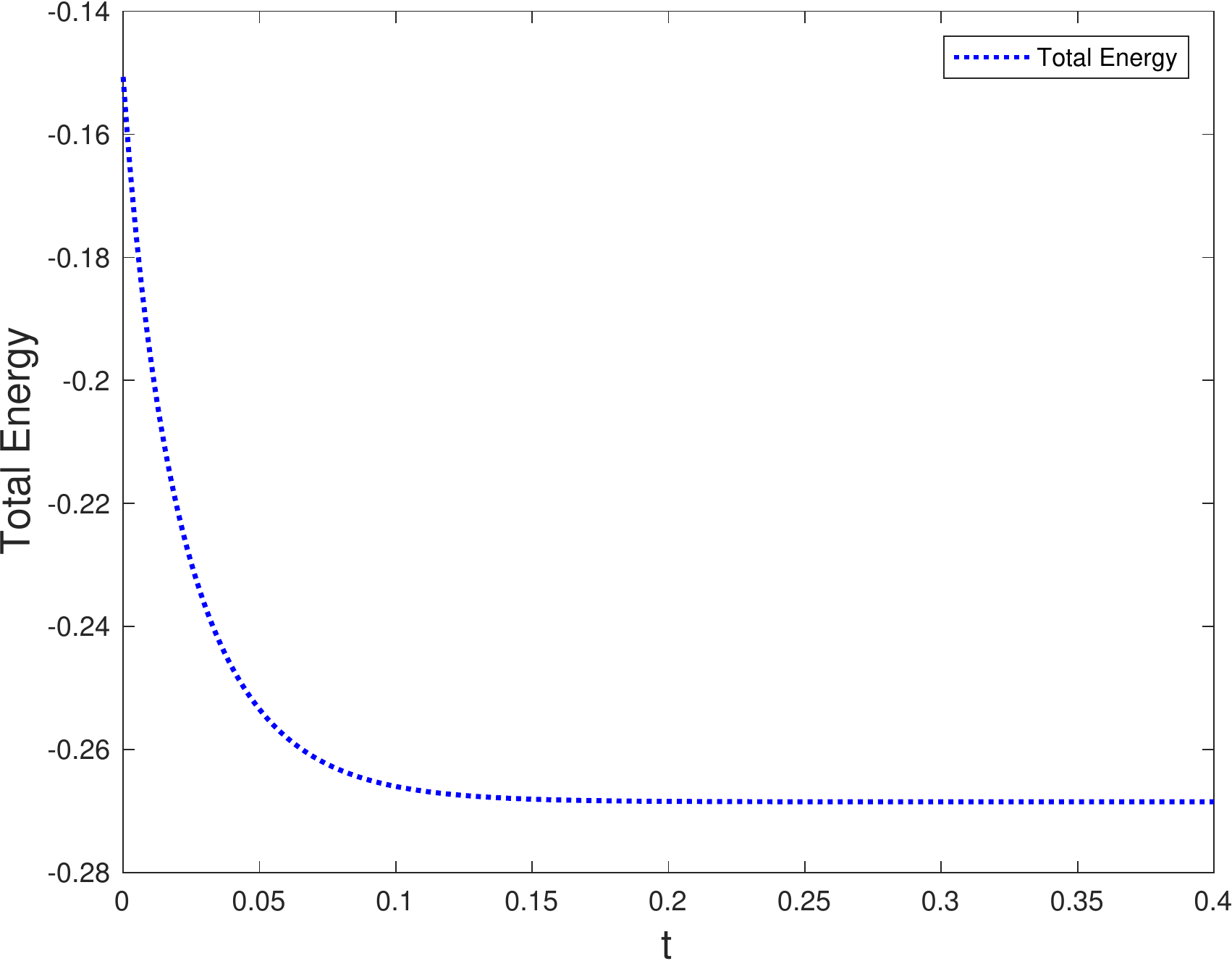}
			\end{minipage}
			\begin{minipage}[t]{0.5\linewidth}
				\centering
				\includegraphics[width=1.0\linewidth]{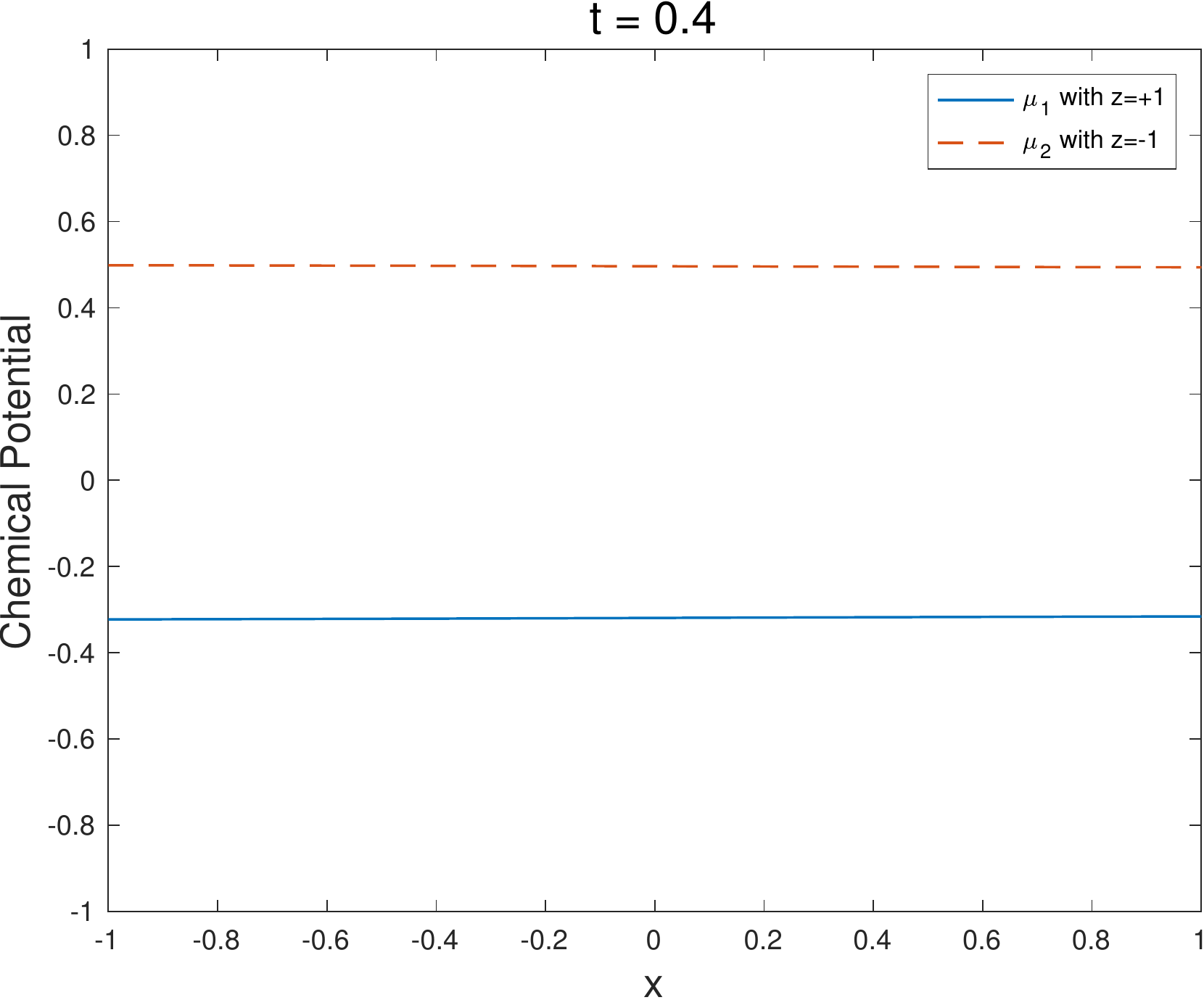}
			\end{minipage}
		}
		\centering
		\caption{Multiple Species Example in One-dimension: Left: {\bf Energy decay.} The time-energy plot of the field model (\ref{model2a}) equiped with the initial conditions (\ref{iniex1}) with the mesh size $\Delta x$ being 0.001 and  $\Delta t$ being 0.0001. Right: {\bf Constant chemical potential.} Discrete chemical potential at time $t = 0.4$ with the mesh size $\Delta x$ being 0.001 and  $\Delta t$ being 0.0001.}
		\label{im2}
	\end{figure}
	
	\subsubsection{Finite Size Effect} As we mentioned in the beginning of the paper, the kernel $\mathcal{W}(x)$ in the model (\ref{model2a})-(\ref{model2c}) represents the steric repulsion arising from the finite size and the corresponding potential function $\Phi_{\mathcal{W}}(x)$ is taken as $\Phi_{\mathcal{W}}(x) = \left(\mathcal{W} * \theta\right)(x)$. Notice that the strength of the kernel $\mathcal{W}(x)$ is indicated by the parameter $\eta$, which means the larger $\eta$ is, the stronger the nonlocal steric repulsion effect is, and thus the less peaked the concentrations of the steady state are. And $\eta = 0$ means steric repulsion vanishes. Here we aim to explore this phenomenon by different values of the parameter $\eta$.
	Let $\eta = 16, 4, 1, \frac{1}{16}, \frac{1}{128}, 0$, the mesh size $\Delta x = 0.0025$ and $\Delta t = 0.001$, Fig \ref{eta2} shows different steady state solutions with different $\eta$, where we can find that the finite size effect ($\eta \ne 0$) makes the concentrations $c_m, \ m = 1, 2,$ not overly peaked and we can verify that the nonlocal field induced by $\mathcal{W}$ effectively captures the steric repulsion arising from the finite size of the particles.  The numerical result is consistent with that in \cite{liu2010, Liu2019}.  
	
	\begin{figure}[htp] 
		\subfigure{
			\begin{minipage}[t]{0.3\linewidth}
				\centering
				\includegraphics[width=1.0\linewidth]{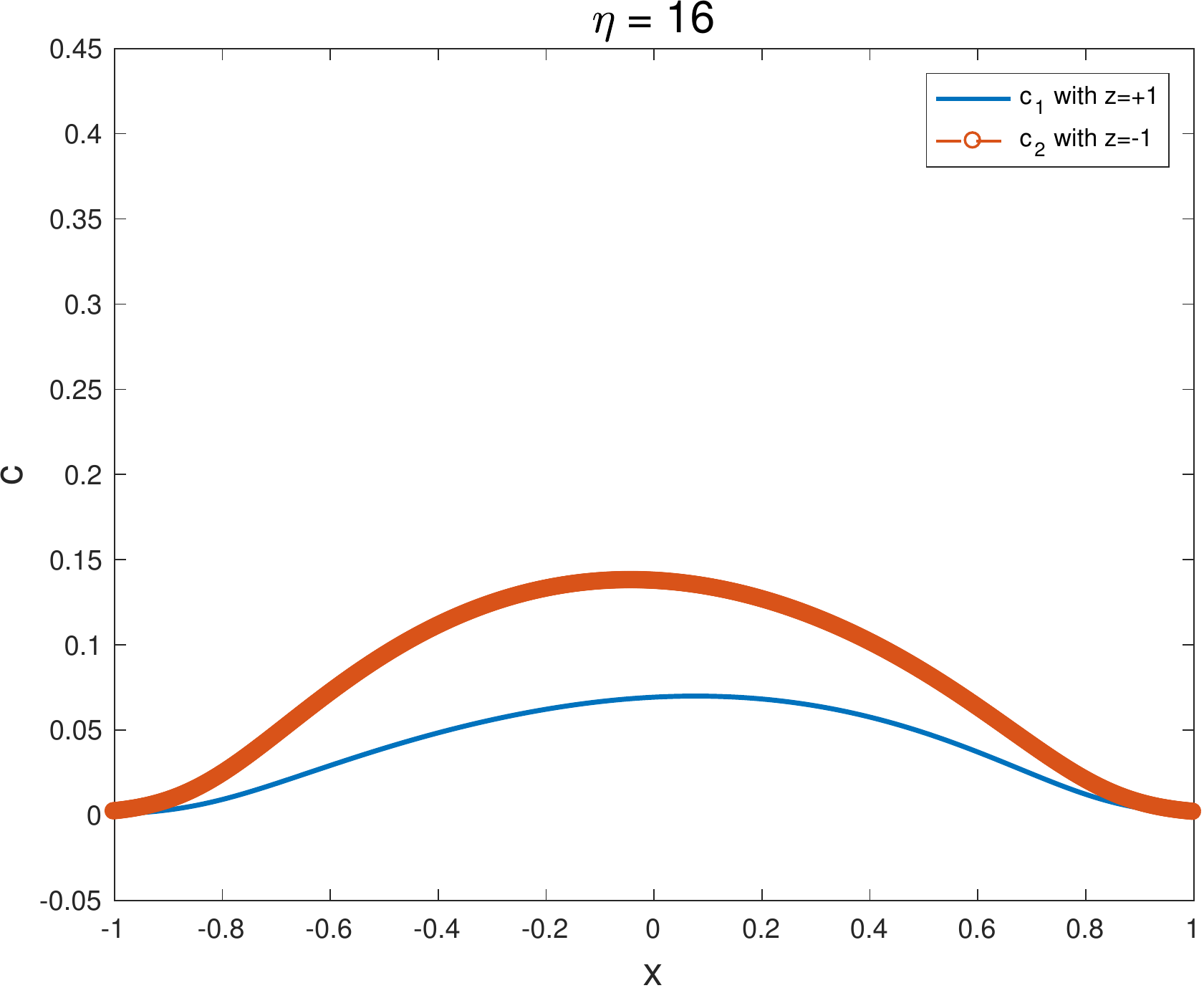}
			\end{minipage}
			\begin{minipage}[t]{0.3\linewidth}
				\centering
				\includegraphics[width=1.0\linewidth]{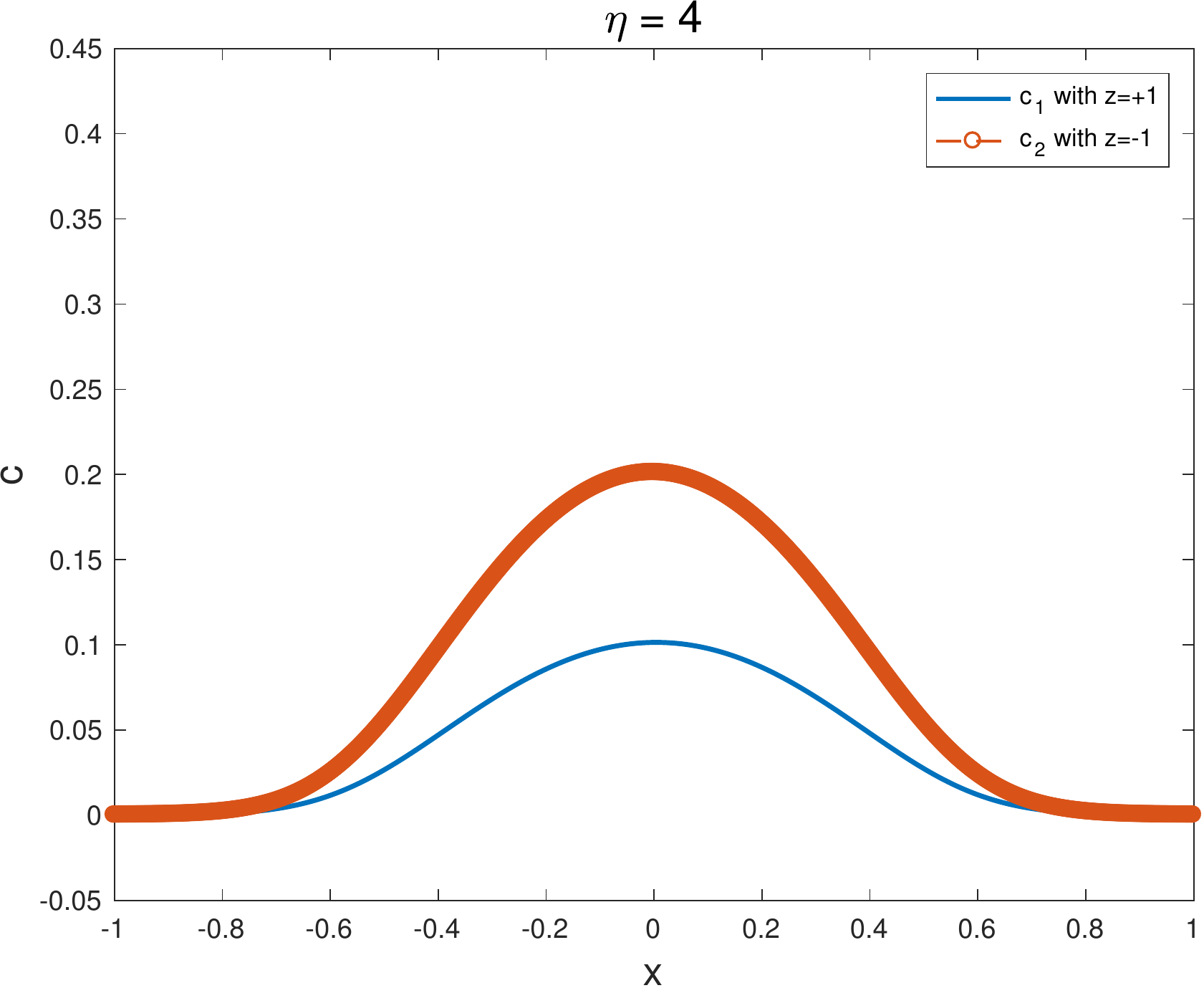}
			\end{minipage}
			\begin{minipage}[t]{0.3\linewidth}
				\centering
				\includegraphics[width=1.0\linewidth]{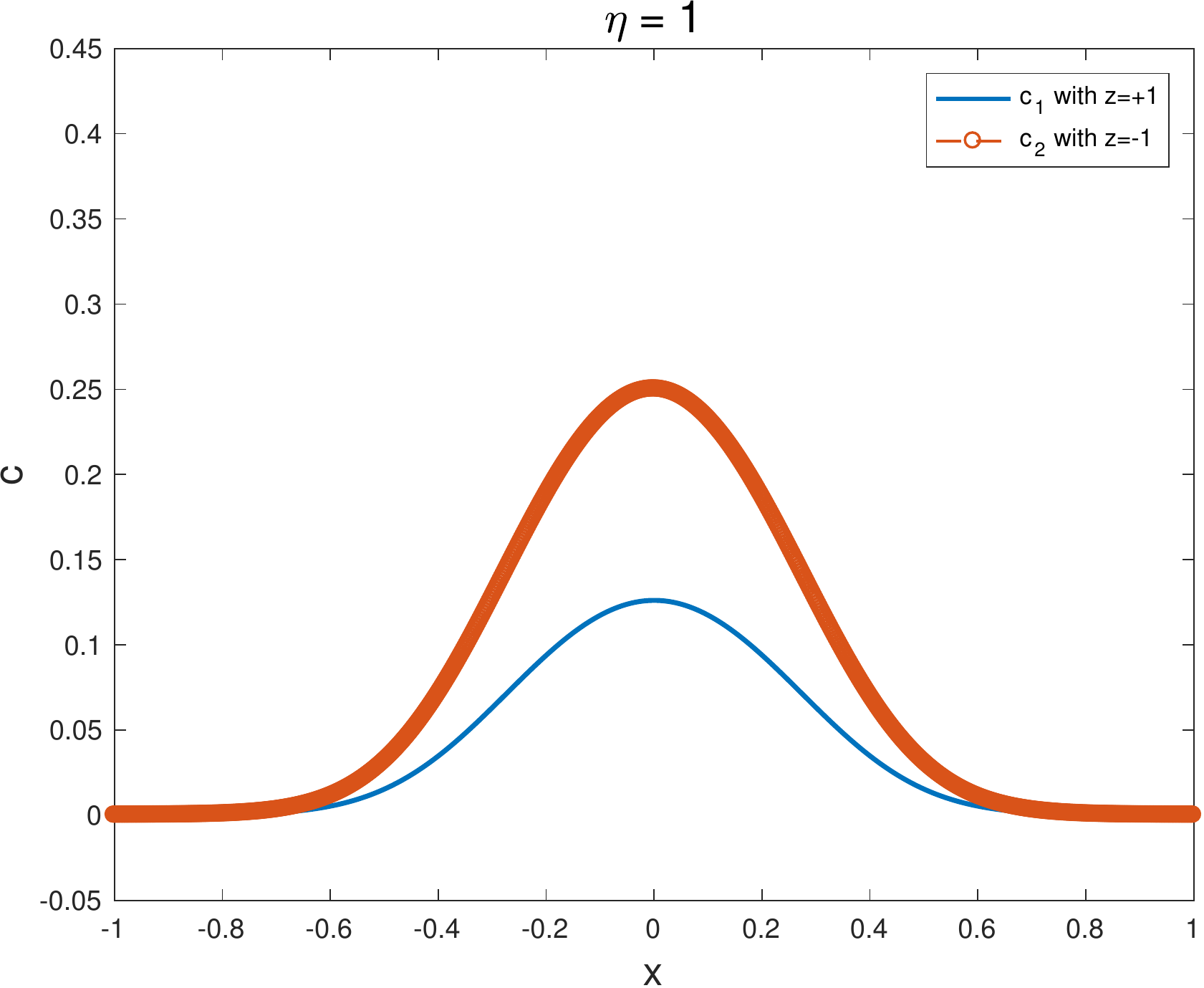}
			\end{minipage}
		}
		\subfigure{
			\begin{minipage}[t]{0.3\linewidth}
				\centering
				\includegraphics[width=1.0\linewidth]{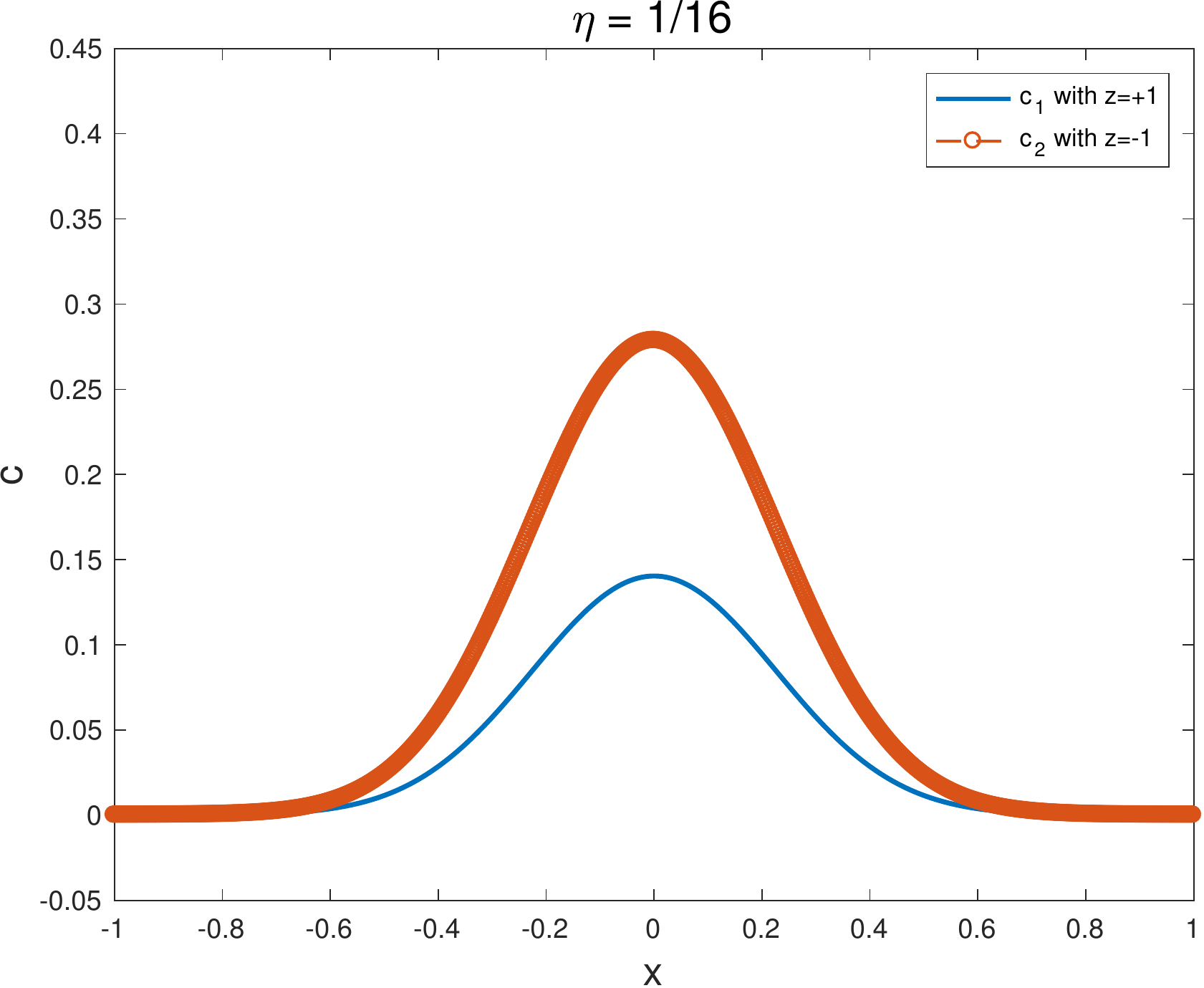}
			\end{minipage}
			\begin{minipage}[t]{0.3\linewidth}
				\centering
				\includegraphics[width=1.0\linewidth]{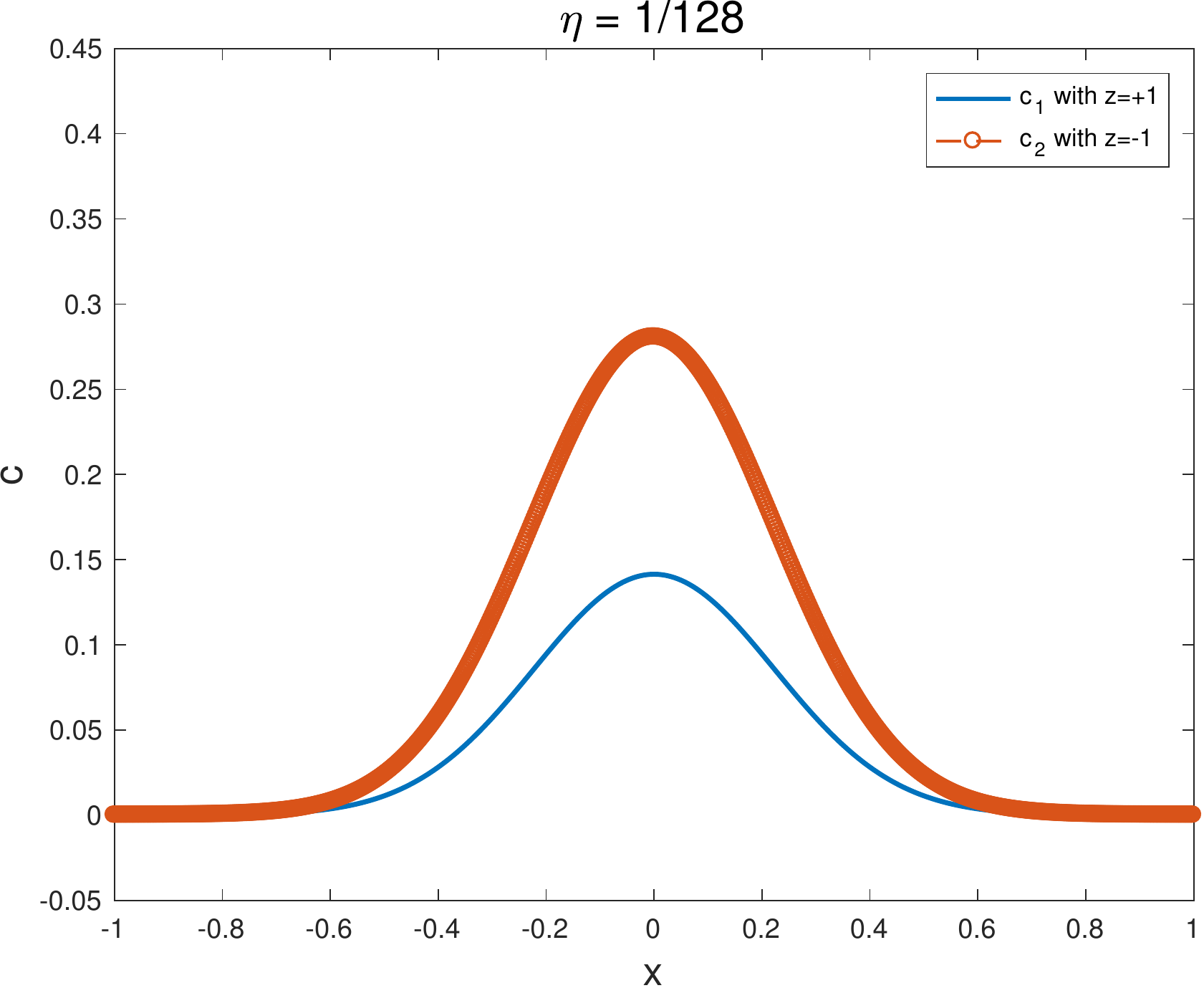}
			\end{minipage}
			\begin{minipage}[t]{0.3\linewidth}
				\centering
				\includegraphics[width=1.0\linewidth]{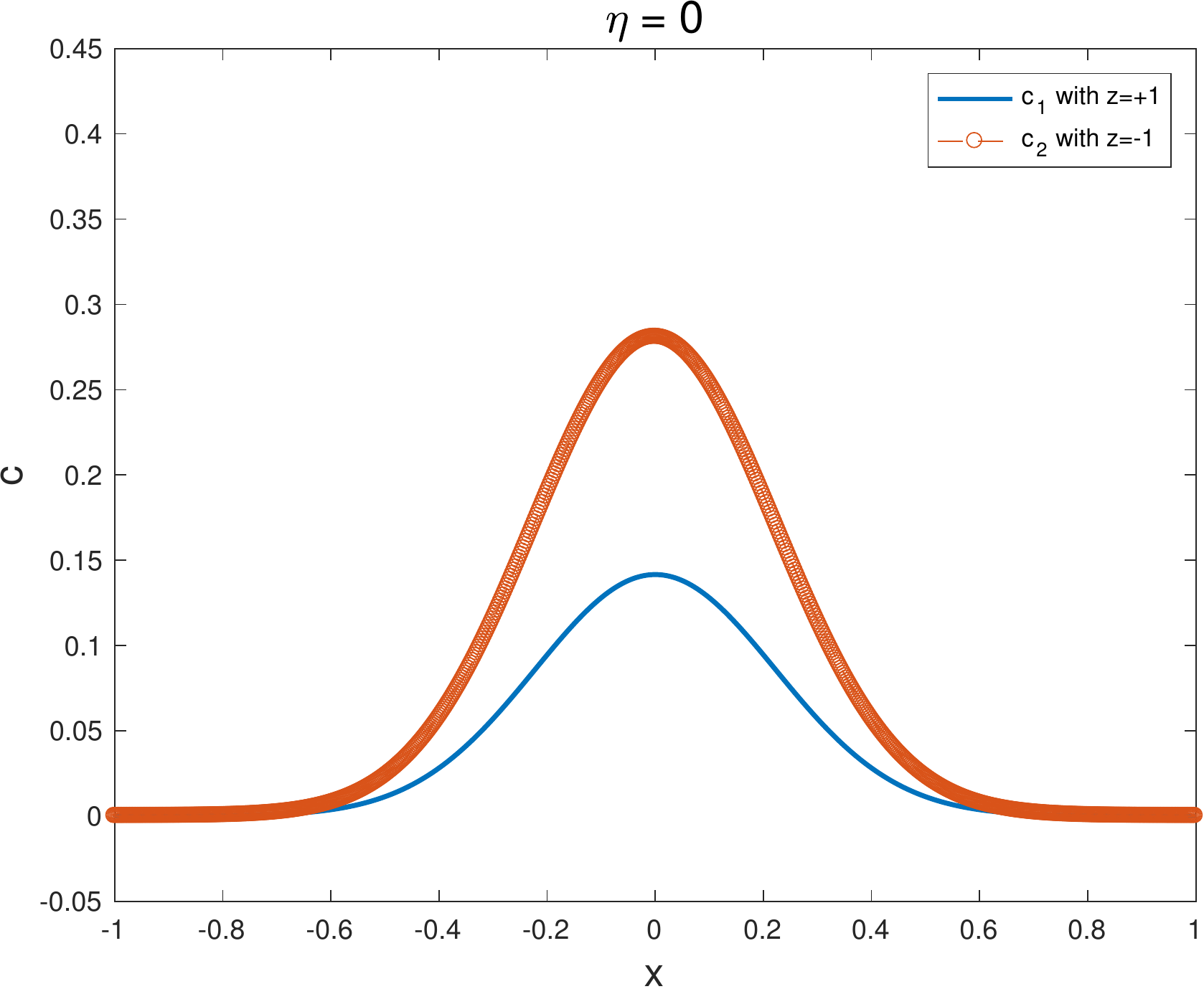}
			\end{minipage}
		}
		\caption{Multiple Species Example in One-dimension: {\bf Finite size effect.} The steady state density solutions $c_m$ with $\eta = 16, 4, 1,\frac{1}{16}, \frac{1}{128}, 0$.}
		\label{eta2}
	\end{figure}
	
	\subsubsection{Boundary Value Problem 1}
	Consider a new boundary value problem: equations (\ref{model2a}) in one-dimension with the kernel $\mathcal{W}(x) = \frac{\eta}{|x|^{1/2}}$, $\mathcal{K}(x) = \exp(-|x|)$ and the initial conditions (\ref{model2c})  given by (\ref{iniex1}).
	Instead of giving a confining external potential, we add a constant  electric field whose field intensity is 10 to the solutions, i.e. the field system (\ref{model2a})-(\ref{model2c}) becomes
	\begin{equation}
	\begin{aligned}
	\partial_t c_m (x, t) &= \nabla \cdot \left( \exp\{- \left( z_m \mathcal{K}*\rho + \mathcal{W}*\theta + z_m V_{0} \right) \} \nabla \frac{c_m}{\exp\{- \left( z_m \mathcal{K}*\rho + \mathcal{W}*\theta + z_m V_{0} \right) \}} \right), 
	\\
	c_m(x, 0) &= c^0_{m}(x), ~~m = 1, \cdots, M,
	\end{aligned}
	\end{equation}
	where $V_{0} = 10 x$ and no flux boundary is considered. 
	
	Denote the electric potential energy of the m-th species by  $\Phi_{0} = z_m V_{0}$, then the corresponding velocity $v_0 = -\partial_{x} \Phi_{0} = -10 z_m$. For positive electric charges, $v_0 < 0$, which means positively charged ions are driven towards the left boundary while for negative electric charges, $v_0 > 0$, negatively charged ions are driven towards the right boundary.  Next we aim to investigate this phenomenon numerically that such electric field can make positive and negative electric charges gather on different ends and how the nonlocal repulsion modifies the profile of the steady states.
	
	Take $\eta = 0, 1, 2, 4$, the computation domain as $[-L, L], \ L = 1$, Fig \ref{ex3c1} shows the behavior of the ionic species on domain $[-L, L]$ with the mesh size $\Delta x = 0.0025$ and $\Delta t = 0.001$ at $t = 1$. From Fig \ref{ex3c1} we can make conclusions that for given $\eta$,  the concentrations of different ionic species accumulate at different boundaries and the repulsion with large $\eta$ flattens and thickens the boundary layer.
	
	\begin{figure}[htp] 
		\subfigure[]{ 
			\includegraphics[width=0.45\textwidth]{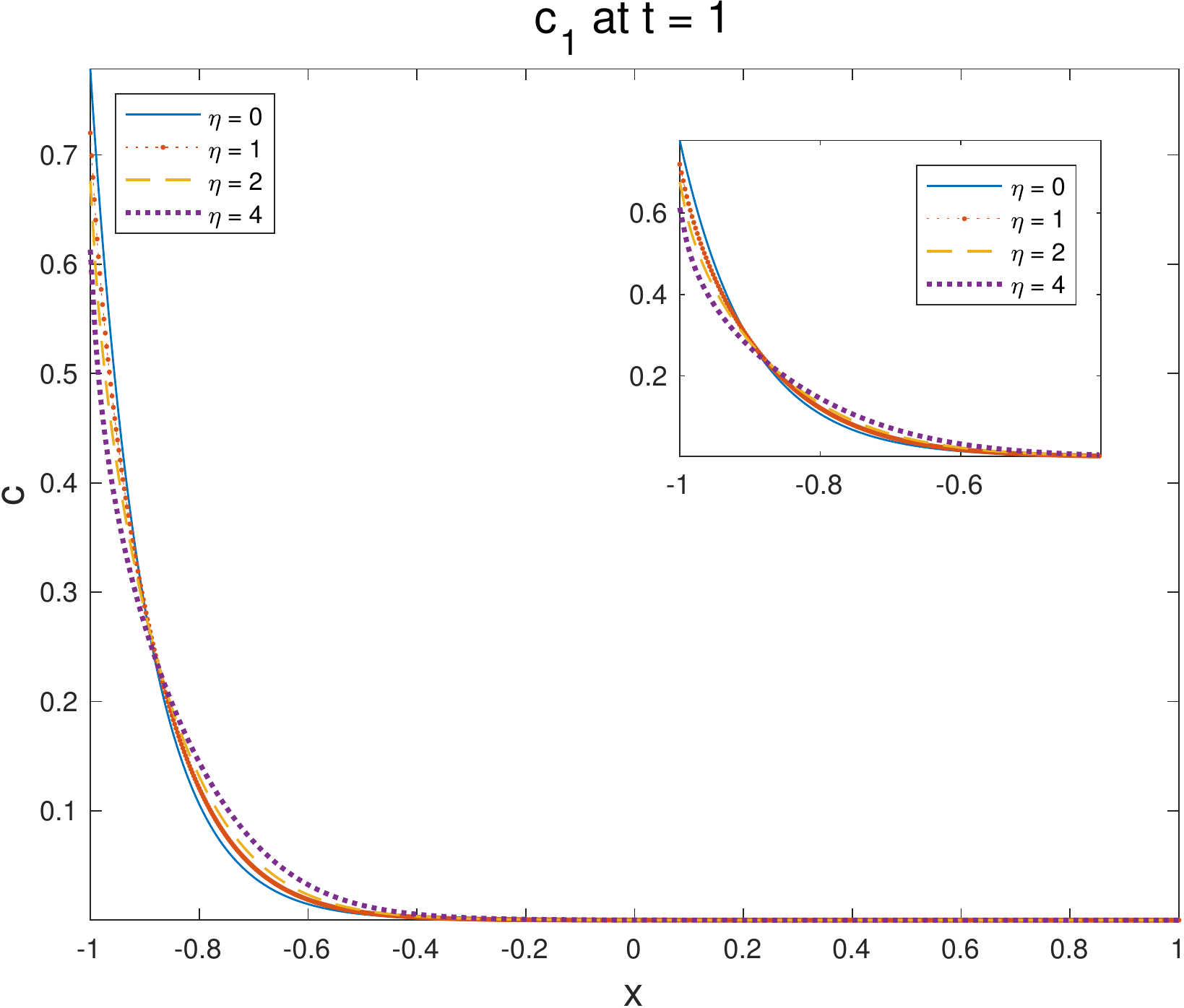}} 
		\subfigure[]{ 
			\includegraphics[width=0.45\textwidth]{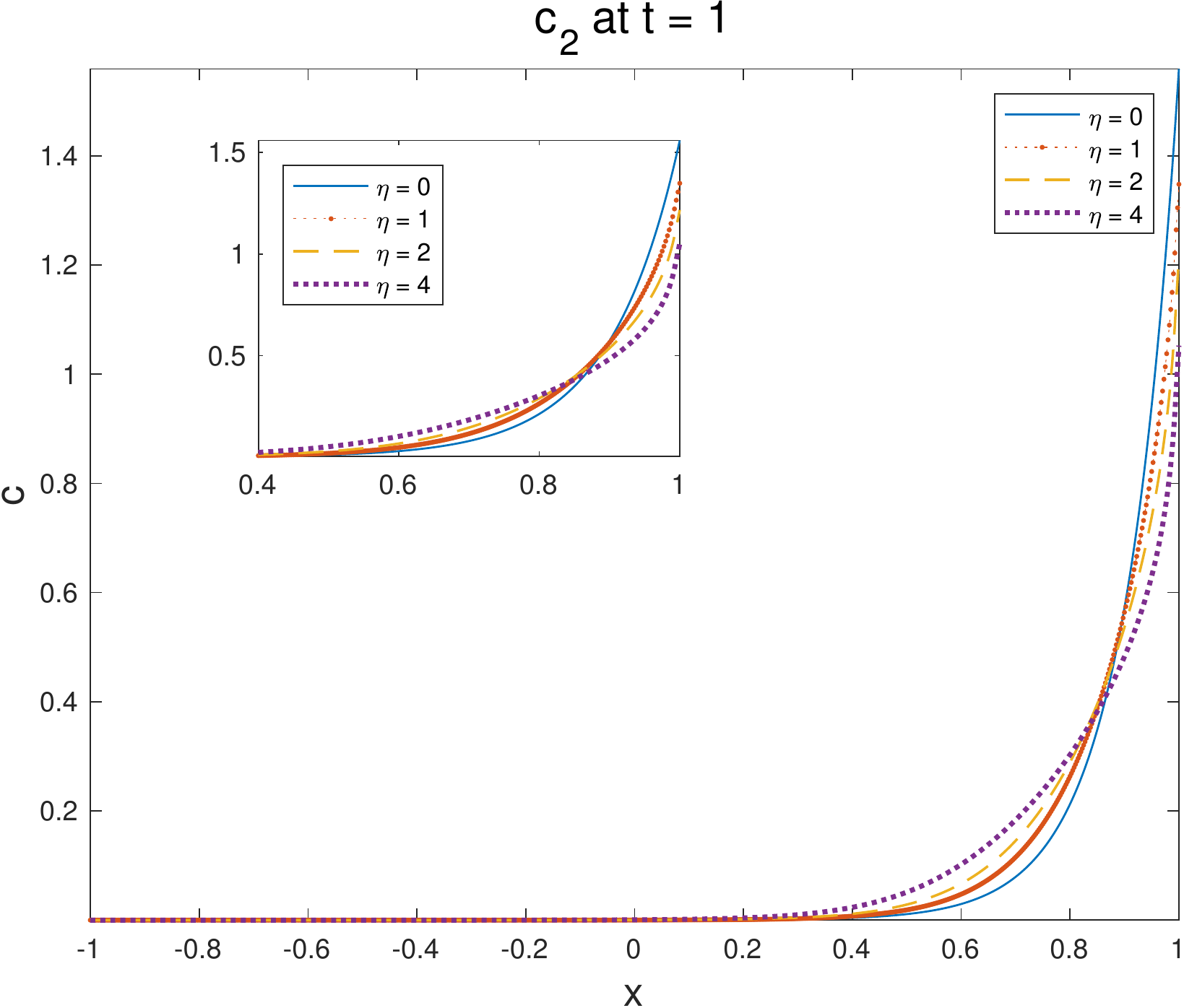} 
		}
		\caption{Boundary Value Problem 1: (a): $c_1$ with different $\eta$ at time $t = 1$. (b): $c_2$ with different $\eta$ at time $t = 1$.}
		\label{ex3c1}
	\end{figure}
	
	\subsubsection{Boundary Value Problem 2}
	Similar to the numerical test in \cite{hu2019}, consider another boundary value problem: on domain $[-L, L]$, the parabolic equations (\ref{model2a}) in one-dimension with the singular kernel $\mathcal{W}(x) = \frac{\eta}{|x|^{1/2}}$, the initial conditions (\ref{model2c}) given by (\ref{iniex1}) and the electrostatic potential
	$\phi_{\mathcal{K}}(x) = \int _{-L}^{L} \mathcal{K}(x-y) \rho(y) \,\mathrm{d} y$ related to $c_m$ being determined by Gauss's law. Then our model becomes
	\begin{equation}
	\label{problem2}
	\begin{aligned}
	\partial_t c_m (x, t) &= \nabla \cdot \left( \exp\{-(z_m \phi_{\mathcal{K}} + \mathcal{W}*\theta)\} \nabla \frac{c_m}{\exp\{-(z_m \phi_{\mathcal{K}} + \mathcal{W}*\theta)\}} \right), 
	\\
	- \Delta \phi_{\mathcal{K}}(x, t) &= \rho, \\
	c_m(x, 0) &= c^0_{m}(x), ~~m = 1, \cdots, M, \\
	\alpha \phi_{\mathcal{K}}(-1, t) &- \beta \partial_x \phi_{\mathcal{K}}(-1, t) = -10, \\ \alpha \phi_{\mathcal{K}}(1, t) &+ \beta \partial_x \phi_{\mathcal{K}}(1, t) = 10, \quad \alpha = 1, \quad \beta = 0.01.
	\end{aligned}
	\end{equation}
	
	In this model, we compute the electrostatic potential $\phi_{\mathcal{K}}(x)$ via the Poisson equation instead of the convolution of the kernel ${\mathcal{K}}$ and the charge density $\rho$. The boundary conditions of the model (\ref{problem2}) make positive charged particles move to the left boundary and negative charged particles move to the right boundary.
	
	Here, take $\eta = 0, 1, 2, 4$, $\ L = 1$, Fig \ref{ex3c} shows the results of the concentration at the boundary on domain $[-L, L]$ with the mesh size $\Delta x = 0.0025$ and $\Delta t = 0.001$ at $t = 1$. It's observed that on the qualitative level, the solutions behaviors in this test agree with those of the previous test.
	
	\begin{figure}[htp] 
		\subfigure[]{ 
			\includegraphics[width=0.45\textwidth]{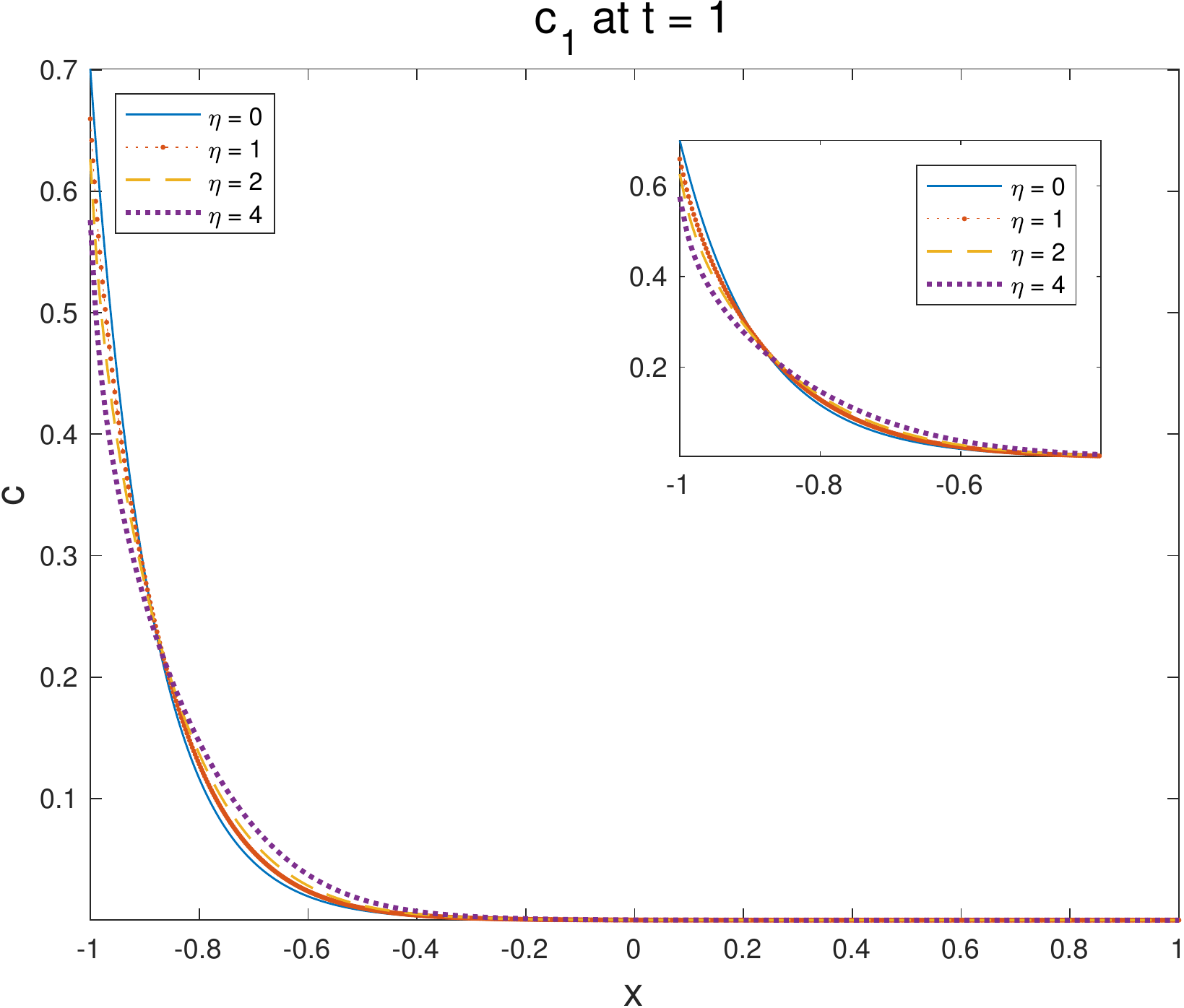}} 
		\subfigure[]{ 
			\includegraphics[width=0.45\textwidth]{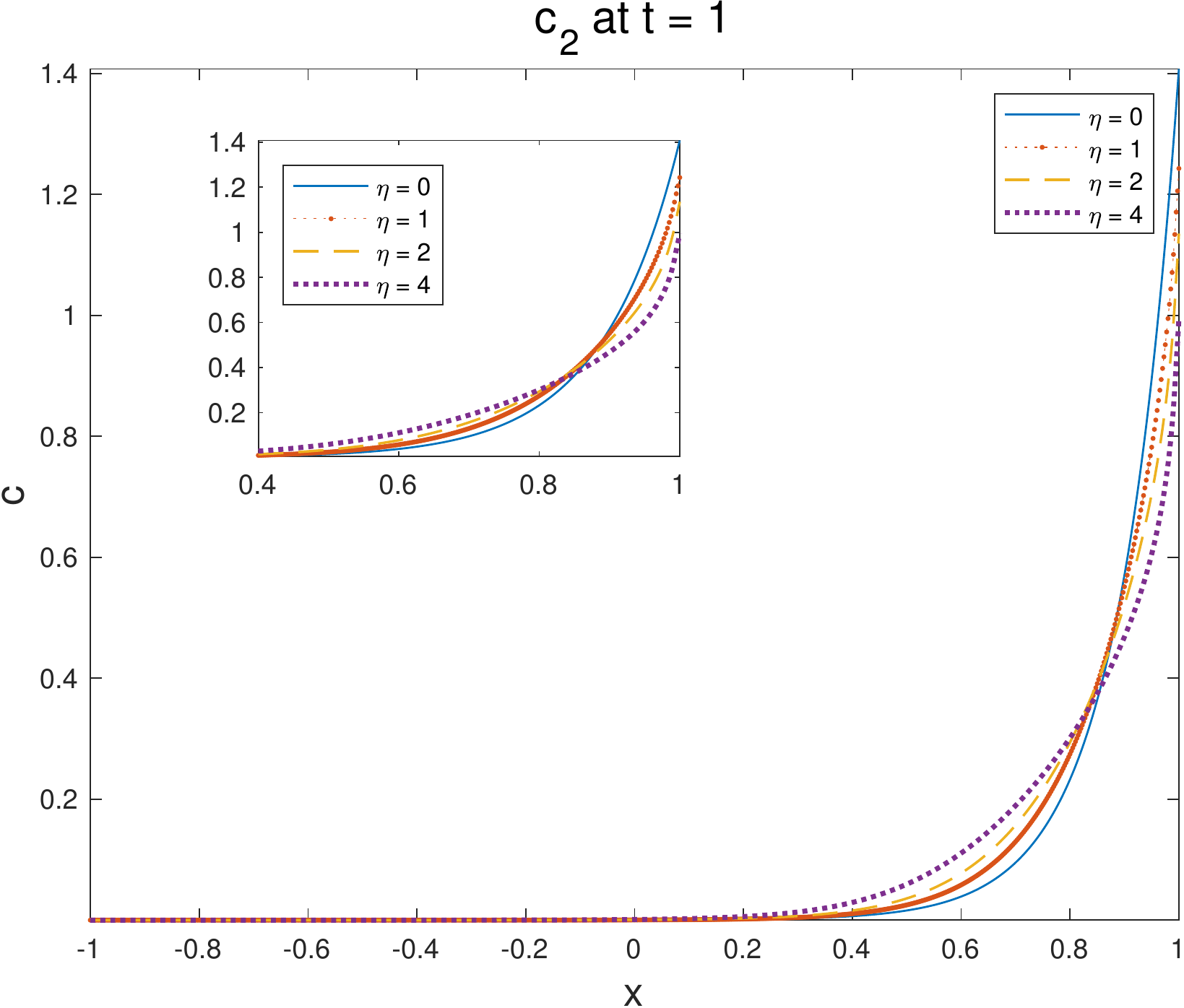} 
		}
		\caption{Boundary Value Problem 2: (a): $c_1$ with different $\eta$ at time $t = 1$. (b): $c_2$ with different $\eta$ at time $t = 1$.}
		\label{ex3c}
	\end{figure}
	
	\subsection{Numerical Experiments in Two-dimension}
	
	In this part, we consider the equations (\ref{model2a}) in two-dimension with the two-dimensional singular kernel $\mathcal{W}(x, y) = \frac{\eta}{r^{3/2}}, r = \sqrt{x^2 + y^2}$, the singular kernel $\mathcal{K}(x, y) = - \frac{1}{2 \pi} \ln r$, the external potential $V_{\text{ext}}(x, y) = 10 r^2$ and the initial conditions (\ref{model2c}) given by (\ref{iniex3}).
	We point out that $\mathcal{W}$ is repulsive and is more singular than $\mathcal{K}$, and it is a more realistic model, comparatively, $\mathcal{W}$ is dominating in the short range and $\mathcal{K}$ effectively determines the long range interaction. 
	
	On computation domain $[-L, L] \times [-L, L], \ L = 1$, we take $\eta = 1$, the mesh size $\Delta x = \Delta y = 0.02$ and $\Delta t = 0.0004$, Fig \ref{c1oft} shows the time evolution of the ion concentrations of $1$-th the ionic species.
	\begin{figure}[htp] 
		\subfigure[]{ 
			\includegraphics[width=0.3\textwidth]{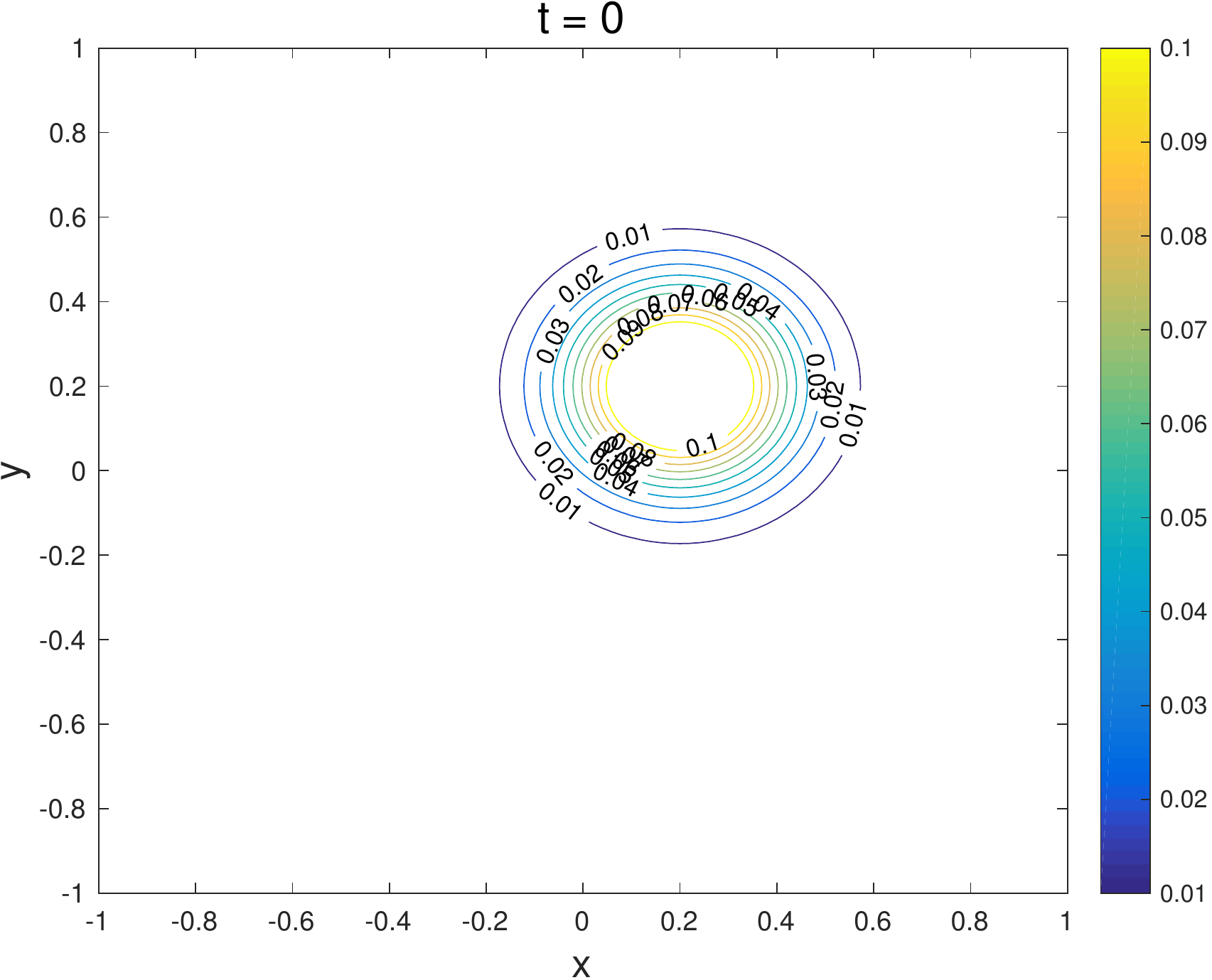}}  
		\subfigure[]{ 
			\includegraphics[width=0.3\textwidth]{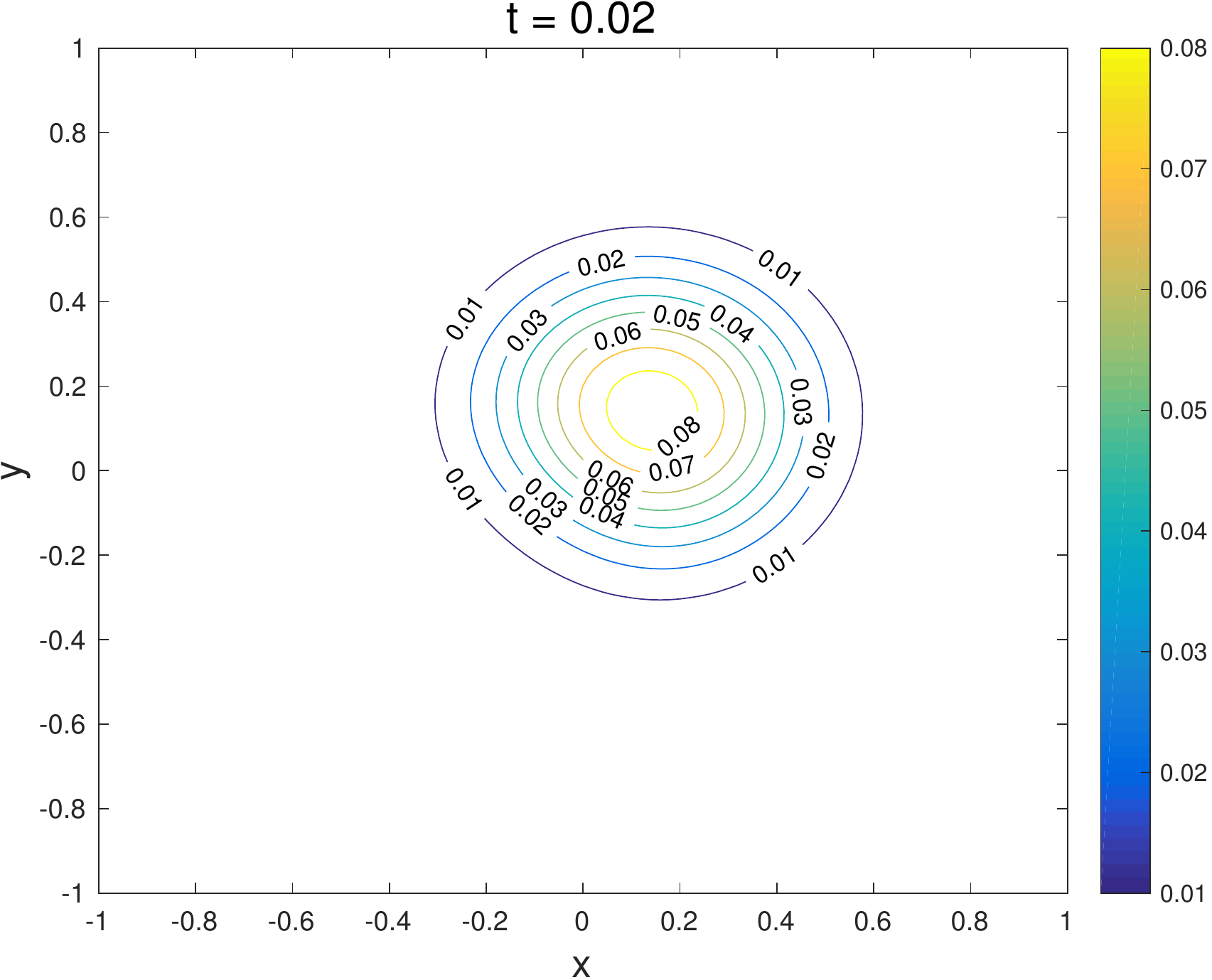}  
		}
		\subfigure[]{ 
			\includegraphics[width=0.3\textwidth]{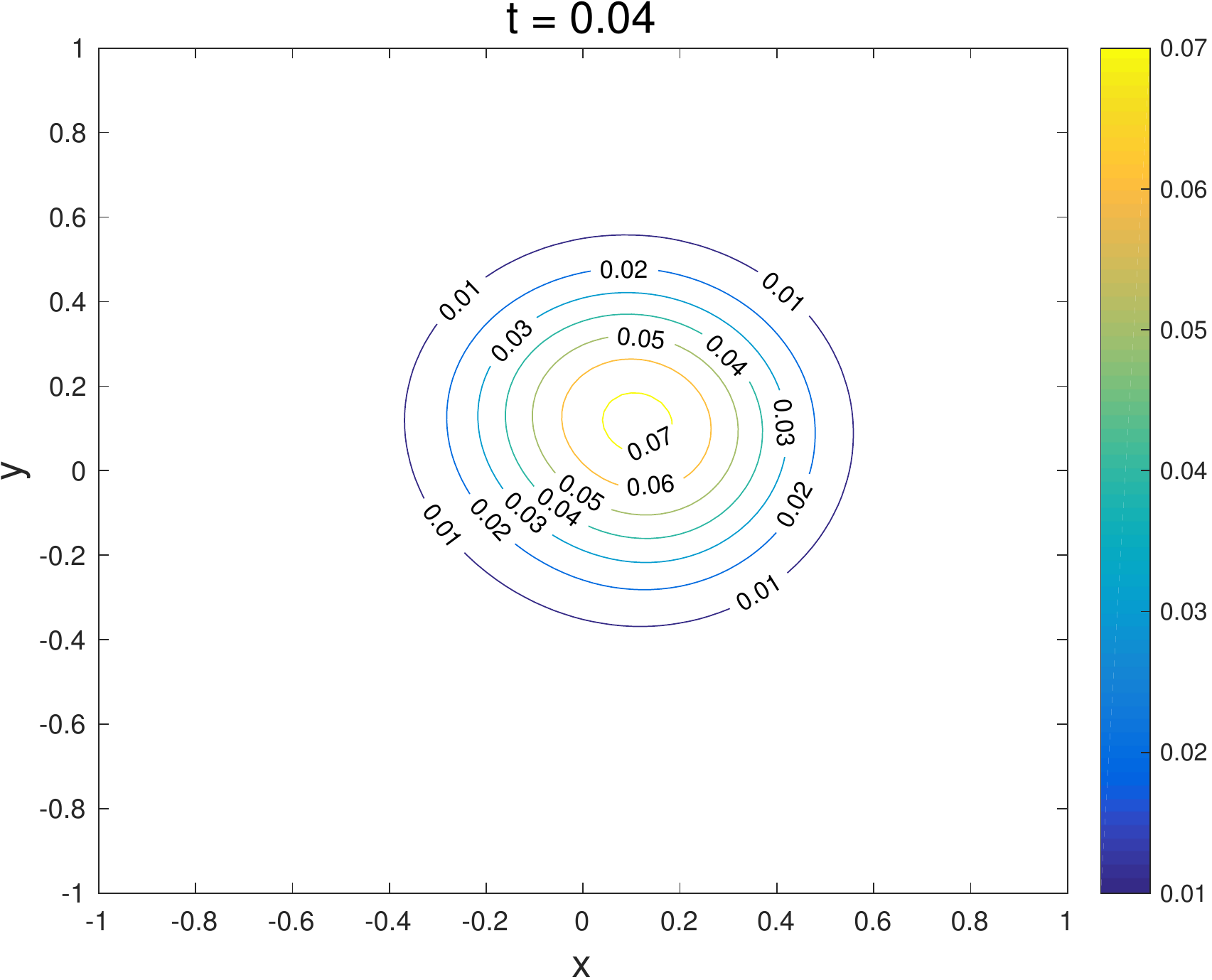}  
		}
		\subfigure[]{ 
			\includegraphics[width=0.3\textwidth]{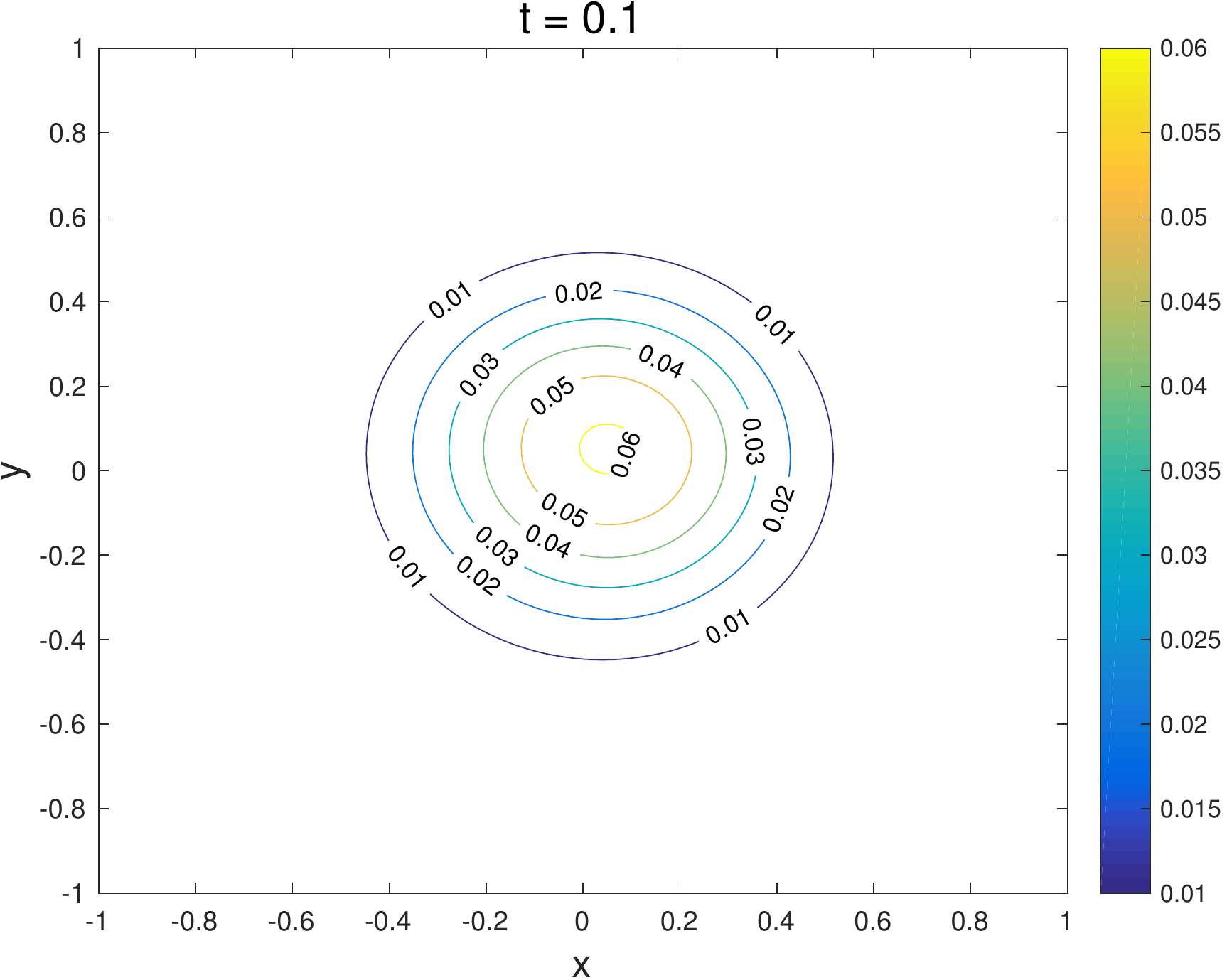} 
		}
		\subfigure[]{ 
			\includegraphics[width=0.3\textwidth]{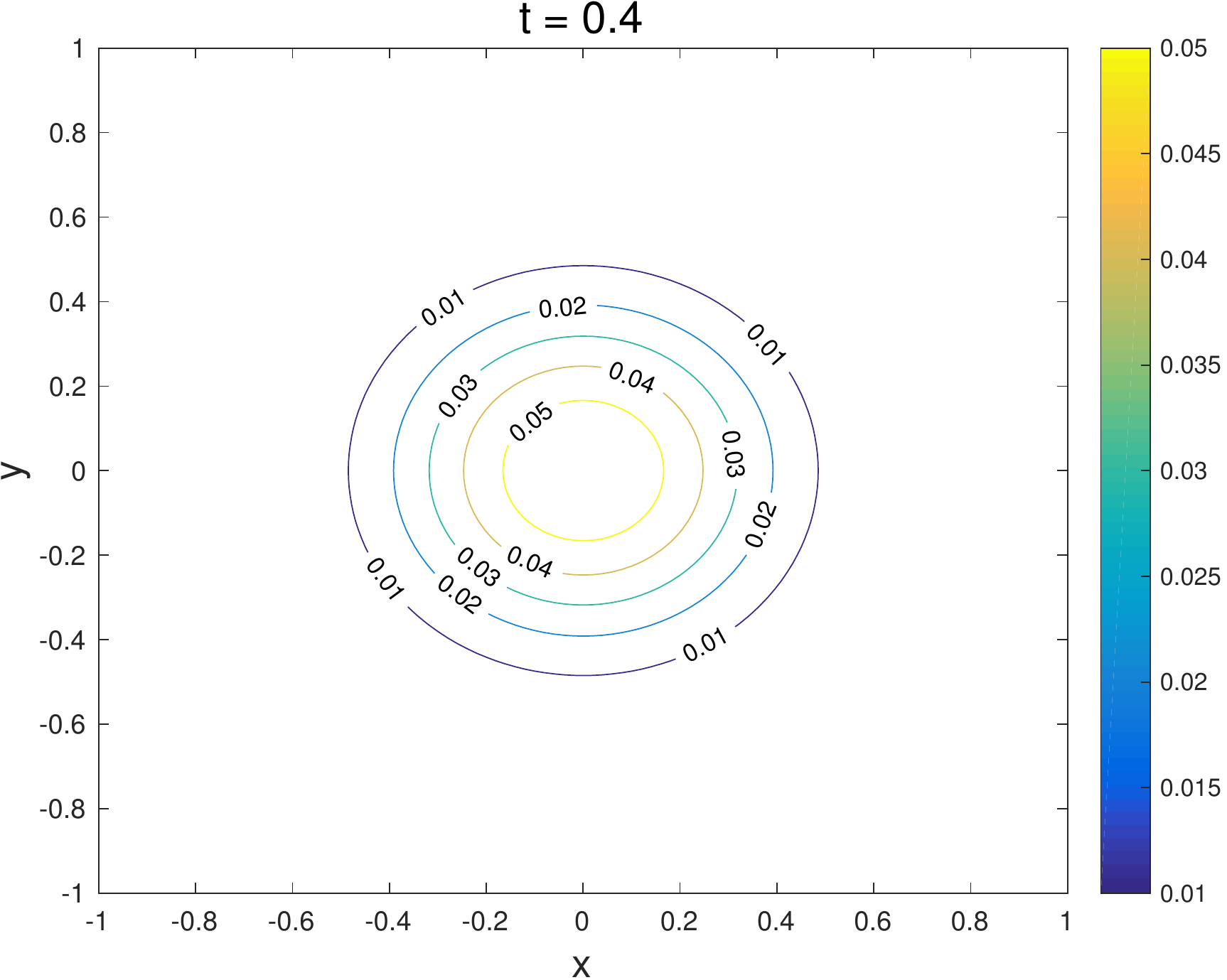}
		}
		\subfigure[]{ 		\includegraphics[width=0.3\textwidth]{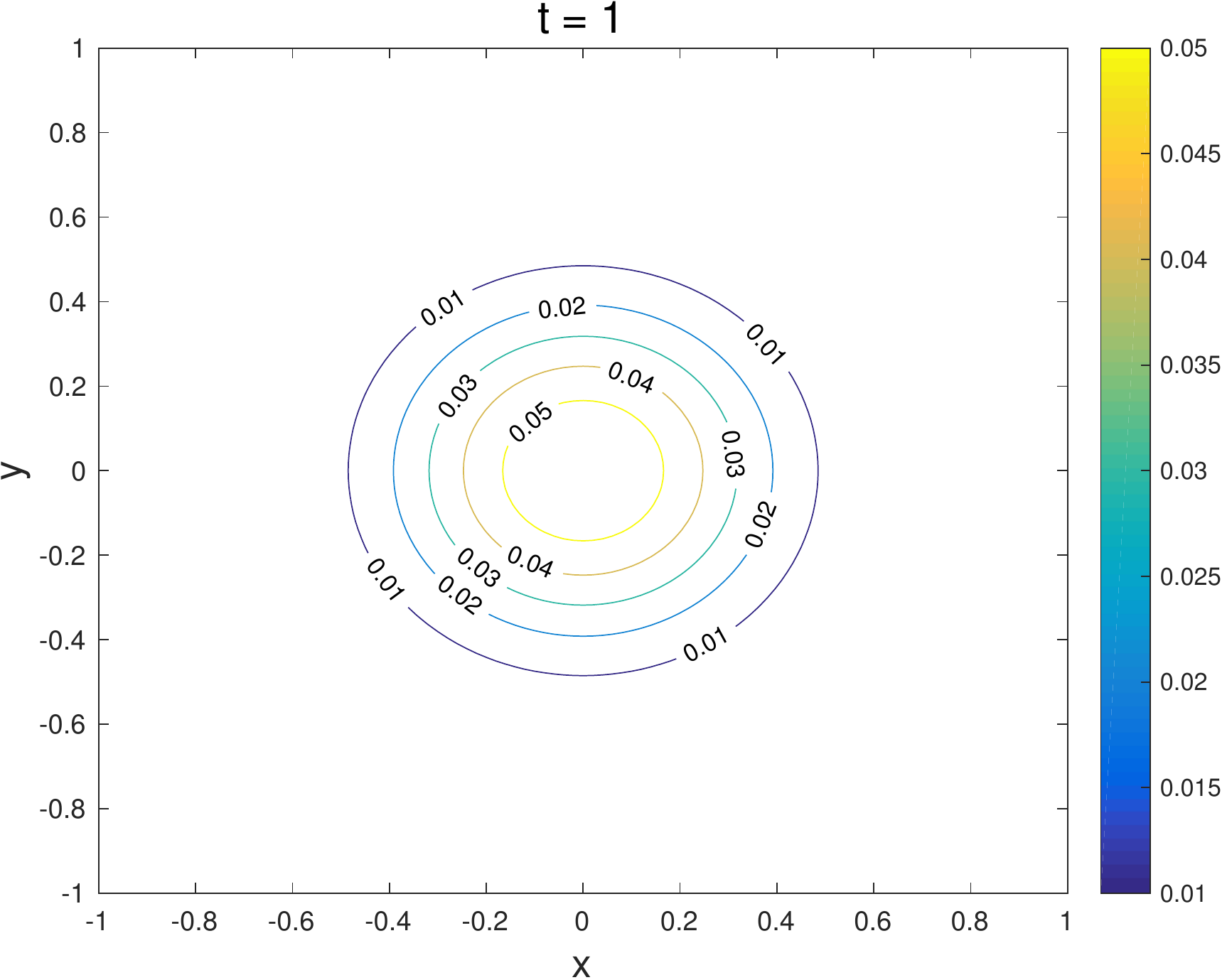}
		}
		\caption{Multiple Species Example 1 in Two-dimension: the transport of $1$-th the ionic species with the mesh size $\Delta x$ and $\Delta y$ being 0.02 and the time $t$ changing from 0 to 1.}
		\label{c1oft}
	\end{figure}
	
	\subsubsection{Finite Size Effect} 
	Similar to the one-dimensional case, the kernel $\mathcal{W}$ in the model (\ref{model2a})-(\ref{model2c}) represents the steric repulsion arising from the finite size 
	and the strength of the kernel $\mathcal{W}$ is indicated by the parameter $\eta$, which means the larger $\eta$ is, the stronger the nonlocal steric repulsion effect is, and thus the less peaked the concentrations of the steady state are. And $\eta = 0$ means steric repulsion vanishes. 
	Here we aim to explore this phenomenon by different values of the parameter $\eta$.
	Let $\eta = 4, 1, \frac{1}{4}, \frac{1}{16}, \frac{1}{64}, 0$, the mesh size $\Delta x = \Delta y = 0.02$ and $\Delta t = 0.0004$, 
	Fig \ref{c1ofeta} shows different steady state solutions with different  $\eta$, where we can find that the finite size effect makes the concentrations $c_m, \ m = 1, 2,$ not overly peaked and we can verify that the nonlocal field induced by $\mathcal{W}$ effectively captures the steric repulsion arising from the finite size of the particles.  The numerical result is also consistent with that in \cite{liu2010, Liu2019}.

	\begin{figure}[htp] 
		\subfigure[]{ 
			\includegraphics[width=0.3\textwidth]{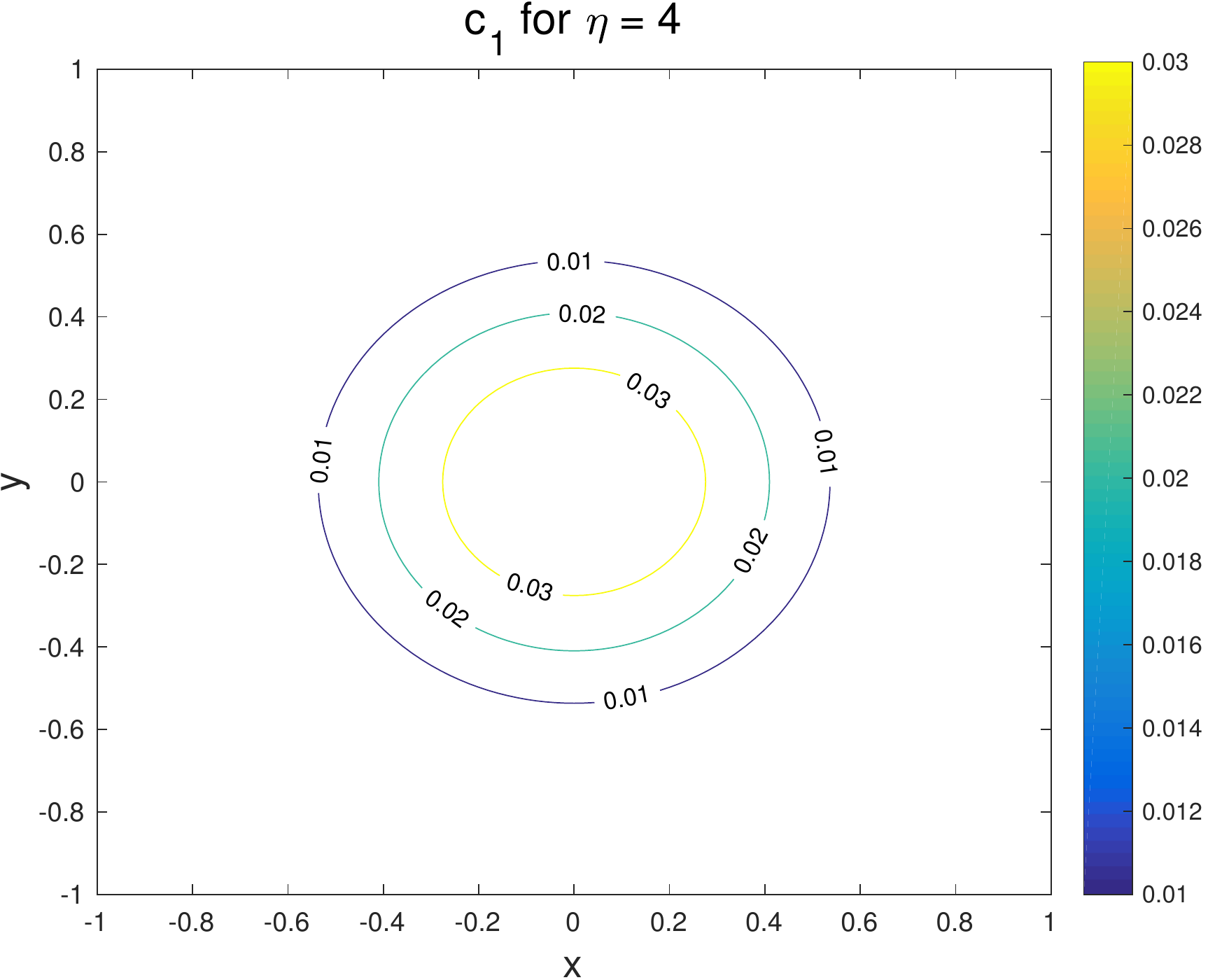}} 
		\subfigure[]{ 
			\includegraphics[width=0.3\textwidth]{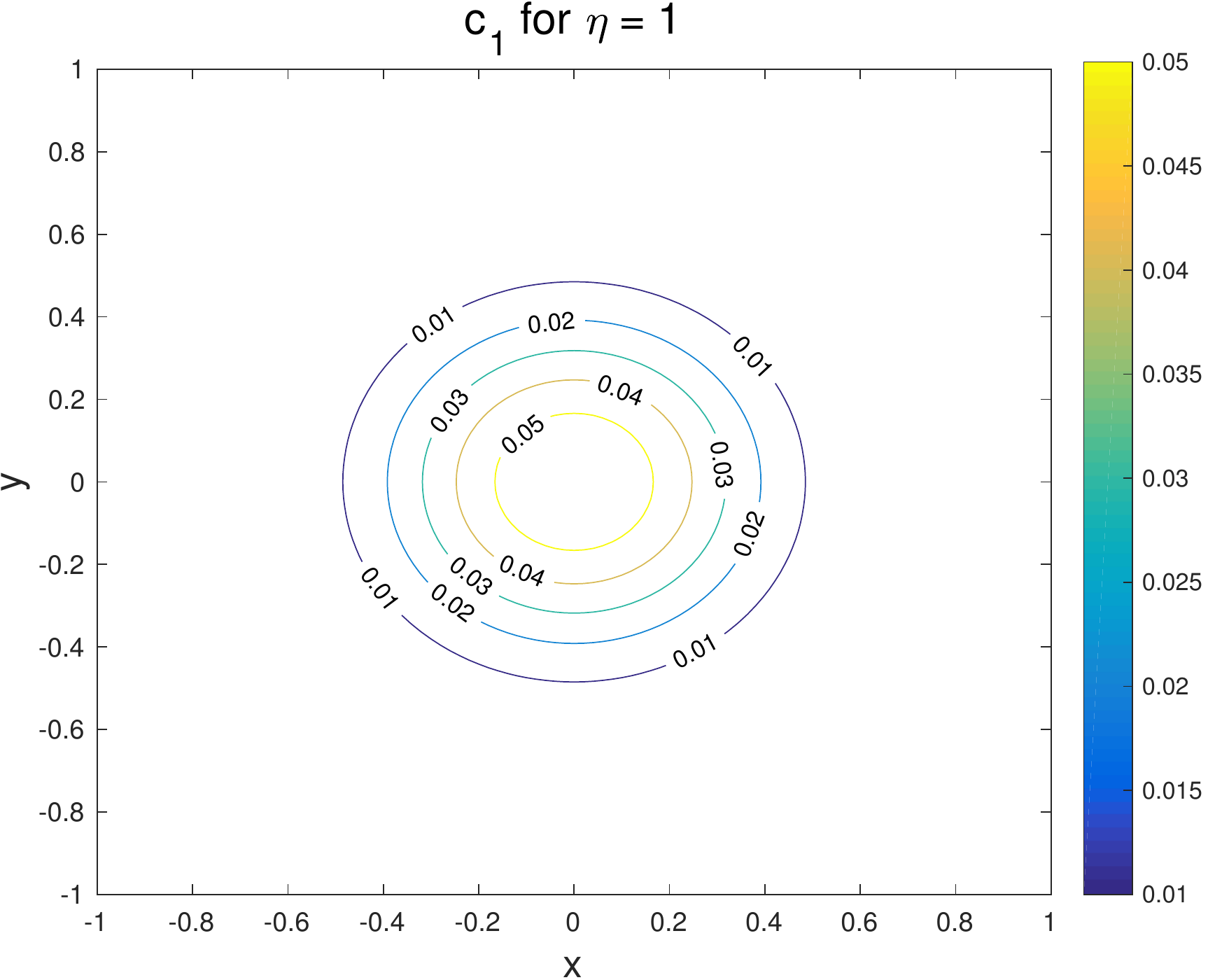} 
		}
		\subfigure[]{ 
			\includegraphics[width=0.3\textwidth]{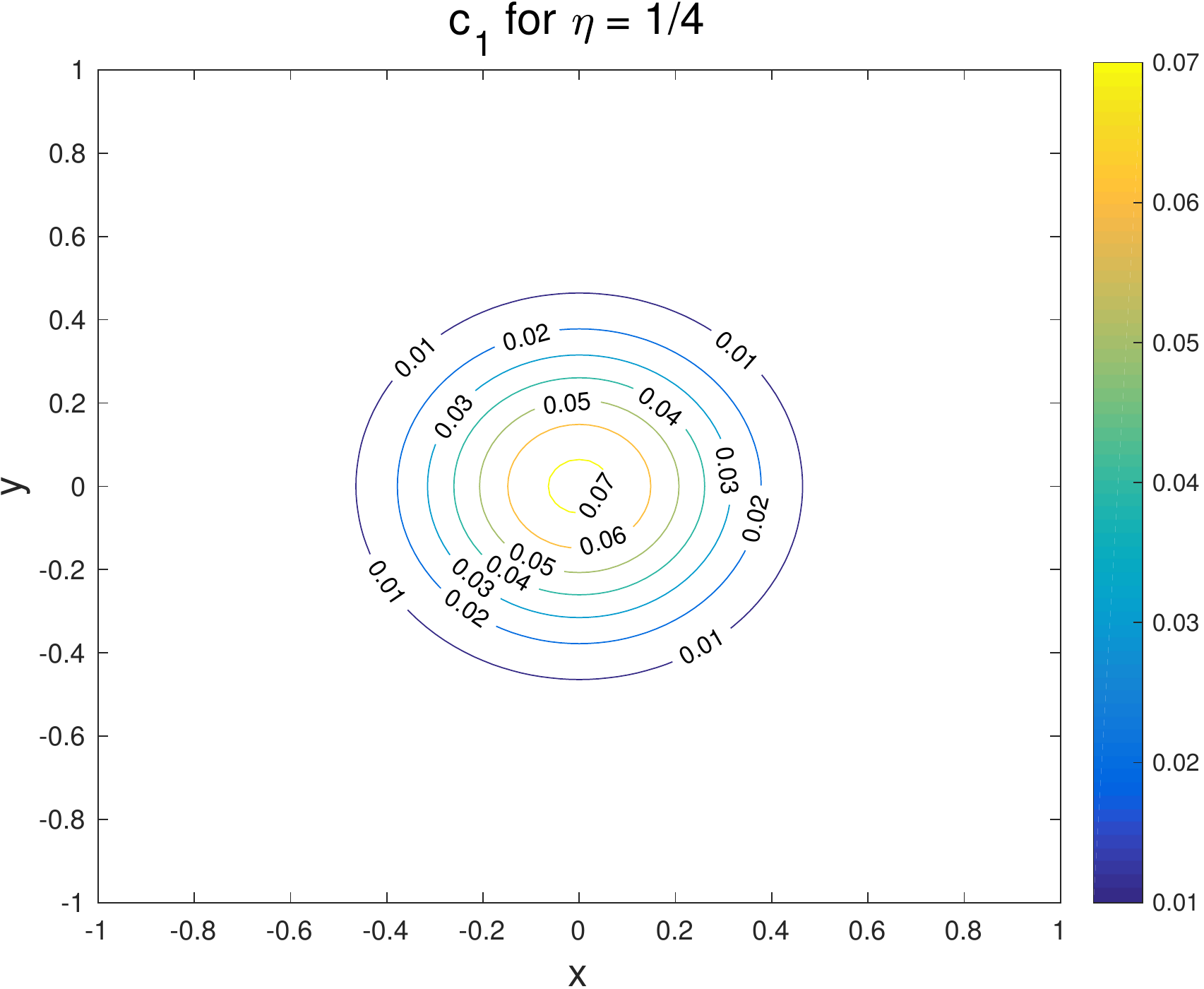} 
		}
		\subfigure[]{ 
			\includegraphics[width=0.3\textwidth]{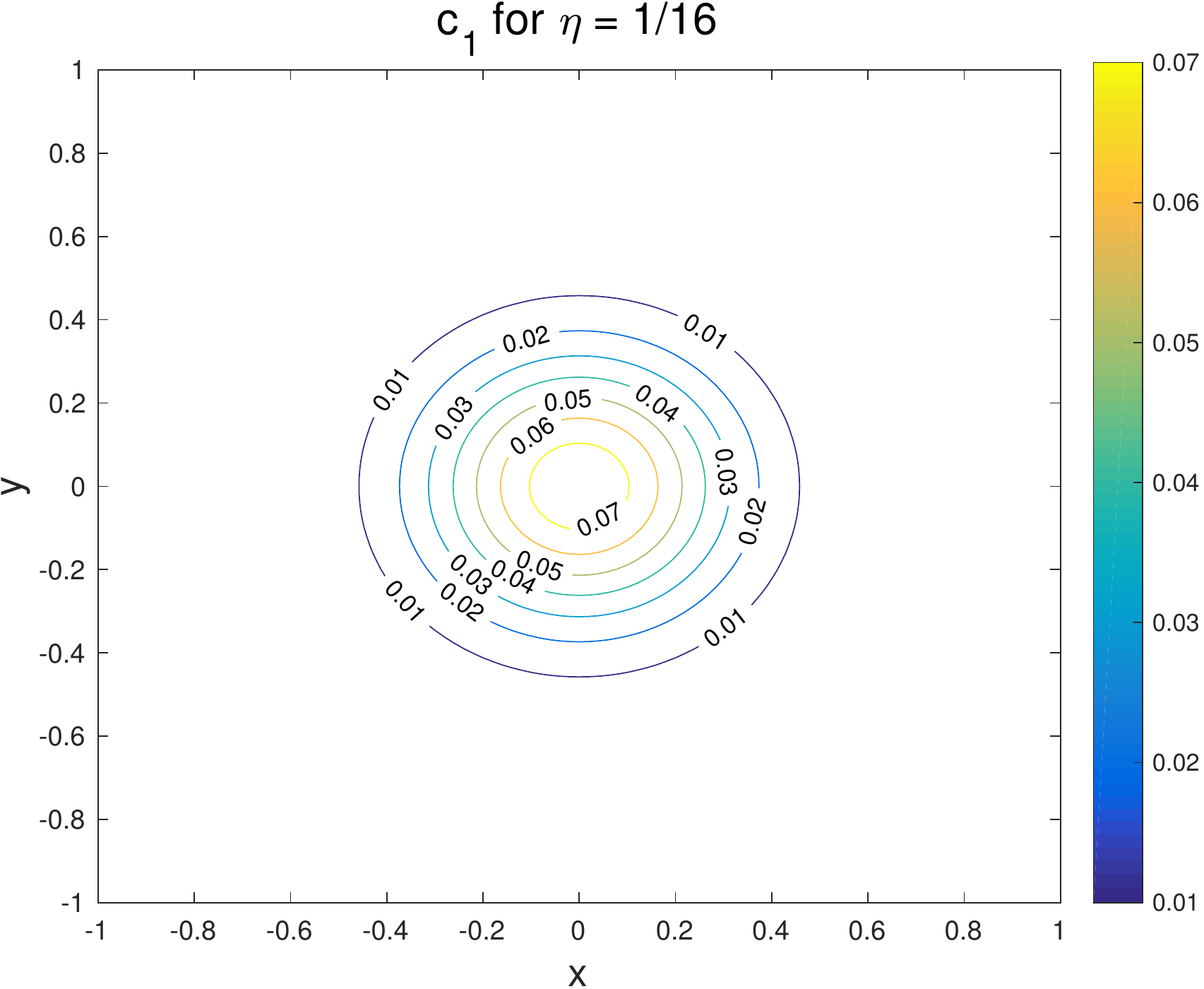}
		}
		\subfigure[]{ 		\includegraphics[width=0.3\textwidth]{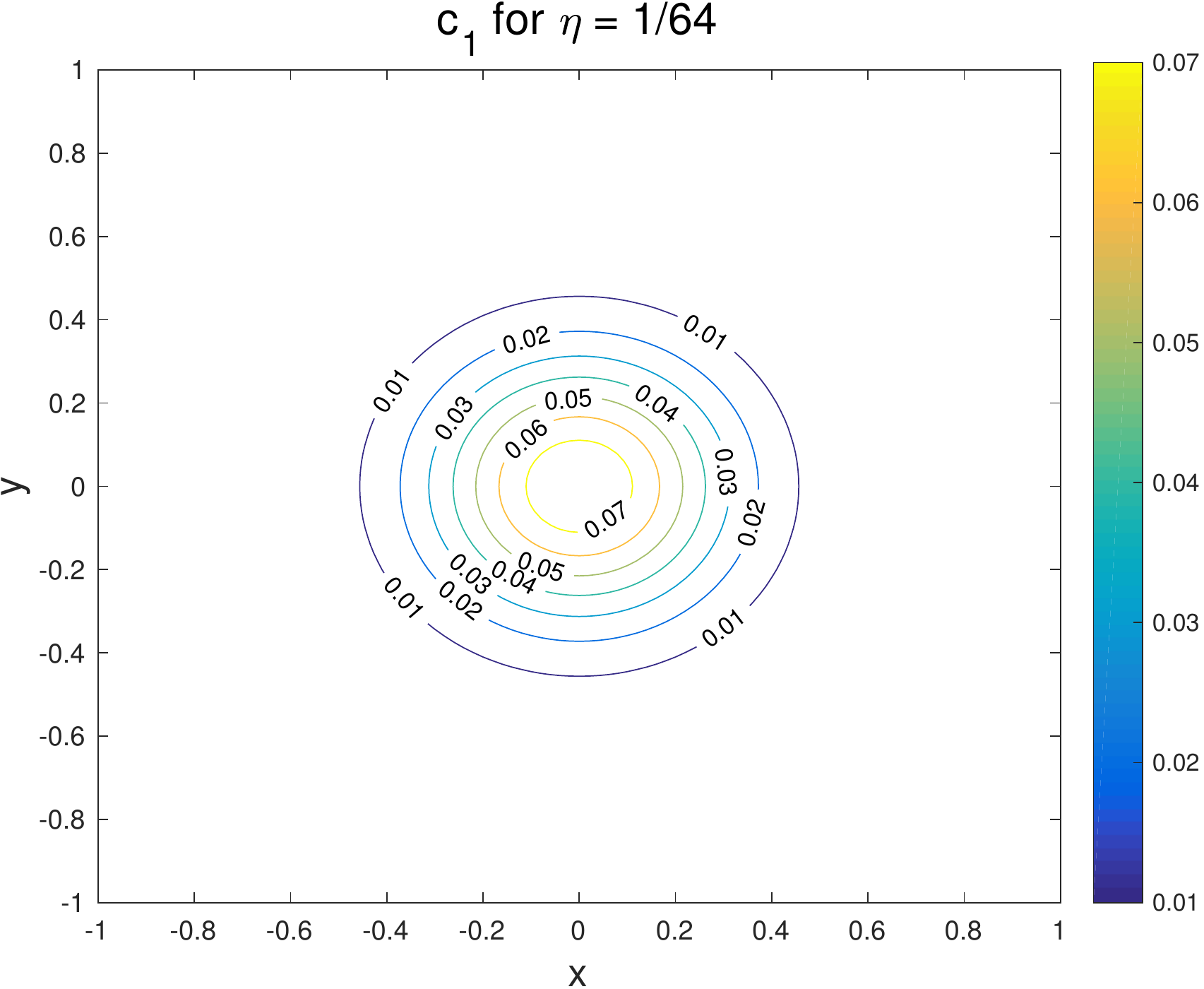} 
		}
		\subfigure[]{ 		\includegraphics[width=0.3\textwidth]{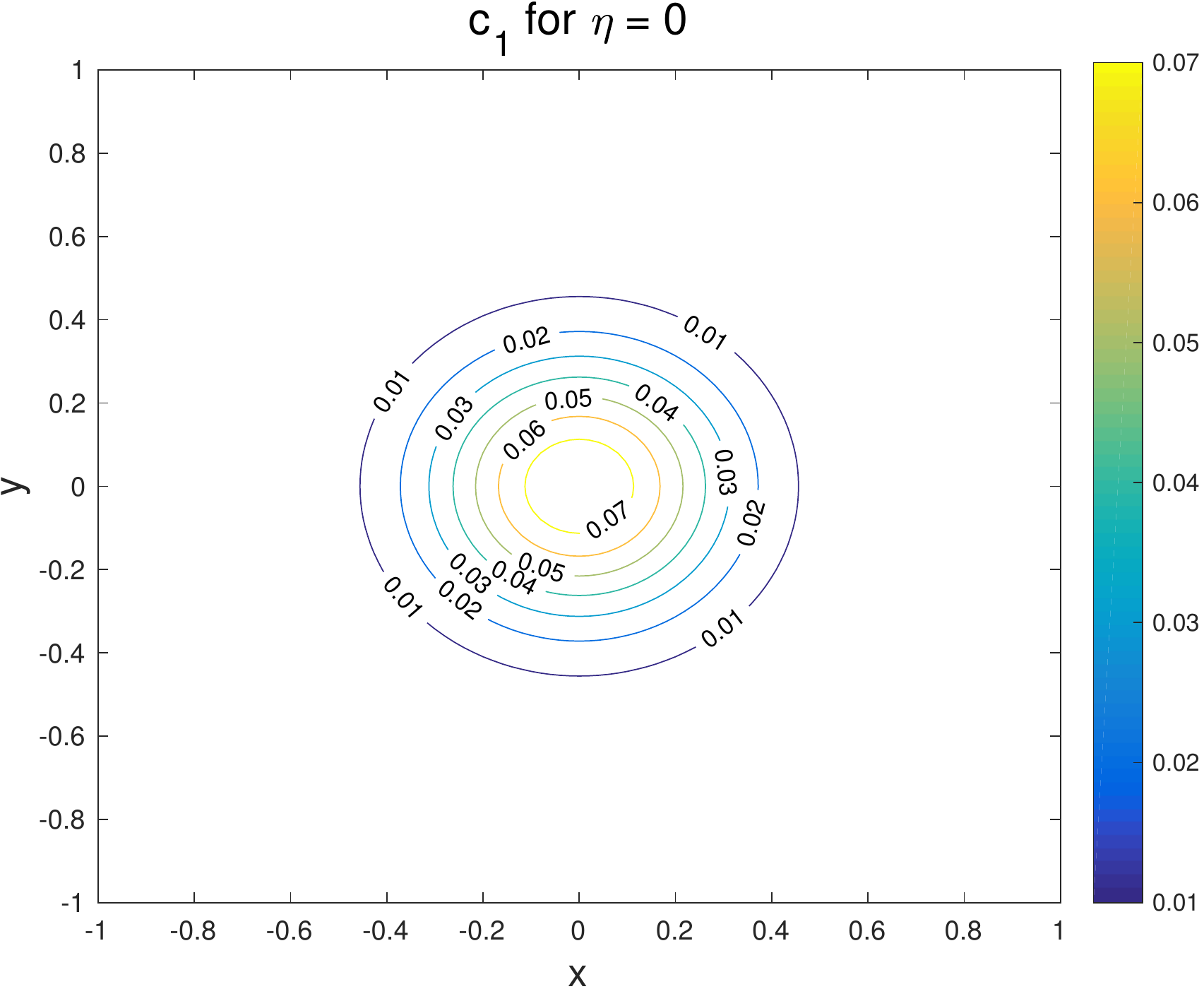} 
		}
		\caption{Multiple Species Example 1 in Two-dimension: {\bf Finite size effect.} the steady state density solutions $c_1$ with $\eta = 4, 1, \frac{1}{4}, \frac{1}{16}, \frac{1}{64}, 0$.}
		\label{c1ofeta}
	\end{figure}
	
	\subsubsection{Asymptotic Time Complexity} 
	For this problem, Table \ref{table1} shows that the asymptotic time complexity of the convolution of singular kernel is $O\left(N\log (N)\right)$.
	
	\begin{table}[htbp]
		\centering  
		\caption{Run Time of Convolution}  
		\label{table1}  
		\begin{tabular}{|c|c|c|c|c|}  
			\hline  
			& & & &  \\[-6pt]  
			mesh size $\Delta x = \Delta y$  & 0.02 & 0.01 & 0.005 & 0.0025 \\  
			\hline
			\hline
			& & & &  \\[-6pt]  
			$N_x = N_y$ &   100 &  200 &  400 & 800  \\ 
			time &    10.0148   & 49.0993  &  206.0754 & 875.1494  \\
			time$/(N\log N), N = N_x \cdot N_y$  &  $0.1087 \times 10^{-3}$  &  $1.158\times 10^{-4}$  &  $1.075\times 10^{-4}$  &  $1.023\times 10^{-4}$  \\
			\hline
		\end{tabular}
	\end{table}
	
	\subsubsection{Steady States in a Multi-well External Potential Function} 
	We consider a more complicated example: the multi-well extern potential function is taken as
	{\small
		$$
		V_{\text{ext}}(x, y) = -3 \exp\left(-10\left(x-\frac{1}{5}\right)^2-10\left(y-\frac{1}{5}\right)^2\right)-2\exp\left(-10\left(x+\frac{3}{10}\right)^2-10\left(y+\frac{1}{5}\right)^2\right)+\left(x^2+y^2\right)
		$$
	}
	and initial conditions (\ref{model2c}) are given by the following form
	\begin{equation}
	\label{iniex4}
	\left\{
	\begin{array}{lll}
	c_1^0 = \frac{5}{ 2 \pi} \exp \left(-20 \left(\left(x - \frac{1}{5}\right)^2 + \left(y - \frac{1}{5}\right)^2 \right)\right) &\text{with} &z_1 = 1, \\
	c_2^0 = \frac{5}{\pi} \exp \left(-20\left(\left(x + \frac{1}{5}\right)^2 + \left(y + \frac{1}{5}\right)^2\right)\right) &\text{with} &z_2 = -1,
	\end{array}
	\right.
	\end{equation}
	meanwhile, other conditions remain the same.
	
	Retake $\eta = 1$, the computation domain as $[-L, L] \times [-L, L], \ L = 1$ and the mesh size $\Delta x = \Delta y = 0.025$, Fig \ref{ex4c1} and Fig \ref{ex4c2} show how the time evolution of the concentrations of the $m$-th ionic species $c_m, m = 1, 2$, change with time $t$ respectively. And Fig \ref{muoft2d2} and Fig \ref{eoft2d2} show the relation between the time $t$ and the discrete forms of the chemical potential $\mu_m, m = 1, 2$ and energy $\mathcal{F}$. It's observed that the discrete free energy decays and the chemical potential $\mu_m$ goes to a constant while the field model goes to the equilibrium for all $m$. The results are also consistent with the conclusions in \cite{Liu2019}.
	
	\begin{figure}[htp] 
		\subfigure[]{ 
			\includegraphics[width=0.3\textwidth]{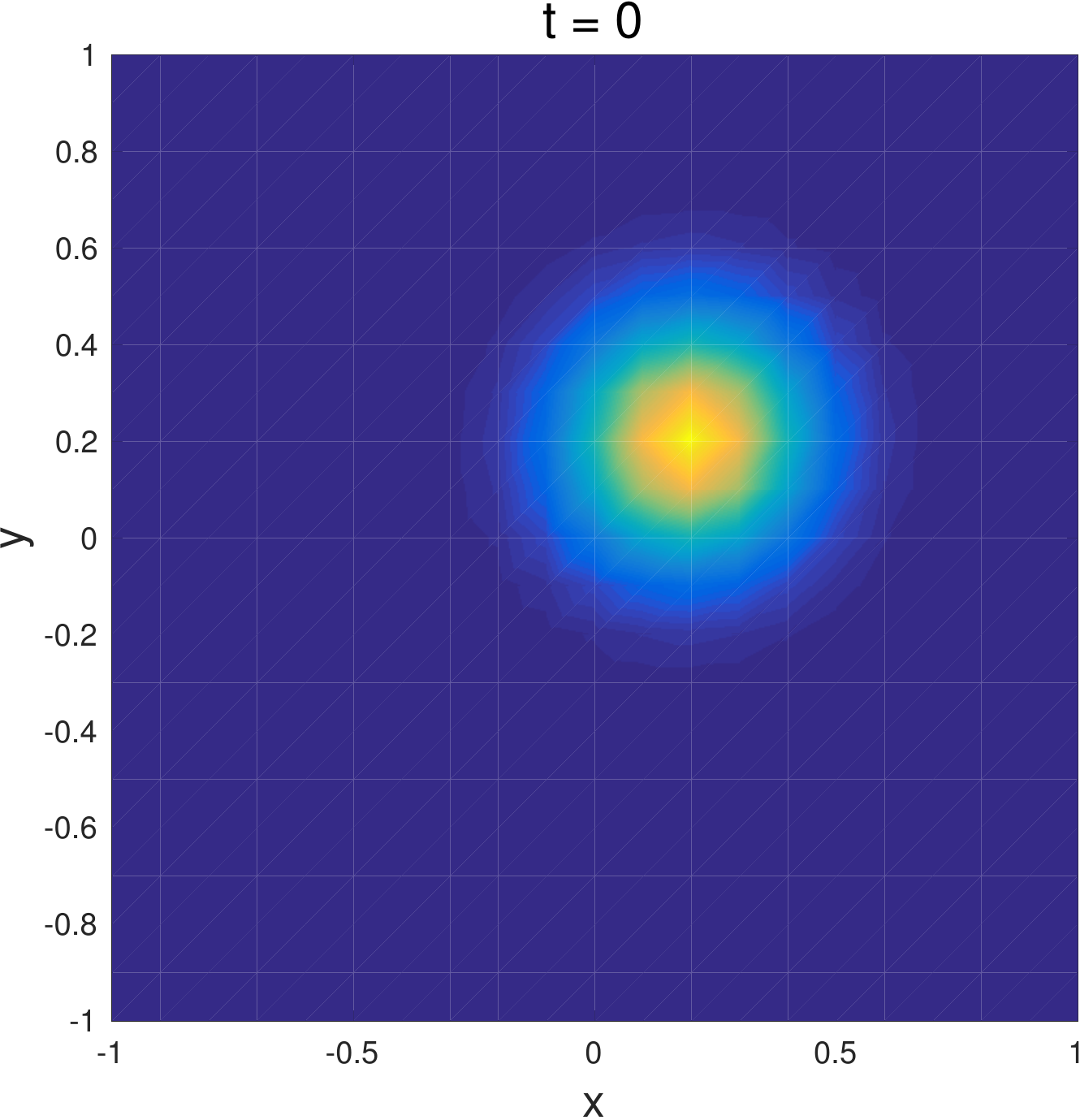}
		} 
		\subfigure[]{ 
			\includegraphics[width=0.3\textwidth]{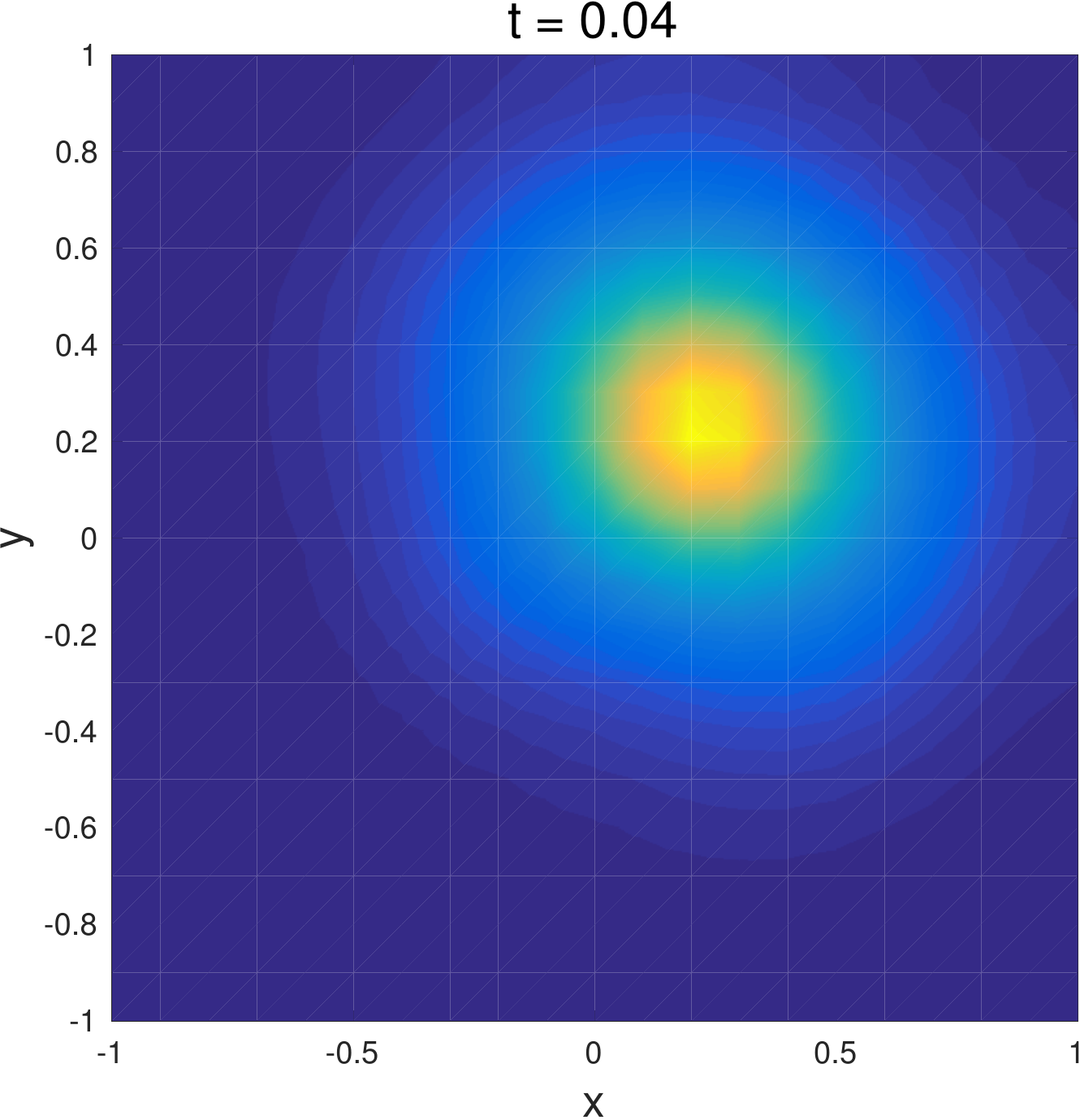} 
		}
		\subfigure[]{ 
			\includegraphics[width=0.3\textwidth]{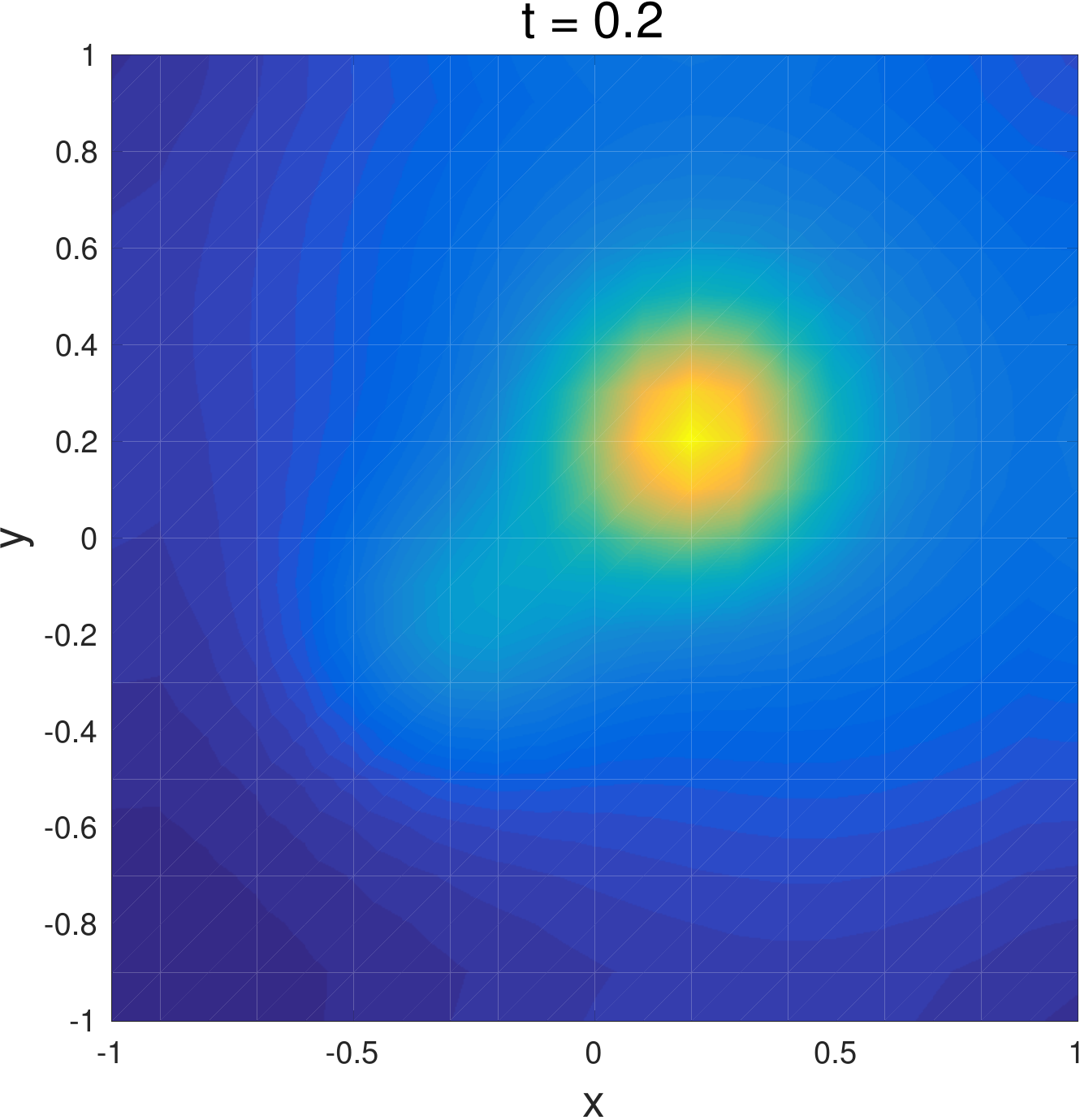} 
		}
		\subfigure[]{ 
			\includegraphics[width=0.3\textwidth]{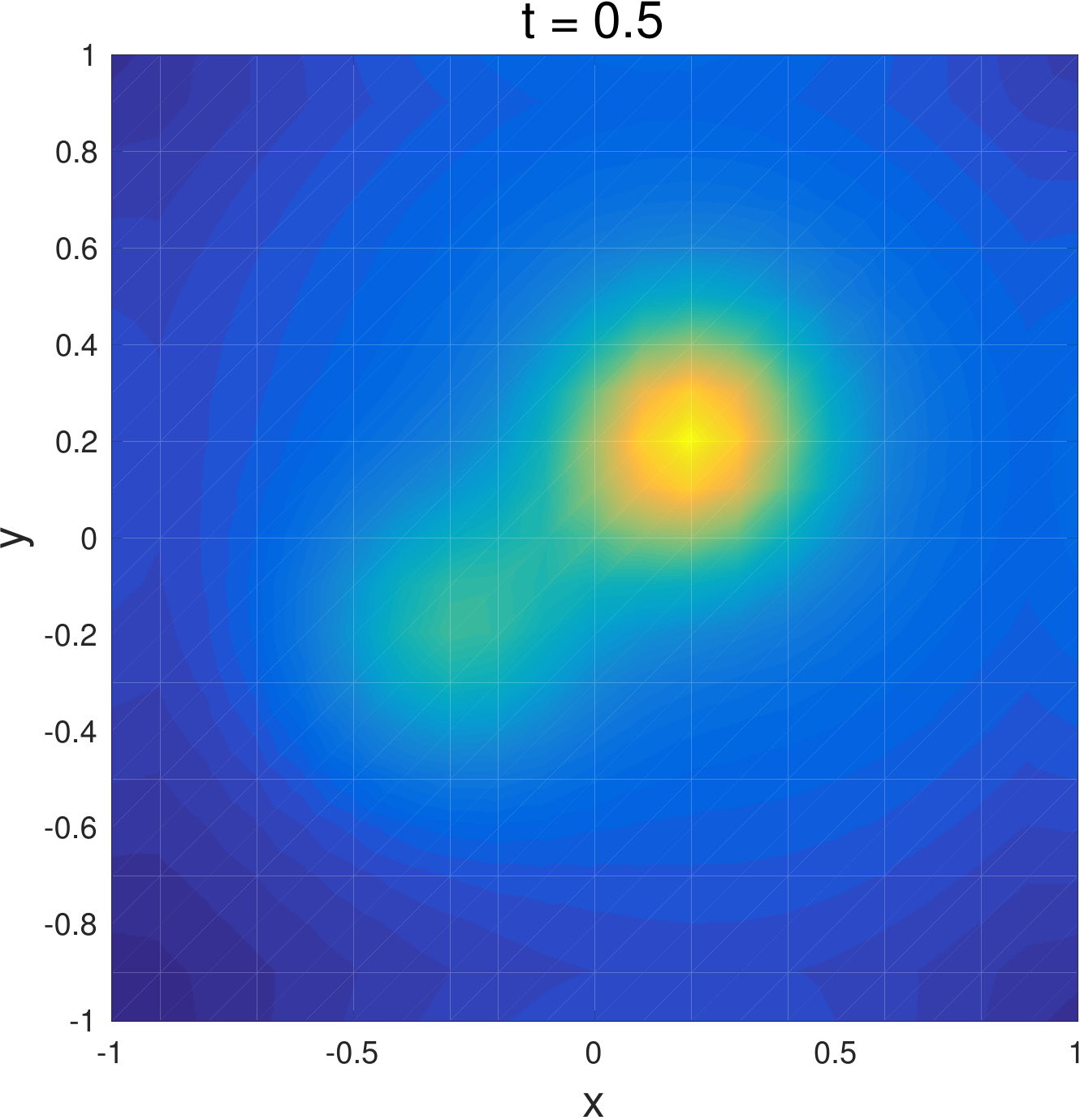}
		} 
		\subfigure[]{ 
			\includegraphics[width=0.3\textwidth]{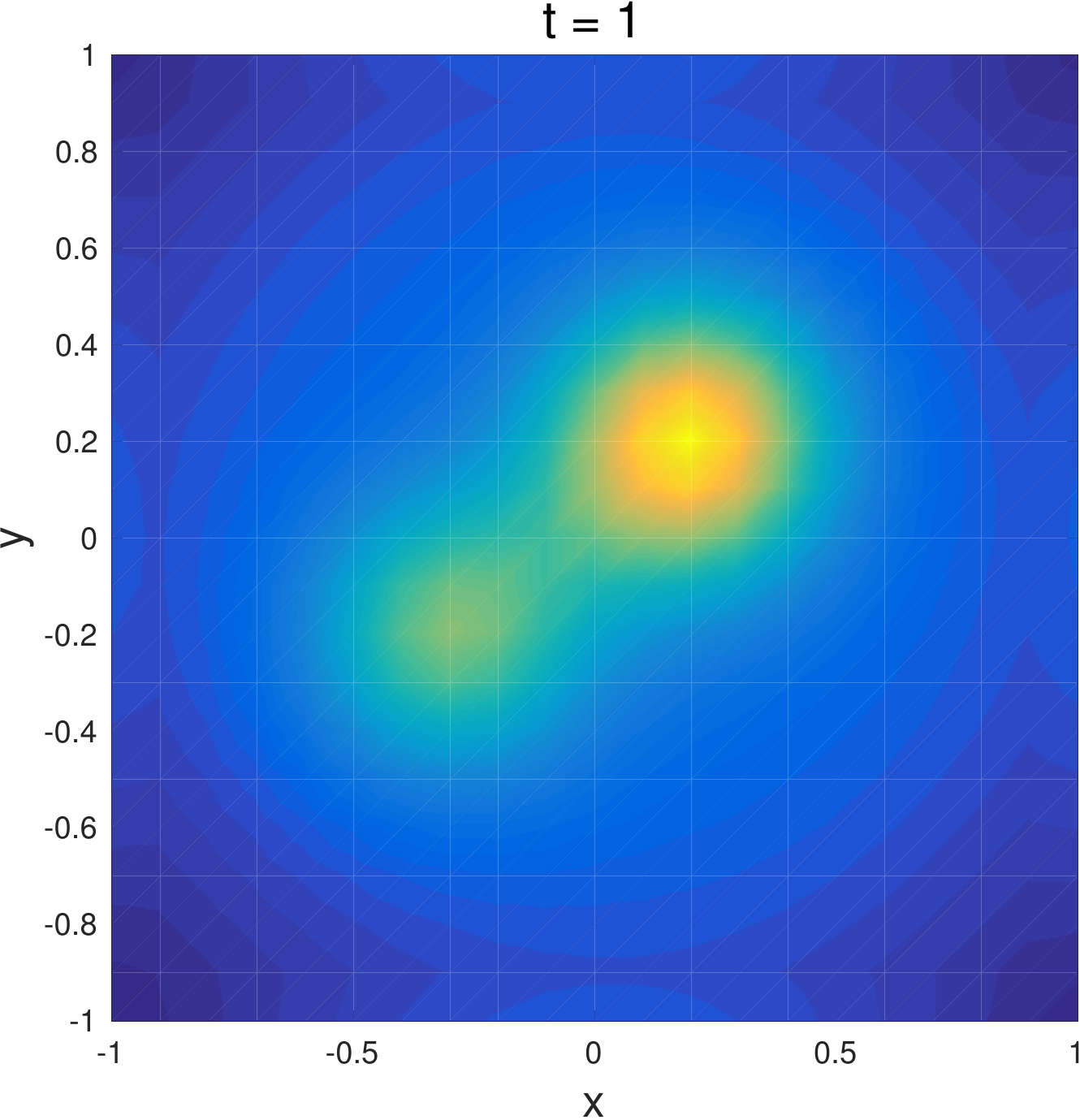} 
		}
		\subfigure[]{ 
			\includegraphics[width=0.3\textwidth]{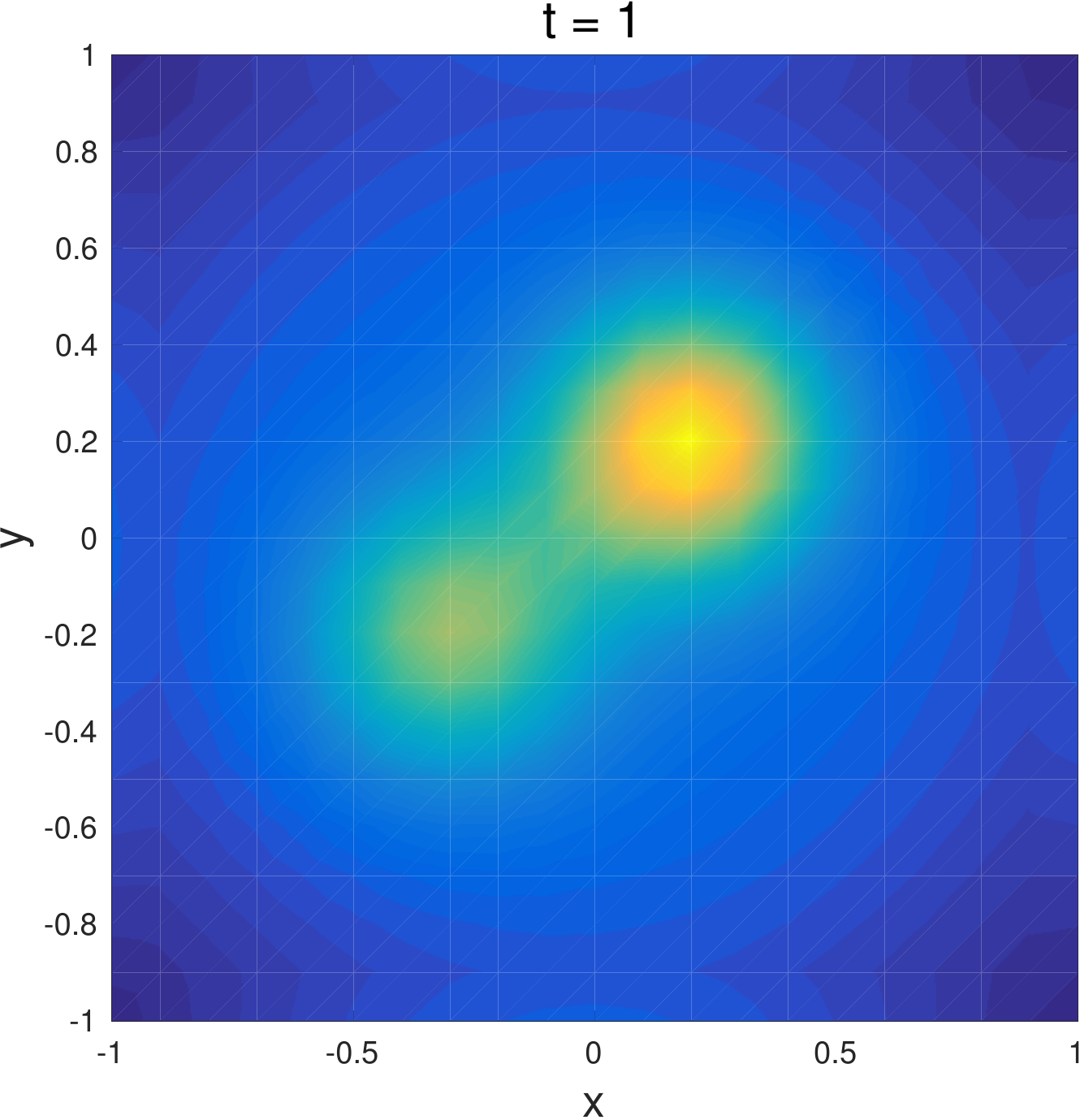} 
		}
		\caption{Multiple Species Example 2 in Two-dimension: the concentration $c_1$ with both the mesh size $\Delta x$ and $\Delta y$ being 0.025, $\Delta t = 0.0004$ and the time $t$ changing from 0 to 3}
		\label{ex4c1}
	\end{figure}
	
	\begin{figure}[htp] 
		\subfigure[]{ 
			\includegraphics[width=0.3\textwidth]{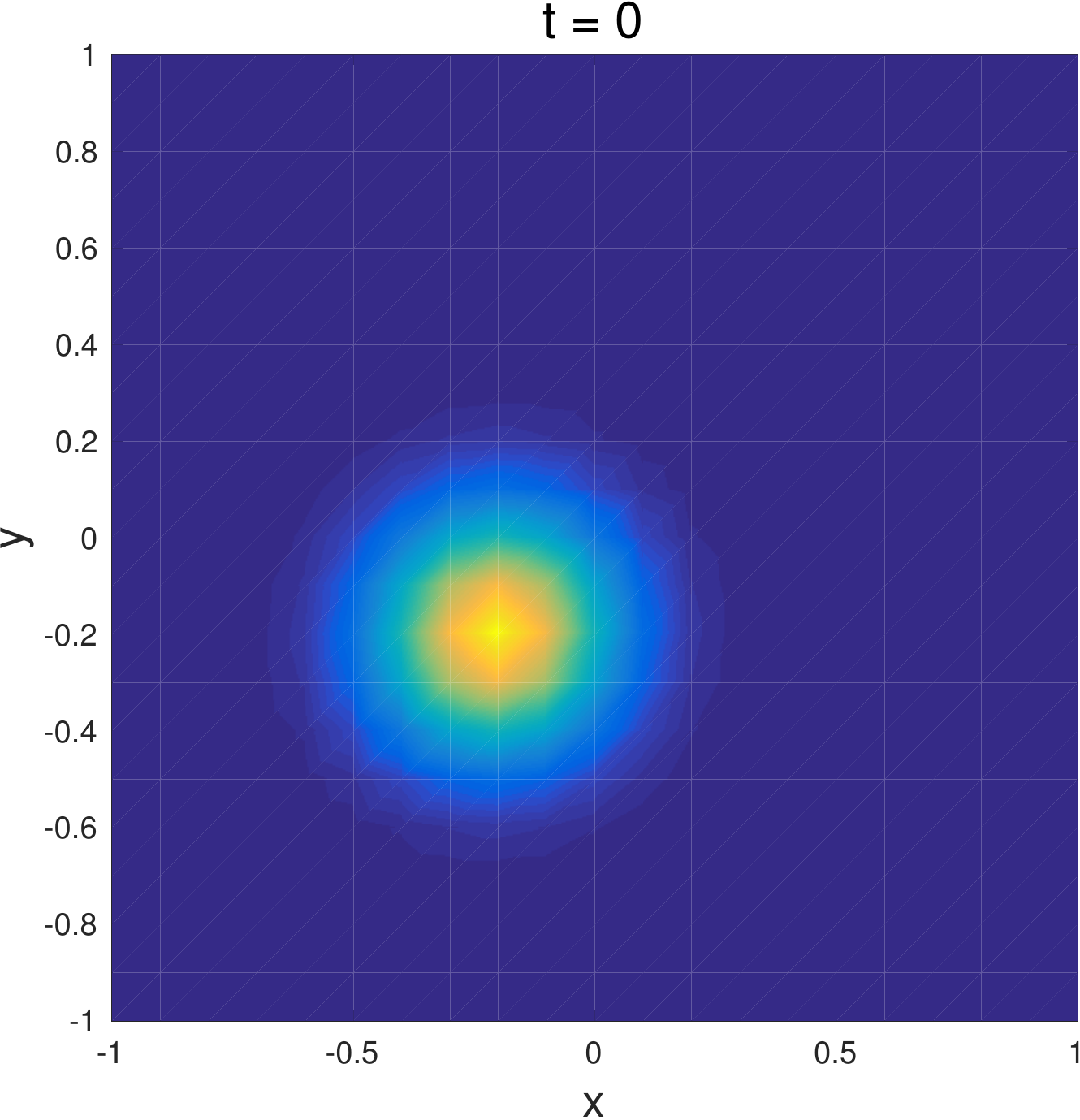}} 
		\subfigure[]{ 
			\includegraphics[width=0.3\textwidth]{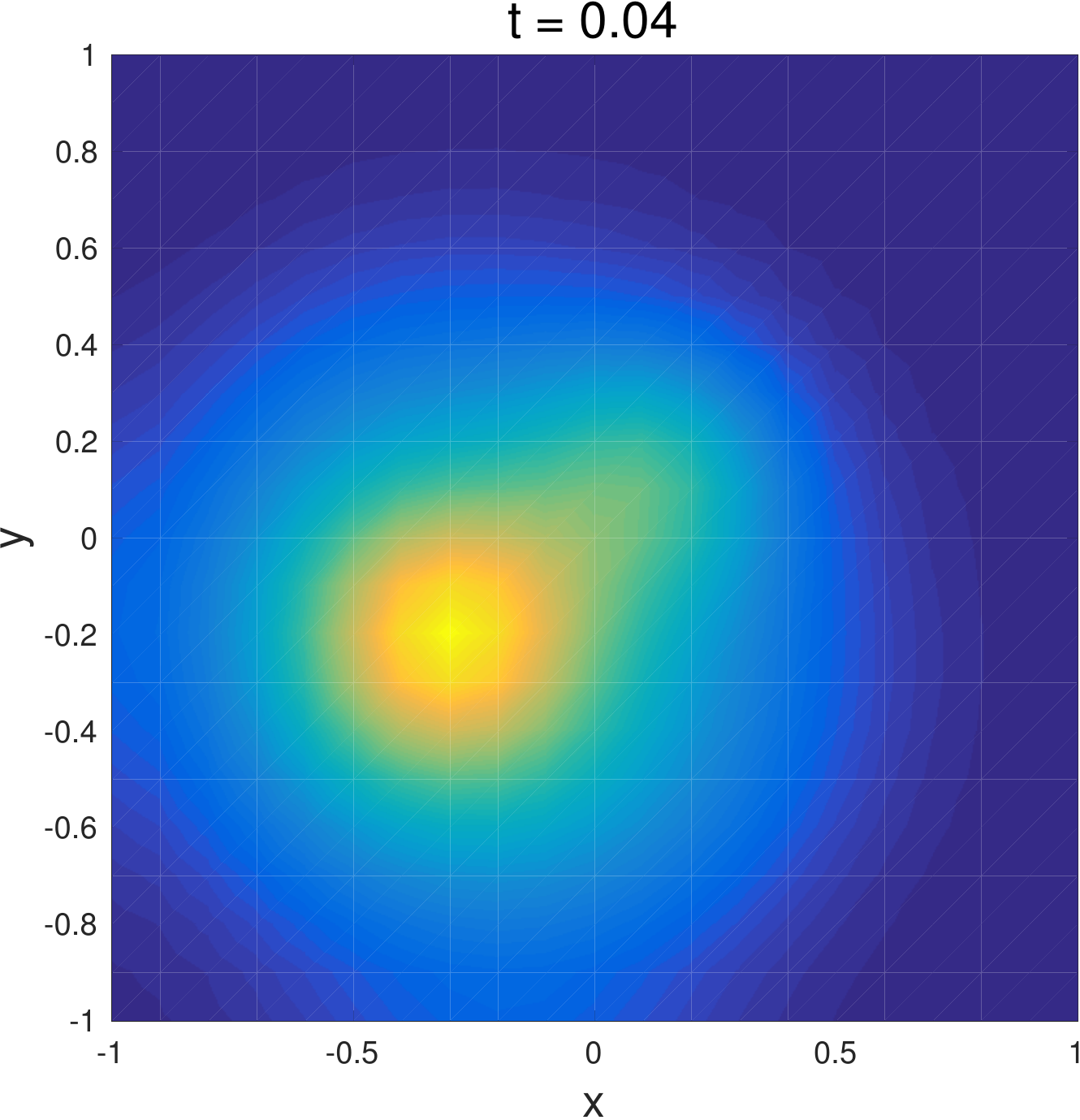} 
		}
		\subfigure[]{ 
			\includegraphics[width=0.3\textwidth]{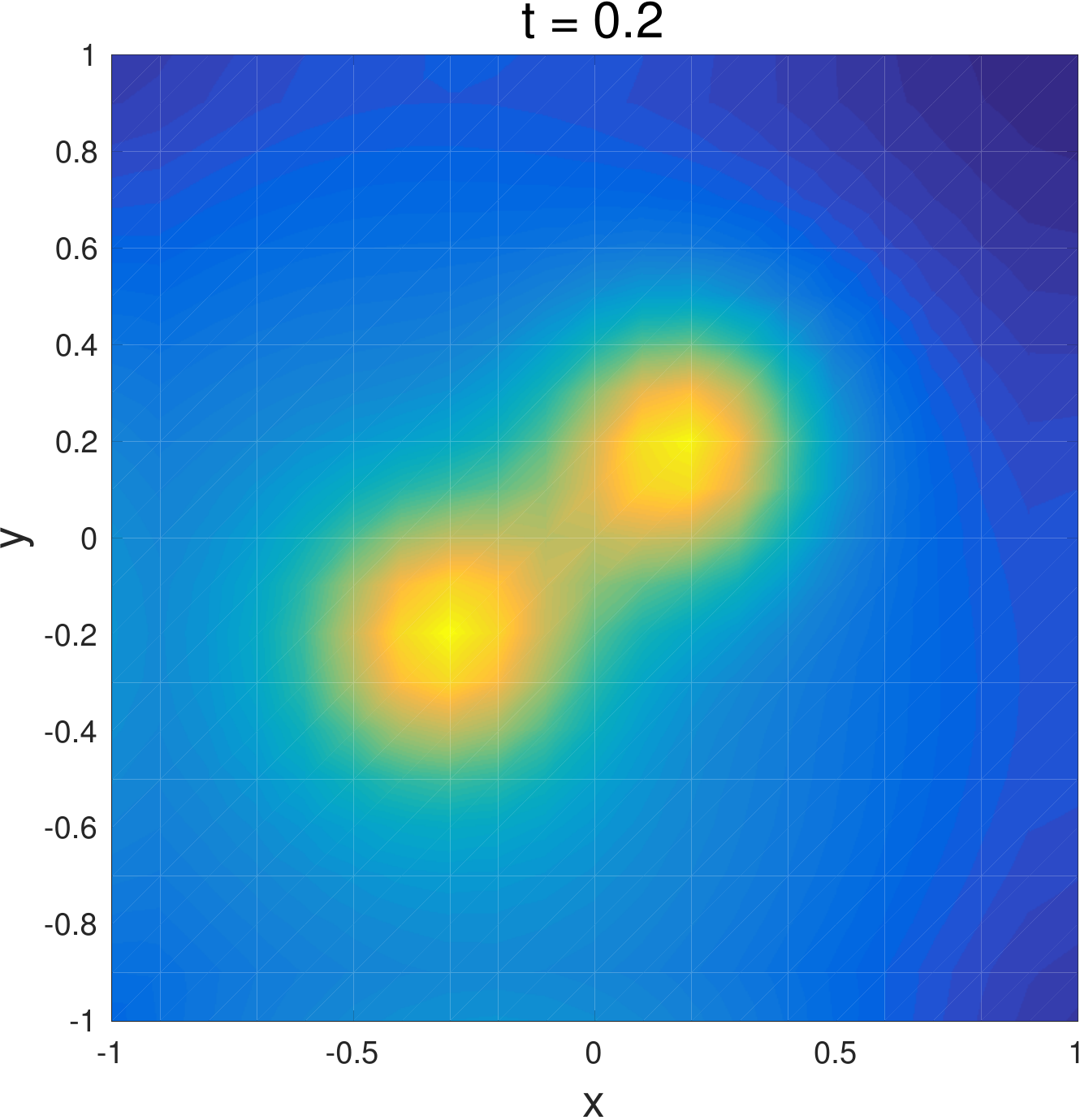}} 
		\subfigure[]{ 
			\includegraphics[width=0.3\textwidth]{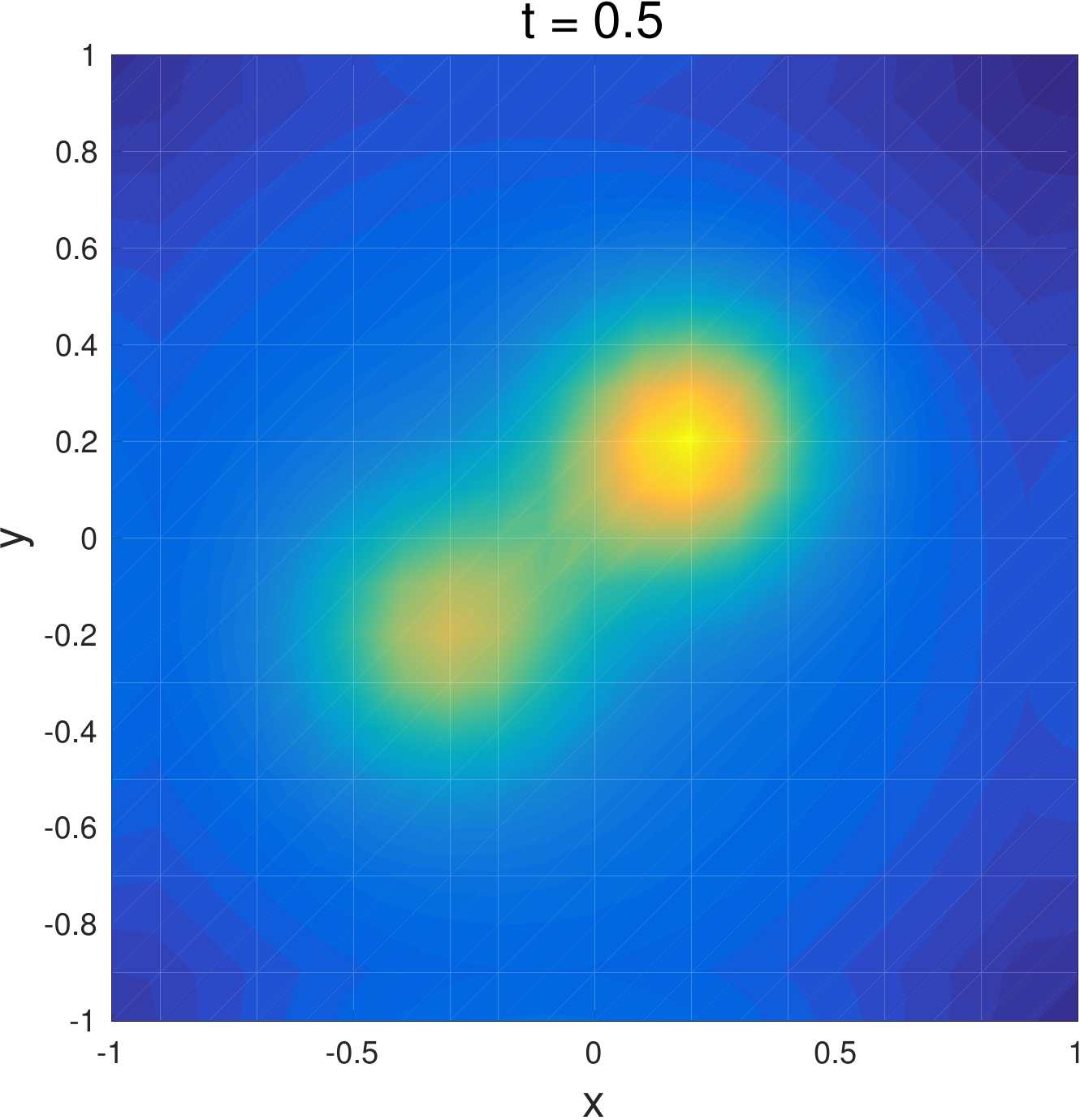} 
		}
		\subfigure[]{ 
			\includegraphics[width=0.3\textwidth]{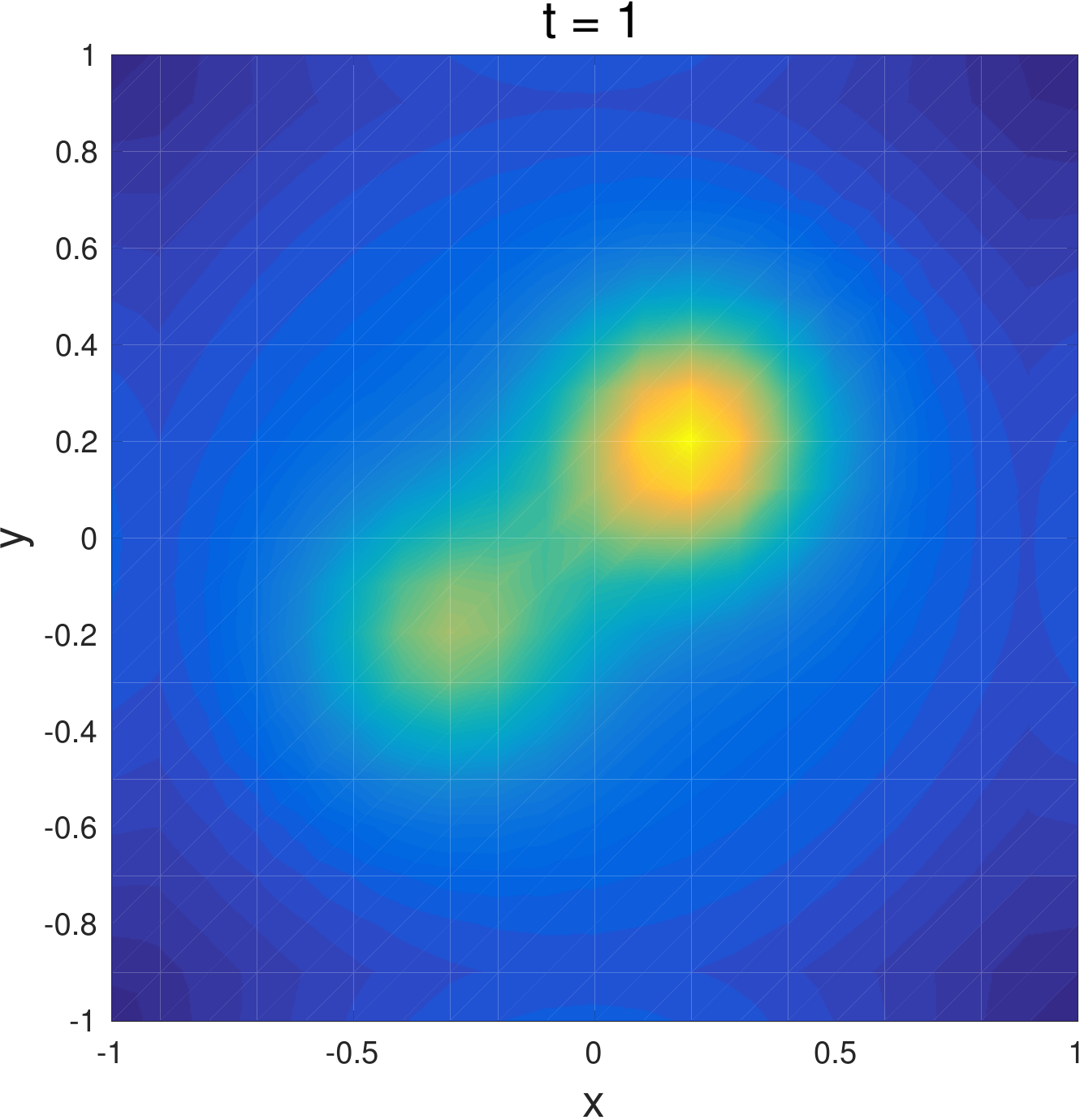}} 
		\subfigure[]{ 
			\includegraphics[width=0.3\textwidth]{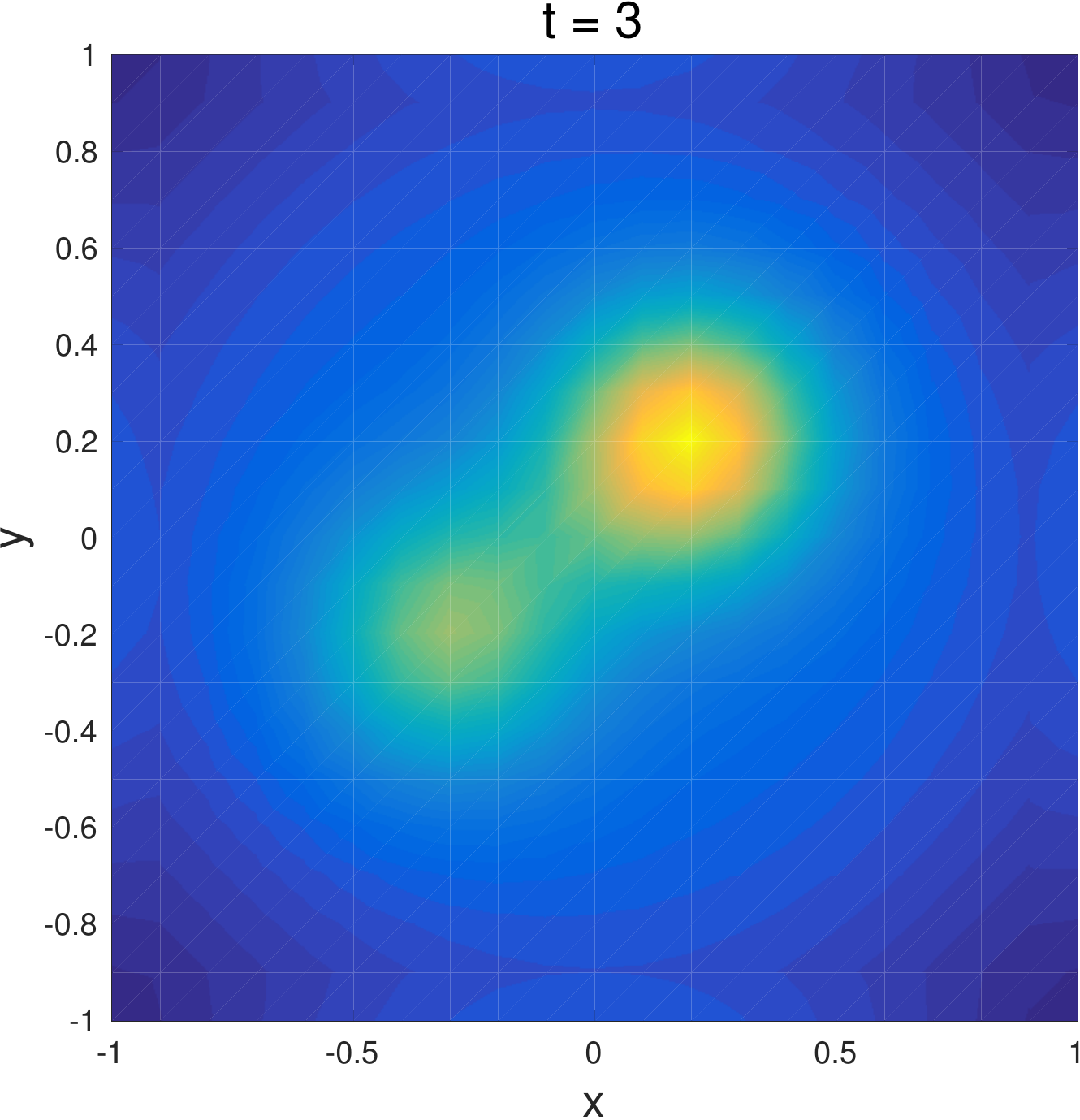} 
		}
		\caption{Multiple Species Example 2 in Two-dimension: the concentration $c_2$ with both the mesh size $\Delta x$ and $\Delta y$ being 0.025, $\Delta t = 0.0004$ and the time $t$ changing from 0 to 3}
		\label{ex4c2}
	\end{figure}
	
	\begin{figure}[htp] 
		\subfigure{
			\includegraphics[width=0.45\linewidth]{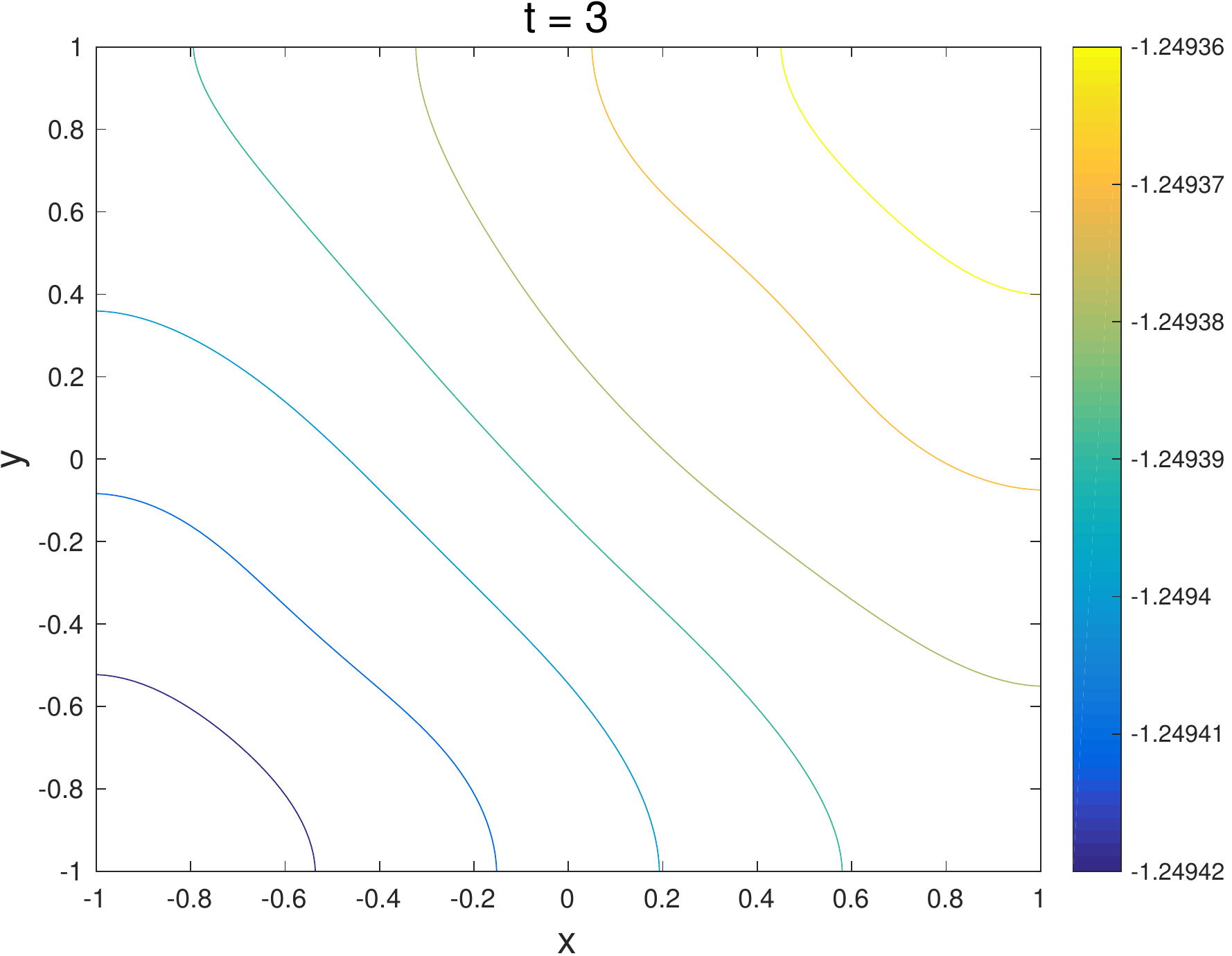}  
		}
		\subfigure{
			\includegraphics[width=0.45\linewidth]{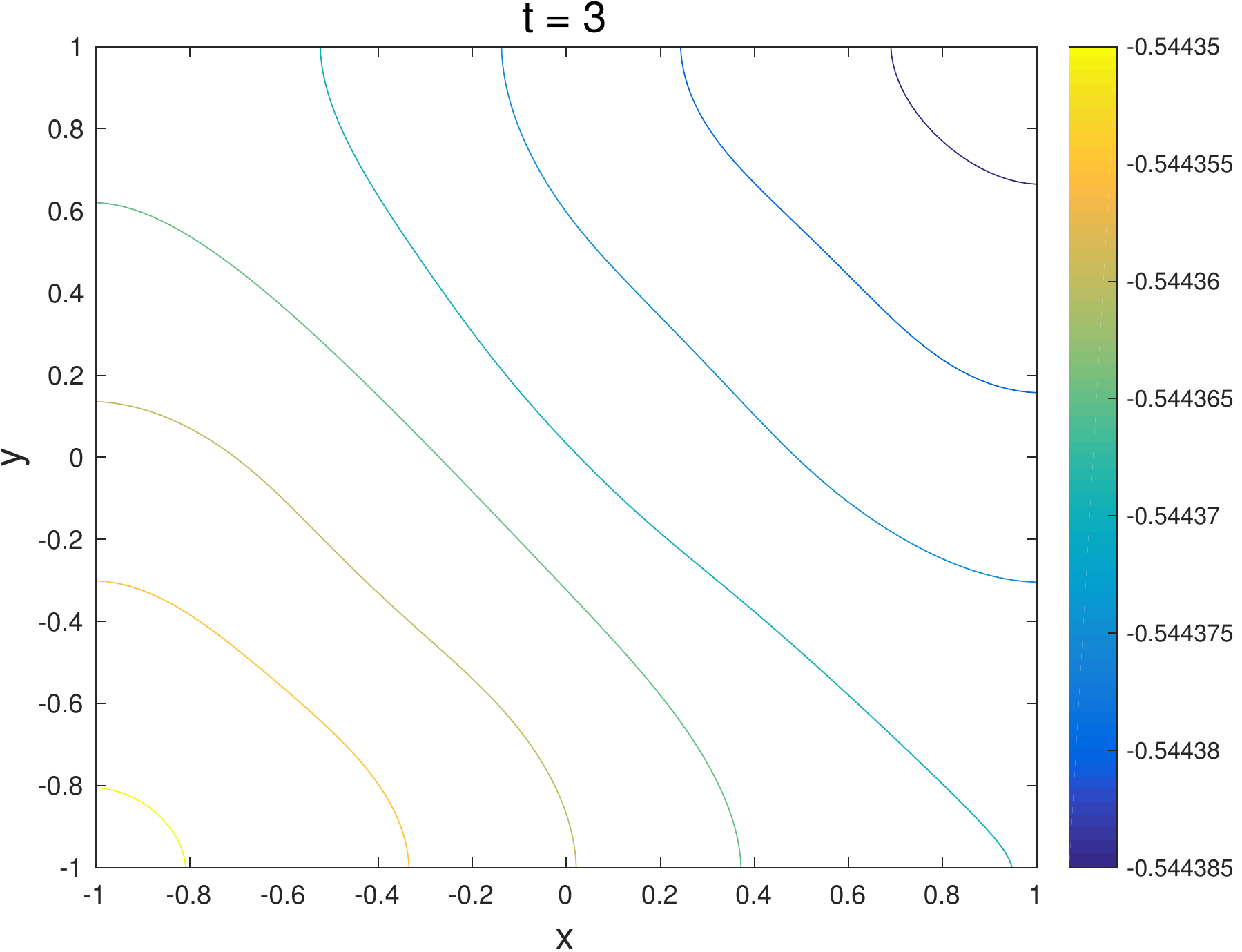}  
		}
		\caption{Multiple Species Example 2 in Two-dimension: the chemical potential of the field model (\ref{model2a})-(\ref{model2c}) equiped with the initial conditions (\ref{iniex4}) at time $t = 3$ with both the mesh size $\Delta x$ and $\Delta y$ being 0.025.}  
		\label{muoft2d2} 
	\end{figure}
	
	\begin{figure}[htp] 
		\subfigure{
			\includegraphics[width=0.5\linewidth]{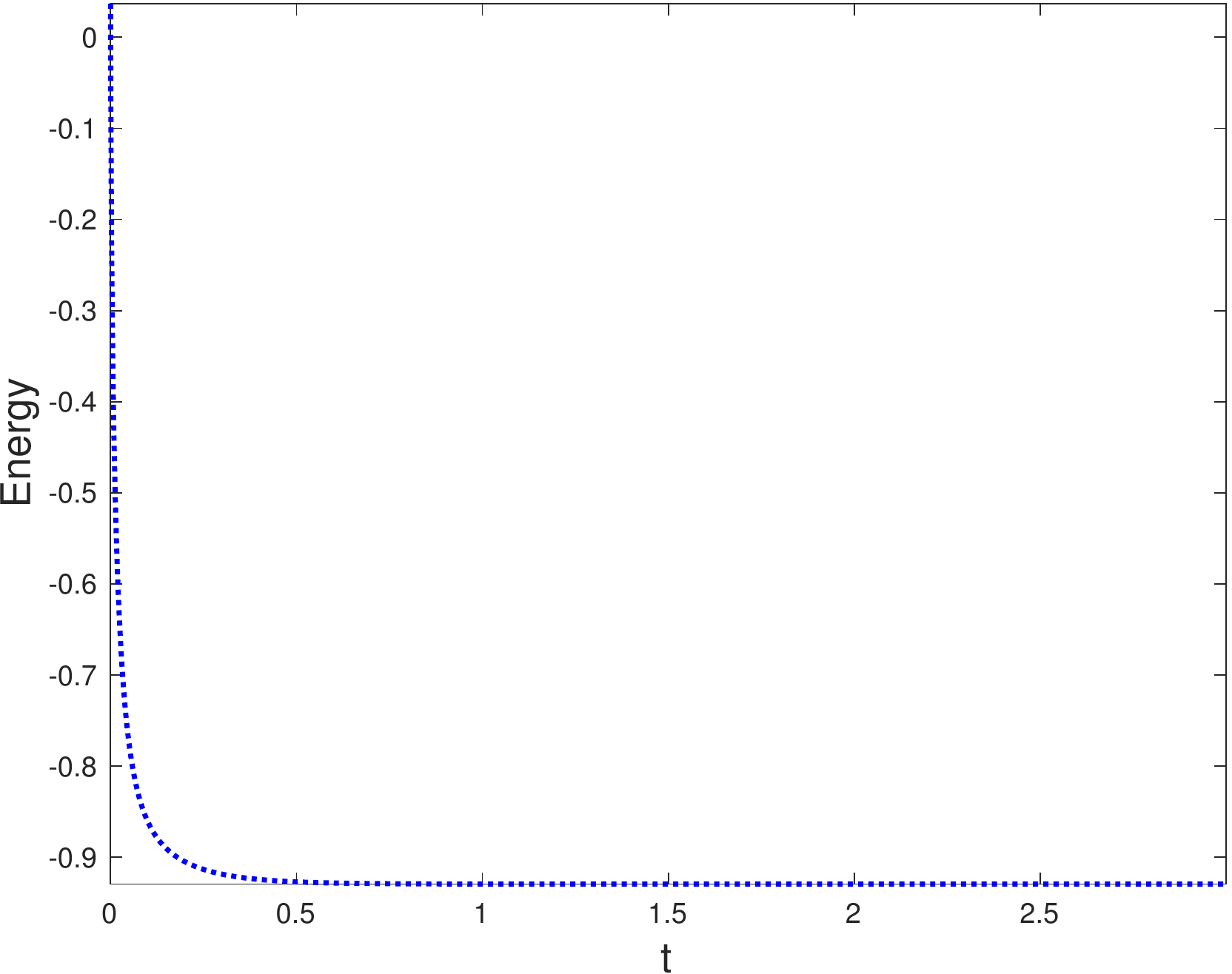}  
		}
		
		\caption{Multiple Species Example 2 in Two-dimension: the time-energy plot of the field model (\ref{model2a})-(\ref{model2c}) equiped with the initial conditions (\ref{iniex4}) with both the mesh size $\Delta x$ and $\Delta y$ being 0.02.}  
		\label{eoft2d2} 
	\end{figure}
	
	\subsection{The Keller-Segel Equations in Two-dimension}
	On domain $\Omega = [-L, L] \times [-L, L], L = 10$, we consider the Keller-Segel system which is a celebrated model for chemotaxis \cite{Keller1970, Keller1971}:
	\begin{equation}\label{ks}
	\begin{aligned}
	&\partial_t c_1(\boldsymbol{x}, t) = \Delta c_1 + \nabla \cdot \left( c_1 \nabla (\mathcal{W}*(c_1 + c_2)) \right),\\
	&\partial_t c_2(\boldsymbol{x}, t) = \Delta c_2 + \nabla \cdot \left( c_2 \nabla (\mathcal{W}*(c_1 + c_2)) \right),\\
	&c_m(\boldsymbol{x},0) = c^0_{m}(\boldsymbol{x}), ~~m=1,2, 
	\end{aligned}
	\end{equation}
	which means that the kernel $\mathcal{K} = 0$, the external potential $V_{\text{ext}} = 0$ and the Newtonian potential $\mathcal{W} = \frac{1}{2 \pi} \ln r$. The first two equations in (\ref{ks}) describe the time evolution of the density $c_m, m = 1, 2,$ where $\Delta c_m$ corresponds to the local diffusion and the nonlocal transport $\nabla \cdot \left( c_m \nabla (\mathcal{W}*(c_1 + c_2)) \right)$ is due to the nonlocal aggregation which is associated with the total mass of the two species. Furthermore, we define the total mass $M_c(t) = \int_{\Omega} (c_1 + c_2)(\boldsymbol{x}, t) \,\mathrm{d} \boldsymbol{x}$. 
	$M_c(t)$ is conserved because the system conserves individual mass, which means
	\begin{equation*}
	M_c(t) = M_c(0) =: M_c^0.
	\end{equation*} 
	Last but not least, we remark that the domain $[-L, L]^2, L = 10,$ can be regarded as an approximation of the whole space $\mathbb{R}^2$ because the attraction part makes the solutions effectively supported in the computational domain. 
	
	For the general single species Keller-Segel Equation in 2D, the mass of the species $M$, determines the long time behavior \cite{Blanchet2006, Liu2016, Perthame2006, He2019}. $M < 8 \pi$ with finite second moment means the unique smooth solution exists globally. On the contrary, $M > 8 \pi$ with finite second moment leads to blowup solutions in finite time.
	Next we explore the solution behavior of the two-species Keller-Segel system (\ref{ks}) which is recently studied in \cite{He2019}.
	
	\subsubsection{Small Initial Data Case}
	Here small initial data means the total mass $M_c^0$ is smaller than $8 \pi$ and thus the system (\ref{ks}) has global solutions \cite{He2019}.
	Set the initial conditions (\ref{model2c}) given by the following equations:
	\begin{equation*}
	\left\{
	\begin{array}{l}
	c_1^0 = \frac{3}{2\pi} \exp \left(-\frac{1}{8} \left(\left(x - 2\right)^2 + y^2\right) \right), \\
	c_2^0 = \frac{3}{2\pi} \exp \left(-\frac{1}{8} \left(\left(x + 2\right)^2 + y^2\right) \right),
	\end{array}
	\right.
	\end{equation*}
	where the total mass $M_c^0 < 8 \pi$. Figure \ref{c1oft_ex3} shows time evolution of the concentration $c_1$ and the steady state of this model at time $t = 50$, where we set the mesh size $\Delta x = 0.16, \Delta t = 0.01$. It's observed that due to the aggregation, the concentrations remain globally bounded.
	
	\begin{figure}[htp] 
		\subfigure[]{ 
			\includegraphics[width=0.3\textwidth]{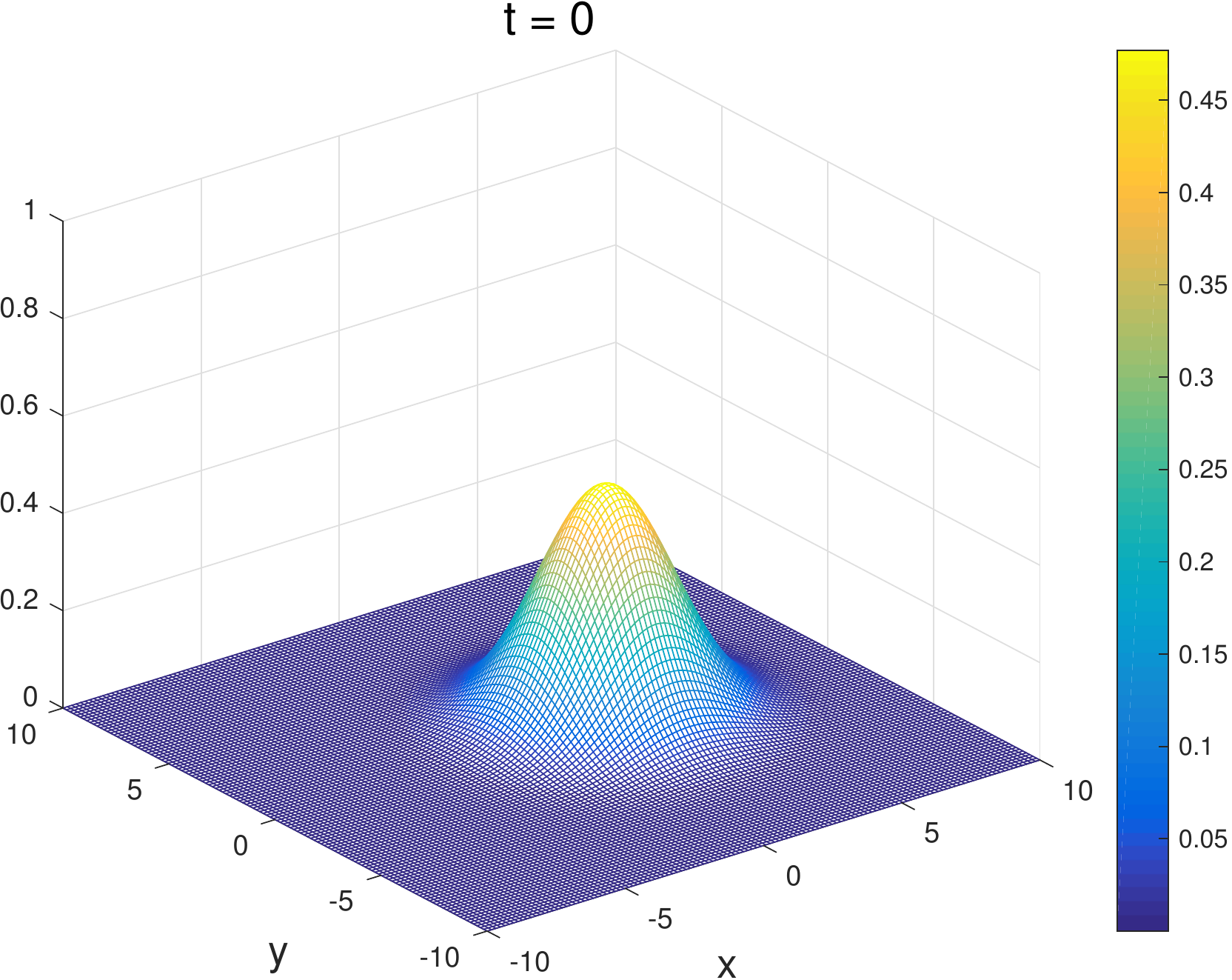}}  
		\subfigure[]{ 
			\includegraphics[width=0.3\textwidth]{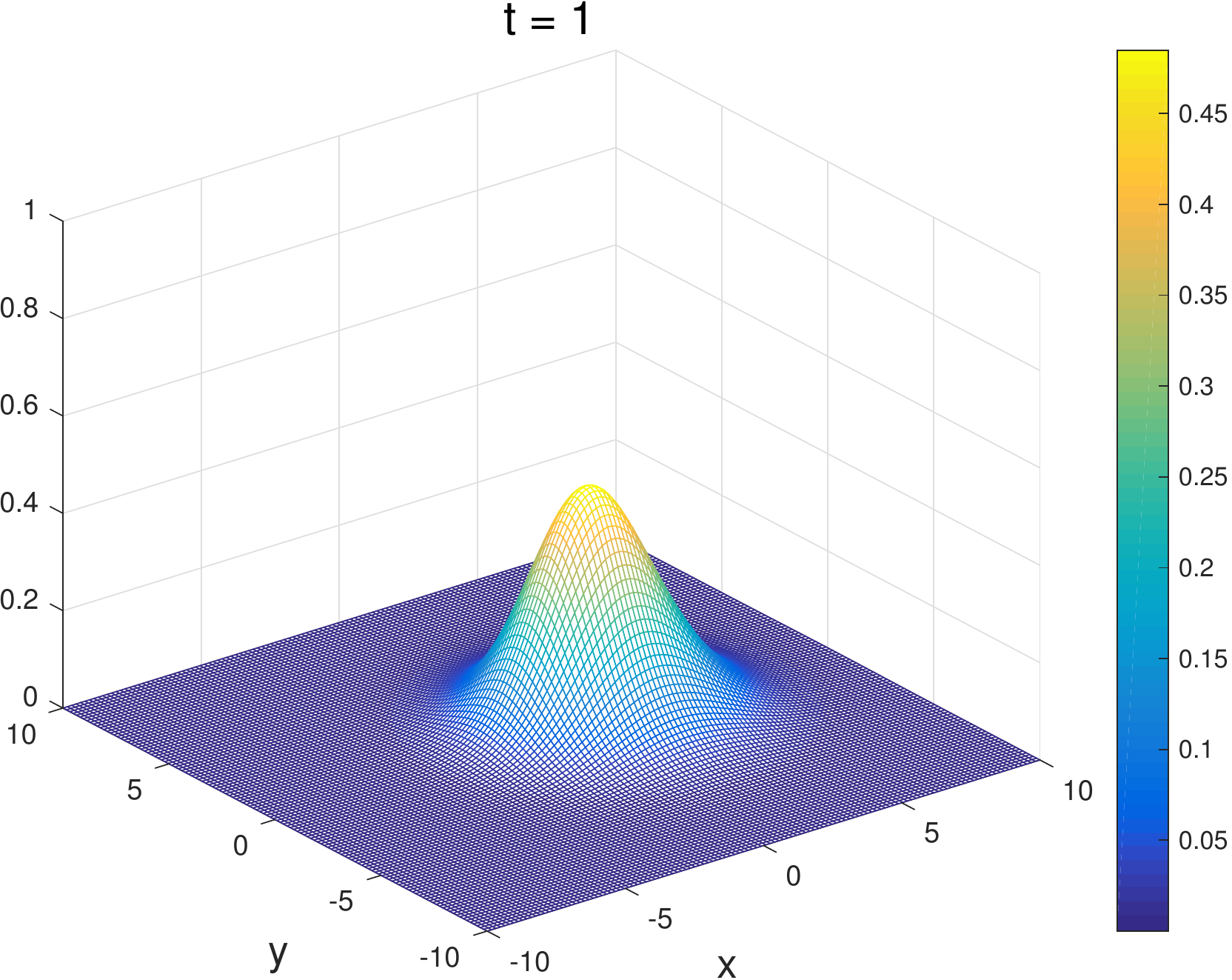}  
		}
		\subfigure[]{ 
			\includegraphics[width=0.3\textwidth]{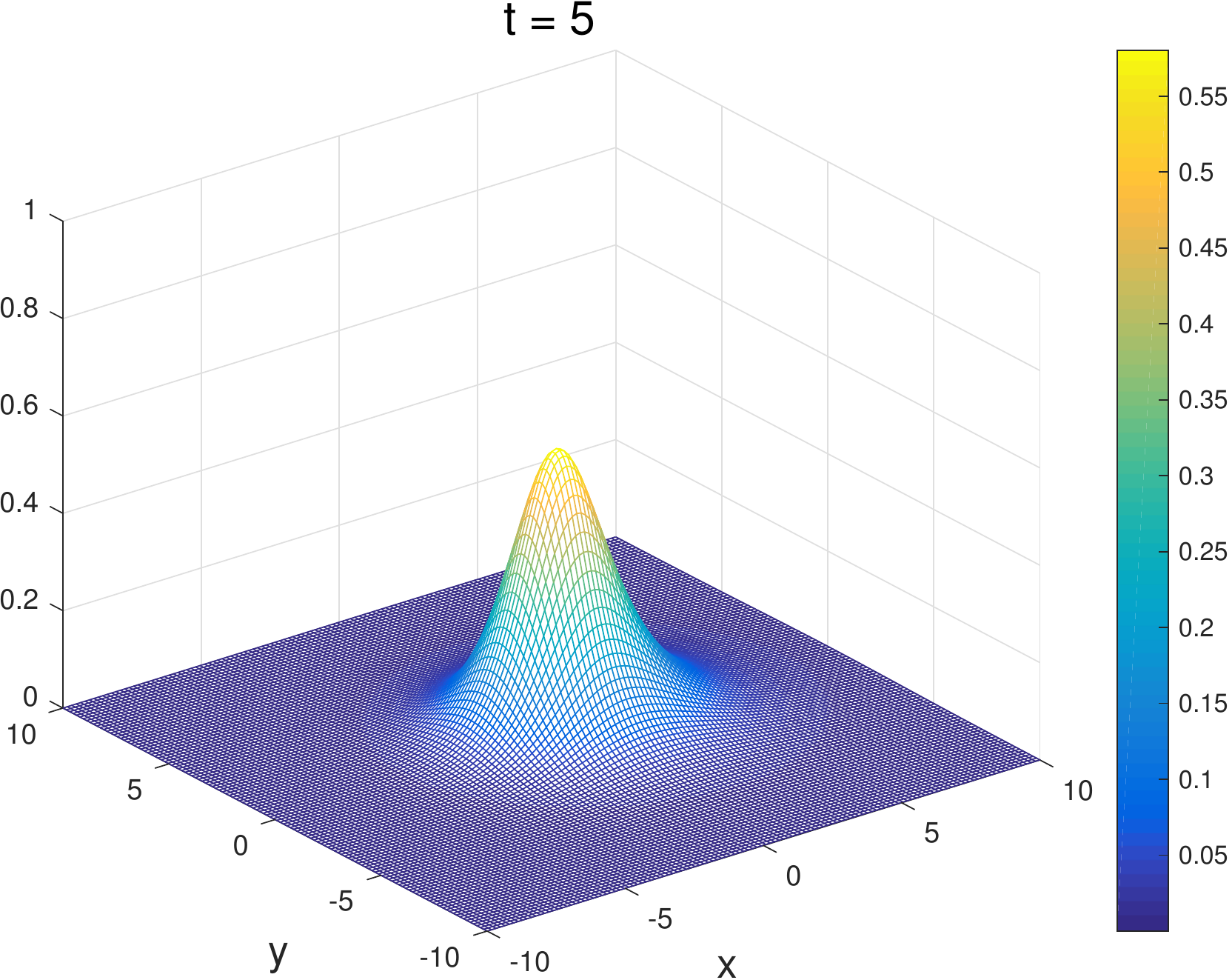}  
		}
		\subfigure[]{ 
			\includegraphics[width=0.3\textwidth]{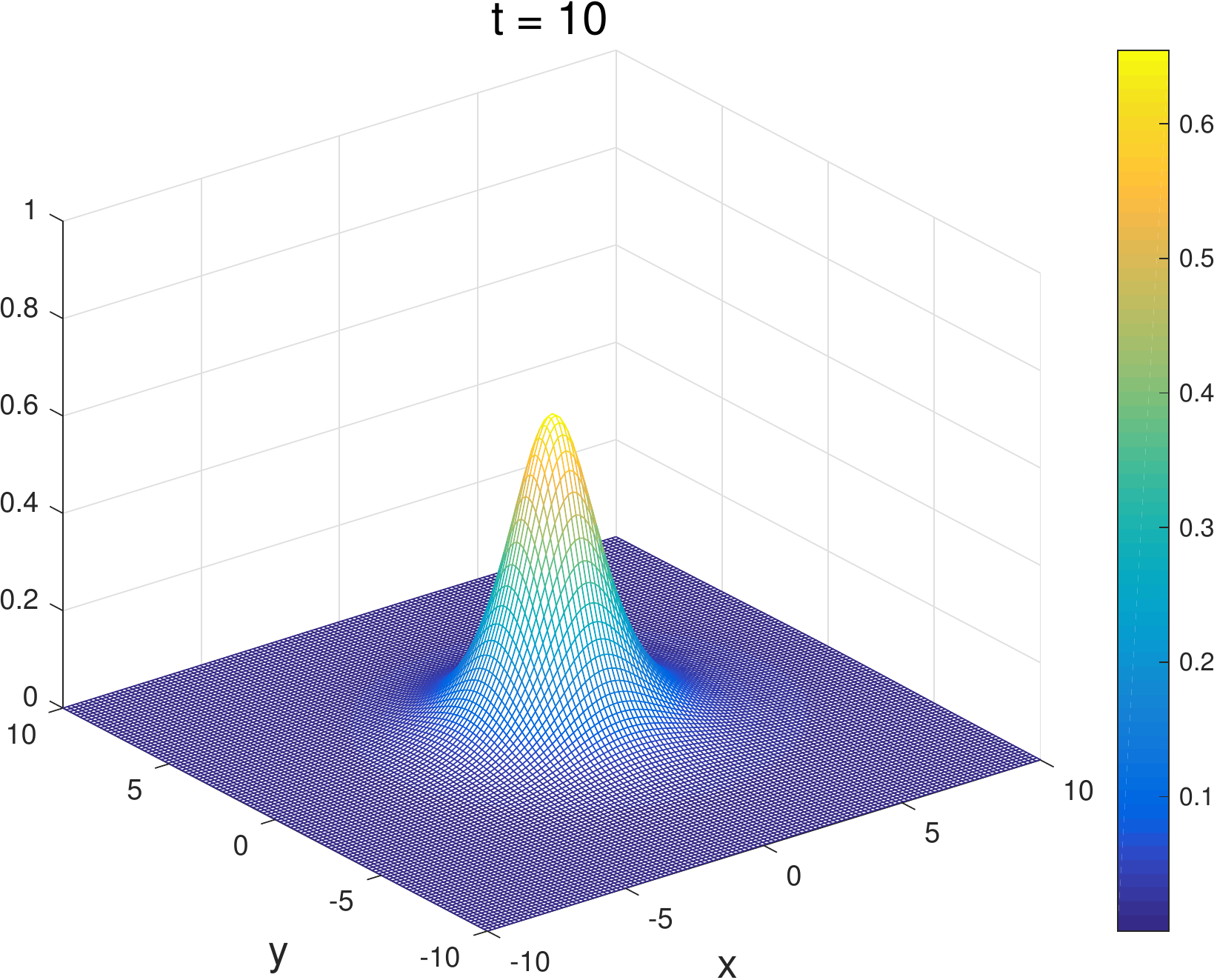} 
		}
		\subfigure[]{ 
			\includegraphics[width=0.3\textwidth]{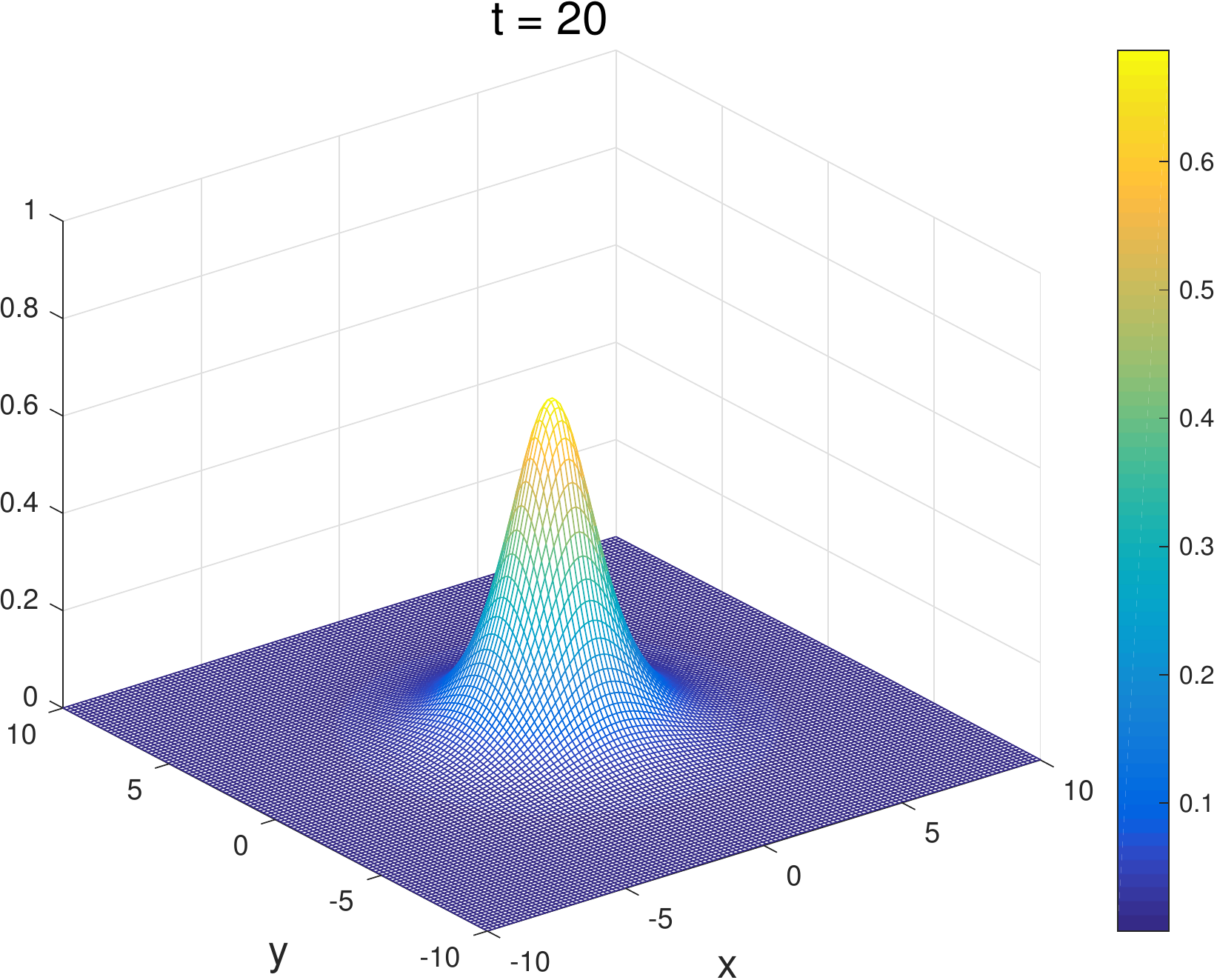}
		}
		\subfigure[]{ 		\includegraphics[width=0.3\textwidth]{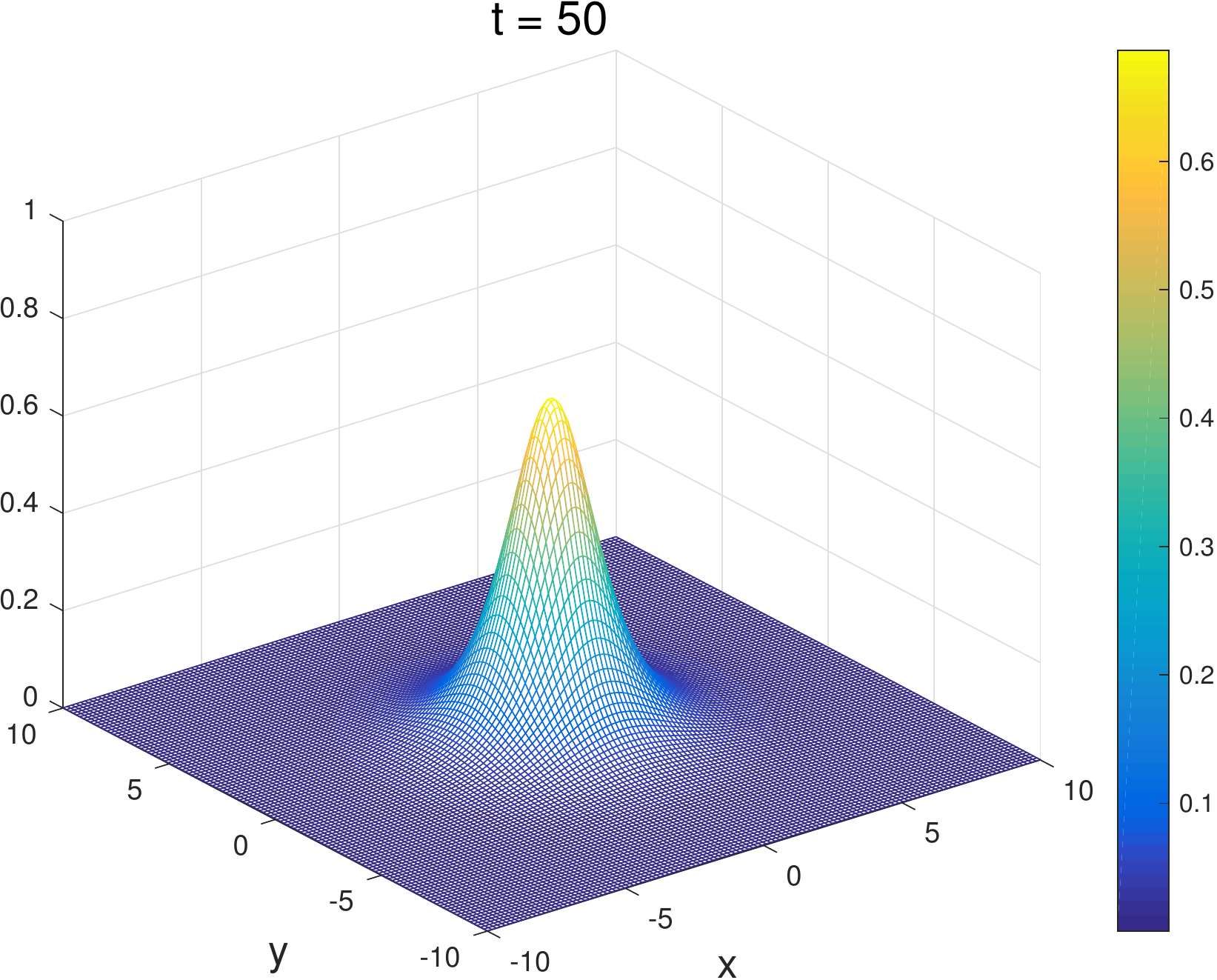}
		}
		\caption{The Keller-Segel Equations in Two-dimension: the space-concentration $c_1$ curves with both the mesh size $\Delta x$ and $\Delta y$ being 0.16 and the time $t$ changing from 0 to 50.}
		\label{c1oft_ex3}
	\end{figure}
	
	
	\subsubsection{Large initial data case}
	
	Here large initial data means the total mass $M_c^0$ is larger than $8 \pi$ and thus the solution blows up in finite time \cite{He2019}. In this part, we show two examples.
		
	{\bf Example 1: }At first, we set the initial conditions (\ref{model2c}) given by the following equations:
	\begin{equation}
	\label{iniex31}
	\left\{
	\begin{array}{l}
	c_1^0 = 4 \exp \left(-\frac{1}{4} (x^2 + y^2) \right), \\
	c_2^0 = 8 \exp \left(-\frac{1}{2} (x^2 + y^2) \right).
	\end{array}
	\right.
	\end{equation}
	Fig \ref{p1} shows the concentration $c_1$ and $c_2$ at time $t = 0, 0.05, 0.1, 0.15$ and time evolution of the $l^{\infty}$ norm of them, where we set the mesh size $\Delta x = 0.08, \Delta t = 0.01$. It's observed that all the solutions blow up in finite time and $c_2$ blows up faster than $c_1$ under the initial condition (\ref{iniex31}).
	
	\begin{figure}[htp] 
		\subfigure{
			\includegraphics[width=0.85\linewidth]{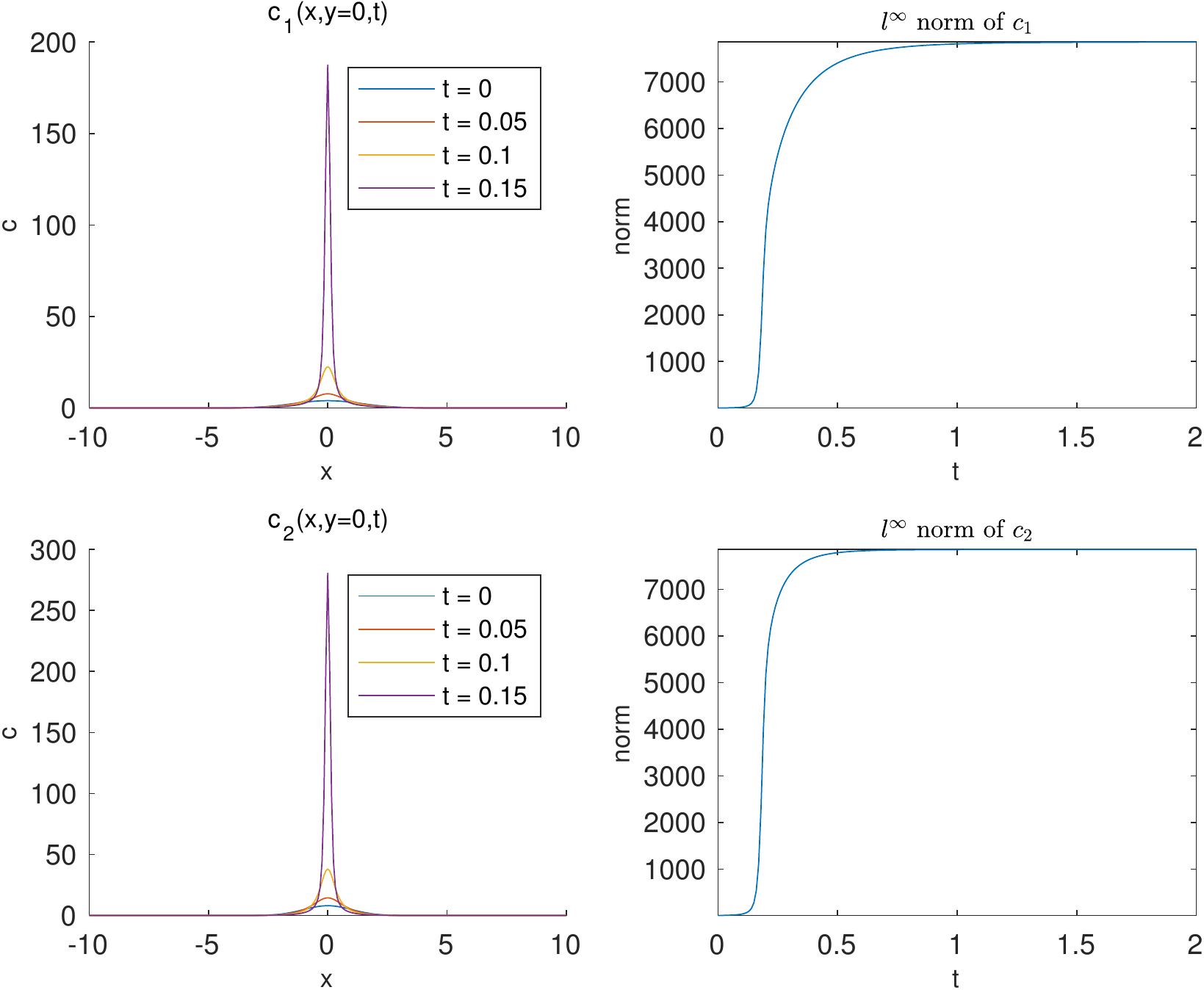}  
		}	
		\caption{The Keller-Segel Equations in Two-dimension: the concentration $c_1$ and $c_2$ at time $t = 0, 0.05, 0.1, 0.15$ and time evolution of the $l^{\infty}$ norm of them as for the initial condition (\ref{iniex31}).}  
		\label{p1} 
	\end{figure}
	
	{\bf Example 2: }Secondly, we set the initial conditions (\ref{model2c}) given by the following equations:
	\begin{equation}
	\label{iniex32}
	\left\{
	\begin{array}{l}
	c_1^0 = 12 \exp \left(-\left(\left(x-2\right)^2 + y^2\right) \right), \\
	c_2^0 = 12 \exp \left(-\left(\left(x+2\right)^2 + y^2\right) \right).
	\end{array}
	\right.
	\end{equation}
	\begin{figure}[htp] 
		\subfigure{
			\includegraphics[width=\linewidth]{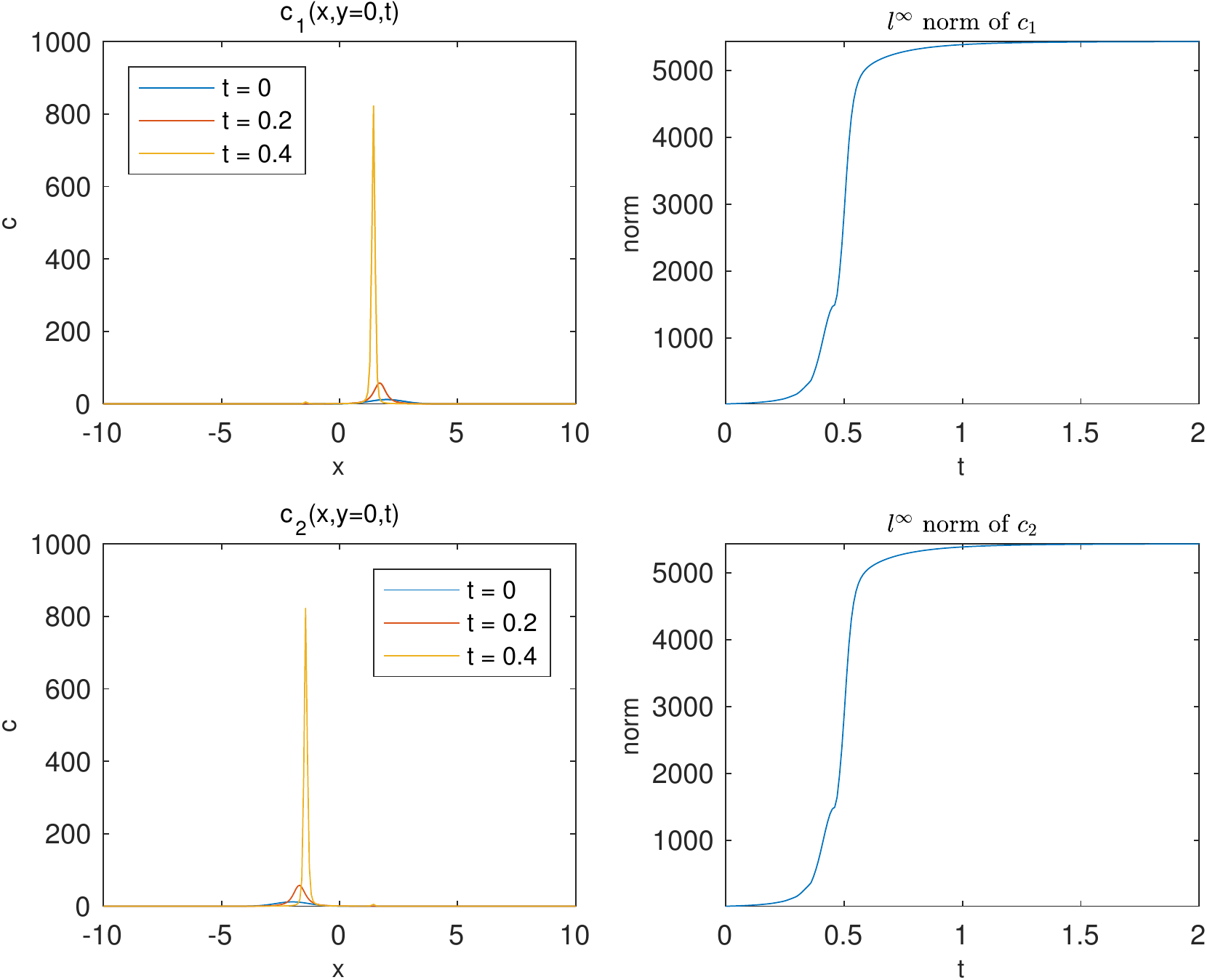}
		}	
		\caption{The Keller-Segel Equations in Two-dimension: the concentration $c_1$ and $c_2$ at time $t = 0, 0.2, 0.4$ and time evolution of the $l^{\infty}$ norm of them as for the initial condition (\ref{iniex32}).}  
		\label{p2} 
	\end{figure}
Fig \ref{p2} shows the concentration $c_1$ and $c_2$ at time $t = 0, 0.2, 0.4$ and time evolution of the $l^{\infty}$ norm of them, where we set the mesh size $\Delta x = 0.08, \Delta t = 0.01$. it is seen from Fig \ref{p2} that all the solutions blow up in finite time under the initial condition (\ref{iniex32}) and due to the nonlocal attraction, the two density functions move closer to each other before forming singular solutions.

	\section{Conclusion}
	
	In this paper, we focus on the symmetric form of the nonlinear and nonlocal parabolic model for multi-species ionic fluids (\ref{sym1a})-(\ref{sym1c}) and develop an unconditionally stable finite volume scheme. Our scheme is of second order in space, first order in time and preserves the analytical properties, such as positivity preservation, mass conservation and free energy dissipation. Furthermore, our scheme involves accurate and efficient fast algorithm on the convolution terms with singular but integrable kernels. 
	And the second order spatial accuracy of the fast convolution algorithm inherits from the finite-volume discretization of the density.
	We also provide series of numerical experiments to demonstrate the properties, such as unconditional stability, numerical convergence, energy dissipation, the finite size effect, the complexity in computing the convolution with singular kernels, the concentration of ions at the boundary and the blowup phenomenon of the Keller-Segel equations.
	
	\section{Appendix}
	In this appendix, we provide some proofs of theorems given in Section 2.
	
	\subsection{Proof of Theorem \ref{thm1b}} \label{app1}
	\begin{proof}
		The finite volume scheme (\ref{bsys2})-(\ref{bsys3}) with (\ref{sysodeb}) and (\ref{sysfluxb}) can be written in the following form 
		{\small
			\begin{equation}
			\label{bpp}
			\bar{c}_{m, j}^{n + 1} +  A_{m, j + \frac{1}{2}}^{n} \left\{ \frac{\bar{c}_{m, j }^{n + 1}}{\exp\{- f_{m, j}^{n} \}} - \frac{\bar{c}_{m, j + 1}^{n + 1}}{\exp\{-f_{m, j + 1}^{n} \}} \right\} + A_{m, j - \frac{1}{2}}^{n} \left\{ \frac{\bar{c}_{m, j }^{n + 1}}{\exp\{-f_{m, j}^{n} \}} - \frac{\bar{c}_{m, j - 1}^{n + 1}}{\exp\{-f_{m, j - 1}^{n} \}} \right\} = \bar{c}_{m, j}^{n},
			\end{equation}
		}
		where the coefficient
		\begin{equation*}
		A_{m, j + \frac{1}{2}}^{n} =  \exp\{-f_{m, j + \frac{1}{2}}^n\} \Delta t/ (\Delta x)^2.
		\end{equation*}
		The positivity preserving property can be proved by a contradiction argument.
		
		Assume that $\bar{c}_{m, j}^{n} \geqslant 0$ for any $j$ and $m$, but $\bar{c}_{m, j}^{n + 1} < 0$ for some $j$ and $m$. Assume that $\frac{\bar{c}_{m, j }^{n + 1}}{\exp\{- f_{m, j}^{n} \}}$ takes the minimum at $j = j_0$ and $m = m_0$. Taking $j = j_0$ and $m = m_0$ in equation (\ref{bpp}) we find that the left side of the equation is negative while the right side is non-negative. Thus we conclude that the explicit-implicit scheme preserves positivity.
	\end{proof}
	
	\subsection{Proof of Theorem \ref{thm2b} }  \label{app2}
	\begin{proof}
		We only need to prove that 
		\begin{equation}
		\begin{aligned}
		\dfrac{\mathrm{d}}{\mathrm{d} t} E_{\Delta}(t) &= \dfrac{\mathrm{d}}{\mathrm{d} t} \left[ \Delta x \sum_{m = 1}^M \sum_{j} \bar{c}_{m, j} \log \frac{\bar{c}_{m, j}}{\exp\{- \frac{1}{2} f_{m, j}\}} \right] \\
		&= \Delta x \sum_{m = 1}^M \sum_{j} \left[ \log \dfrac{\bar{c}_{m, j}}{\exp\{ -\frac{1}{2} f_{m, j} \}} \dfrac{\mathrm{d} }{\mathrm{d} t} \bar{c}_{m, j} + \bar{c}_{m, j} \dfrac{\mathrm{d}}{\mathrm{d} t} \log \dfrac{\bar{c}_{m, j}}{\exp\{ -\frac{1}{2} f_{m, j} \}} \right] \\
		&= \Delta x \sum_{m = 1}^M \sum_{j} \left[ \log \dfrac{\bar{c}_{m, j}}{\exp\{ -\frac{1}{2} f_{m, j} \}} \dfrac{\mathrm{d} }{\mathrm{d} t} \bar{c}_{m, j} + \dfrac{\mathrm{d}}{\mathrm{d} t} \bar{c}_{m, j} + \frac{1}{2} \bar{c}_{m, j} \dfrac{\mathrm{d}}{\mathrm{d} t}  f_{m, j} \right] \\
		&= \Delta x \sum_{m = 1}^M \sum_{j} \left\{ \log \dfrac{\bar{c}_{m, j}}{\exp\{ -f_{m, j} \}} + 1 \right\} \dfrac{\mathrm{d}}{\mathrm{d} t} \bar{c}_{m, j} \\
		&= \Delta x \sum_{m = 1}^{M} \sum_{j}\mu_{m, j} \dfrac{\mathrm{d}}{\mathrm{d} t} \bar{c}_{m, j}.
		\end{aligned}
		\end{equation}
		The fourth equality holds because of $T^{\mathcal{K}}_{i-j} = T^{\mathcal{K}}_{j-i}$ and $T^{\mathcal{W}}_{i-j} = T^{\mathcal{W}}_{j-i}$.
		
		According to (\ref{bsys2}), we have 
		
		\begin{equation}
		\dfrac{\mathrm{d}}{\mathrm{d} t} E_{\Delta}(t) = - \sum_{m = 1}^{M} \sum_{j} \left\{ \left\{ \log \dfrac{\bar{c}_{m, j}}{\exp\{ - f_{m, j} \}} + 1 \right\} (F_{m, j + \frac{1}{2}} - F_{m, j - \frac{1}{2}}) \right\}.
		\end{equation}
		Using Abel's summation formula, we obtain
		\begin{equation*}
		\begin{aligned}
		\dfrac{\mathrm{d}}{\mathrm{d} t} E_{\Delta}(t) 
		=& -  \sum_{m = 1}^{M} \sum_{j} \left\{ \log \dfrac{\bar{c}_{m, j}}{\exp\{ - f_{m, j} \}} - \log \dfrac{\bar{c}_{m, j + 1}}{\exp\{ - f_{m, j + 1} \}} \right\} F_{m, j + \frac{1}{2}} \\
		=& \frac{1}{\Delta x} \sum_{m = 1}^{M} \sum_{j} \exp\{- f_{m, j + \frac{1}{2}}(t)\} \cdot \\
		&\left\{ \log \dfrac{\bar{c}_{m, j}(t)}{\exp\{ -f_{m, j}(t) \}} - \log \dfrac{\bar{c}_{m, j + 1}(t)}{\exp\{ -f_{m, j + 1}(t) \}} \right\}   \left\{ \frac{\bar{c}_{m, j + 1}(t)}{\exp\{-f_{m, j + 1}(t)\}} - \frac{\bar{c}_{m, j }(t)}{\exp\{-f_{m, j}(t)\}}\right\}  \\
		=&- \frac{1}{\Delta x} \sum_{m = 1}^{M} \sum_{j} \exp\{-f_{m, j + \frac{1}{2}}(t)\} \cdot \frac{1}{\beta_{m, j}(t)} \left\{ \frac{\bar{c}_{m, j + 1}(t)}{\exp\{-f_{m, j + 1}(t)\}} - \frac{\bar{c}_{m, j }(t)}{\exp\{-f_{m, j}(t)\}}\right\}^2 \\
		=& - D_{\Delta}(t) \leqslant 0, \quad \forall t \geqslant 0,
		\end{aligned}
		\end{equation*}
		where $\beta_{m, j}(t)$ is a point between $\frac{\bar{c}_{m, j }(t)}{\exp\{-f_{m, j}(t)\}}$ and $\frac{\bar{c}_{m, j + 1}(t)}{\exp\{-f_{m, j + 1}(t)\}}$ due to the mean-value theorem.
		
	\end{proof}

\bibliographystyle{plain}
\bibliography{arxiv_paper02.bib}
	
\end{document}